\documentclass[11pt,a4wide]{article}
\usepackage[latin,english]{babel}
\usepackage[utf8]{inputenc}
\usepackage[T1]{fontenc}
\usepackage{eufrak}
\usepackage[left=1.5cm, right=1.5cm, bottom=2.5cm]{geometry}
\usepackage{amsmath}
\usepackage{amssymb}
\usepackage{amsbsy}
\usepackage{amsthm}
\usepackage{makeidx}
\usepackage{mathtools}
\usepackage{mathabx, mathrsfs, dsfont}
\usepackage{cases}
\usepackage{braket}
\usepackage[toc,page]{appendix}
\usepackage{dirtytalk}
\usepackage{pdfsync}
\usepackage{graphicx}
\graphicspath{ {./Figures/} } 
\usepackage{epstopdf}
\usepackage{todonotes}
\usepackage{fancyhdr}
\usepackage{math}
\usepackage{cite}
\usepackage{color}
\usepackage{tikz}
\usepackage{subcaption}
\usepackage{bbm}
\usepackage{authblk}
\usepackage[colorlinks=true, pdfstartview=FitV, urlcolor=blue, citecolor=red, linkcolor=blue,pdfencoding=unicode,unicode=true,psdextra]{hyperref}
\usepackage{marginnote}

\usetikzlibrary{matrix,decorations.markings}
\usetikzlibrary{decorations.markings}
\usetikzlibrary{quotes,angles}

\usepackage{esint}

\DeclareMathOperator*{\res}{\textrm{res}}
\newcommand{\rhp}{A_N}
\def\pmtwo#1#2#3#4{\left( \begin{array}{cc}#1&#2\\#3&#4\end{array}\right)}

\newtheorem{RHP}[theorem]{RHP}

\numberwithin{equation}{section}

\title{The formation of a soliton gas condensate for the focusing Nonlinear  Schr\"odinger equation}
\author[1]{Aikaterini Gkogkou}
\author[1]{Guido Mazzuca}
\author[1]{Kenneth D. T-R McLaughlin}
\affil[1]{Department of Mathematics, Tulane University, New Orleans, LA}

\date{\today}

\begin{document}
\maketitle

\abstract{ In this work, we carry out a rigorous analysis of a multi-soliton solution of the focusing nonlinear Schr\"{o}dinger equation as the number, $N$, of solitons grows to infinity.  We discover configurations of $N$-soliton solutions which exhibit the formation (as $N \to \infty$) of a soliton gas condensate. Specifically, we show that when the eigenvalues of the Zakharov - Shabat operator for the NLS equation accumulate on two bounded horizontal segments in the complex plane with norming constants bounded away from $0$, then, asymptotically, the solution is described by a rapidly oscillatory elliptic-wave with constant velocity, on compact subsets of $(x,t)$. We then consider more complex solutions with an extra soliton component, and we show that, in this deterministic setting, the kinetic theory of solitons applies. This is to be distinguished from previous analyses of soliton gasses where the norming constants were tending to zero with $N$, and the asymptotic description only included elliptic waves in the long-time asymptotics.

\section{Introduction}

The focusing Nonlinear Schr\"odiner (NLS) equation is one of the fundamental examples of an integrable partial differential equation for a complex-valued unknown $\psi(x,t)$:
\begin{equation}\label{eq:NLS}
i \partial_t \psi  = -\frac{1}{2} \partial_x^2 \psi - | \psi |^2\psi\,, \quad (x,t) \in \R .
	\end{equation}
This equation is a canonical model for the deformation of a laser pulse traveling down an optical fiber transmission line \cite{Suret2011}, for one-dimensional wave propagation in deep water, and several other physical phenomena \cite{WhatIntegrability}. 

The NLS equation \eqref{eq:NLS} was discovered to be integrable by Zakharov and Shabat \cite{ZS71}.  The equation arises as the compatibility condition of the pair of two linear differential equations, the so-called Lax pair:
\begin{eqnarray}
\label{eq:ZS1}
&&
\partial_{x} \Phi = \begin{pmatrix}
-i z & \psi\\
 -\wo{\psi} & iz
\end{pmatrix} \Phi, \\
&&
 \partial_{t} \Phi = \begin{pmatrix}
-iz^2 + \frac{i}{2} |\psi|^2 & z \psi + \frac{i}{2} \psi_{x} \\ \\[0.1pt]
-z \wo{\psi} + \frac{i}{2} \wo{\psi}_{x} & i z^2 - \frac{i}{2} |\psi|^2
\end{pmatrix} \Phi
\end{eqnarray}
where  $z \in \mathbb{C}$ is a spectral parameter. In order for there to exist a simultaneous solution to both of these differential equations, $\psi(x,t)$ must solve the NLS equation.  But more importantly, as $\psi(x,t)$ evolves according to the NLS equation, the scattering data associated to the differential equation \eqref{eq:ZS1} evolve in an explicit way, which yields a solution procedure for the NLS equation:  from the initial data one computes the scattering data (to be defined momentarily), whose evolution in $t$ is explicit.  For any $t>0$, the solution $\psi(x,t)$ is determined via the inverse scattering transform.  It is then the analysis of the direct scattering transform, and the inverse transform at later times, which provides meaningful information about the behavior of solutions to the NLS equation.  

The NLS equation possesses a broad family of fundamental solutions.  First among these is the family of soliton solutions:
\begin{eqnarray}
\label{eq:1sol}
 \psi_{s} = 2 \eta \ \mbox{sech}\left(2 \eta ( x + 2 \xi t - x_{0})\right)e^{- 2i \left[ \xi x + \left( \xi^{2} - \eta^{2}\right)t\right]} e^{- i \phi_{0}}  ,
\end{eqnarray}
parametrized by four parameters, $\xi, \eta, \phi_{0}$, and $x_{0}$. Soliton solutions are exponentially localized traveling wave solutions, and represent canonical nonlinear phenomena which has been observed and exploited in many areas of application, including fluid dynamics, nonlinear optics, and other areas in the quantum realm \cite{Hasegawa2022,Hiyane2024,Kengne2021,suret2024soliton}.  But there are many more solutions in the catalog, including periodic and quasi-periodic solutions described using algebro-geometric function theory, dispersive shock waves, breathers, as well as the object of study in this paper:  multi-soliton solutions, which we will refer to as $N$-soliton solutions.  The positive integer $N$ refers to the number of fundamental solitons that are in some sense present in the solution.

An $N$-soliton solution typically exhibits the following asymptotic behavior as $|t| \to \infty$:  the solution decomposes into a collection of well-separated, exponentially localized traveling wave solutions, each of which asymptotically approaches a soliton solution as in \eqref{eq:1sol} with predetermined parameters.

The soliton, and $N$-soliton solutions, were part of the original discovery of integrable nonlinear partial differential equations, and there emerged very quickly a dual interpretation of the ``solitonic component'' of a solution as a collection of particles, at least in the $|t| \to \infty$ asymptotic regime.  Of course, for intermediate times, it is not easy to identify any individual solitons in a solution of the PDE, but nonetheless the particle interpretation persisted, with Zakharov \cite{Z71} proposing a kinetic theory of solitons in 1971.  The kinetic theory of solitons is a qualitative theory describing the evolution of a density function for solitons, supposing them to be individual particles which asymptotically interact in a pair-wise fashion.  The density function is coupled to a "velocity profile function" corresponding to the effective velocity of a single particle as it interacts with the gas.  In the last decade, there has emerged new experimental, numerical, and analytical justification for this qualitative theory (see \cite{suret2024soliton} for a review of these developments).  

In this paper we develop the analytical  description of $N$-soliton solutions in a new region of space-time which has been numerically observed (and conjectured) to correspond to a special type of soliton gas called a \textit{soliton condensate}.  The paper is organized as follows.  In Section \ref{sec:ProbSetUp} we present a very brief introduction to the kinetic theory of solitons, explain the characterization of $N$-soliton solutions in the soliton gas condensate scaling via Riemann-Hilbert problems, formulate a Riemann-Hilbert problem for the limiting soliton gas condensate solution, and describe our results regarding the behavior of $N$-soliton solutions of the NLS equation in the condensate scaling.  In Section \ref{sec:dressing}-\ref{sec:local} we carry out the asymptotic analysis of the Riemann-Hilbert problem for the soliton gas condensate.  In Section \ref{sec:solution} we present the proof of our main results, using the previously mentioned asymptotic analysis. In the brief Section \ref{sec:kinetic} we consider the interaction of a tracer soliton with the gas, and show that the kinetic theory of solitons holds for this soliton condensate.  Some technical results are deferred to the appendices.

\section{\texorpdfstring{$N$}{N}-soliton solutions, soliton gasses, and soliton gas condensates}
\label{sec:ProbSetUp}
\subsection{Direct and Inverse Scattering for the NLS equation}
The direct scattering transform is a mapping from a function $\psi(x)$ to the scattering data associated to the differential equation \eqref{eq:ZS1}.  We will give a very brief description here, and refer the reader to \cite{Ken_NLS} for a more detailed description.  One studies two sets of Jost solutions $\Phi^{\pm}(x,z)$ of this differential equation, which are uniquely determined by a normalization condition
\begin{equation}
\Phi^{\pm}(x,z)e^{ i z x \sigma_{3}} \to I \ \mbox{ as } x \to \pm \infty \,, \quad \sigma_3=\begin{pmatrix}
    1& 0 \\
    0&-1
\end{pmatrix}.
\end{equation}
The matrix-valued function $\Phi^{+}(x,z)$ is a fundamental solution to \eqref{eq:ZS1}, as is $\Phi^{-}(x,z)$.  This fact yields the scattering relation between these two solutions,
\begin{equation}
\Phi^{-}(x,z) = \Phi^{+}(x,z) \pmtwo{a(z)}{-\overline{b(z)}}{b(z)}{\overline{a(z)}}, \ \ \ \ \mbox{det}\pmtwo{a(z)}{-\overline{b(z)}}{b(z)}{\overline{a(z)}}= |a(z)|^{2} + |b(z)|^{2} = 1 \ , 
\end{equation}
which in turn yields the reflection coefficient $r(z) = b(z)/a(z)$ and the transmission coefficient $T(z) = \frac{1}{a(z)}$. 

It turns out that the Jost solutions have analyticity properties in the variable $z$, which follows from the usual existence theory.  Specifically, the first column of $\Phi^{-}(x,z)$ (and the second column of $\Phi^{+}(x,z)$) has an analytic extension to $\mathbb{C}_{+}$, the upper-half of the complex $z$  plane, where it exhibits \textit{exponential decay} for $x \to - \infty$ (and the second column of $\Phi^{+}(x,z)$ decays exponentially for $x \to + \infty$).  Similarly, the second column of $\Phi^{-}(x,z)$ and the first column of $\Phi^{+}(x,z)$ have analytic extensions to $\mathbb{C}_{-}$.  

The transmission coefficient also has a meromorphic extension to $\mathbb{C}_{+}$, and each pole of $T(z)$ in $\mathbb{C}_{+}$ corresponds to an $L^{2}$ eigenfunction for the differential equation \eqref{eq:ZS1}. For functions $\psi(x)$ in an open dense subset of $L^1(\R)$, there are at most a finite number of poles of $T(z)$ in $\mathbb{C}_{+}$.  At each pole $\lambda_{j}$, the first column of $\Phi^{-}(x,\lambda_{j})$ and the second column of $\Phi^{+}(x,\lambda_{j})$ are linearly dependent, and the proportionality constant $c_{j}$ relating them is referred to as a normalization constant, and is defined by
\begin{eqnarray}
\left(\Phi^{-}(x,\lambda_{j})\right)^{(1)} = c_{j} \left(\Phi^{+}(x,\lambda_{j})\right)^{(2)}.
\end{eqnarray}

The scattering data (also referred to as the \textit{spectral data}) corresponding to $\psi(x)$ consists of the reflection coefficient $r(k)$, the eigenvalues $\{\lambda_{j} \}_{j=1}^{N}$, and the normalization constants $\{c_{j}\}_{j=1}^{N}$.  Following \cite{Ken_NLS}, we let $\mathcal{D}$ denote the scattering data:
\begin{eqnarray}
\mathcal{D} = \left\{ r(z), \{\lambda_{j}, c_{j} \}_{j=1}^{N} \right\}.
\end{eqnarray}
As $\psi(x,t)$ evolves in time according to the NLS equation \eqref{eq:NLS}, the scattering data also evolves in time, but that evolution is quite simple since the direct scattering transform linearizes the flow:   the eigenvalues do not change in time, and the reflection coefficient and normalization constants evolve explicitly, so that
\begin{eqnarray}
\mathcal{D}(t) = \left\{ r(z,t), \{\lambda_{j}, c_{j}(t) \}_{j=1}^{N} \right\} = \left\{ r(z)e^{2 i t z^{2} }, \{\lambda_{j}, c_{j}e^{ 2 i t \lambda_{j}^{2}  } \}_{j=1}^{N} \right\}.
\end{eqnarray}

We note in passing that the direct scattering transform as discussed above is defined for \textit{generic} initial data.  However, for non-generic data certain spectral singularities exist which require additional information and analysis.  As we are interested in \textit{reflectionless potentials} for which $r(z) \equiv 0$, spectral singularities are not an issue.

The solution $\psi(x,t)$ is uniquely determined by the scattering data, and the mapping from $\mathcal{D}(t)$ to $\psi(x,t)$ is referred to as the inverse scattering transform.  The inverse scattering transform can be characterized in terms of a system of integral equations, referred to as Gelfand-Levitan-Marcenko equations.  Equivalently, it can be characterized in terms of a Riemann-Hilbert problem, which is the approach we take here.  Riemann-Hilbert formulations of the inverse problem have over the last 30 years yielded a remarkable number of different asymptotic results concerning the behavior of solutions of integrable nonlinear partial differential equations, starting with the work of Deift and Zhou \cite{Deift1993} for the modified KdV equation.

For the NLS equation, the Riemann-Hilbert problem is to determine a $2 \times 2$ matrix $M = M(z) = M(z;x,t)$.  We will usually suppress the dependence on $x$ and $t$, and simply write $M$ or $M(z)$.
\begin{RHP}
\label{rhp:reflection_RHP}
    Given spectral data $\mathcal{D}(t)$, find a $2 \times 2$ matrix-valued function $M(z)$ such that
\begin{itemize}
	\item $\ {M}(z)$ is meromorphic in  $ \C \setminus \R $, with simple poles at the points $\{ \lambda_{j}, \overline{\lambda_{j}}\}_{j=1}^{N}$.
    \item For $z \in \mathbb{R}$, $M$ has boundary values $M_{+}(z)$ and $M_{-}(z)$, which satisfy the jump relation 
        \begin{equation}
        \label{eq:MJump}
           M_+(z)=\  M_-(z) \  J_M(z), \quad  J_M = \begin{pmatrix}
                1+|r(z)|^2 & \wo{r(z)}e^{-2i\theta(z;x,t)} \\
                r(z)e^{2i\theta(z;x,t)} & 1
            \end{pmatrix},
        \end{equation}
where $\theta(z;x,t) = zx + z^2t$.
	\item At each of the poles $\lambda_{j}$ in the upper half-plane, $\ {M}(z)$ satisfies the residue condition 
	\begin{equation}
    \label{eq:rescondCP}
		\res_{\lambda_j} M(z) = \lim_{z \to \lambda_j } \left[ M(z)\begin{pmatrix}
			0 & 0 \\
			c_j e^{2i\theta(\lambda_j;x,t)} & 0
		\end{pmatrix}\right]\,,
    \end{equation}
and at each pole $\overline{\lambda_{j}}$ in the lower half-plane:
\begin{equation}
    \label{eq:rescondCM}
\res_{\wo {\lambda_j}} M(z) = \lim_{z \to \wo \lambda_j } \left[ M(z) \begin{pmatrix}
			0 & -\wo c_j e^{-2i\theta(\wo \lambda_j;x,t)} \\
			0 & 0
		\end{pmatrix}\right].
	\end{equation}
\item $\ {M}(z) = \  I + O(z^{-1})$ as $z \to \infty$.
\end{itemize}
\end{RHP}
Existence and uniqueness for this Riemann-Hilbert problem follows from what are by now standard arguments:  uniqueness follows from Liouville's theorem, and existence follows from a careful application of a vanishing lemma as described in \cite{Zhou89}.  The solution $\psi(x,t)$ to the NLS equation corresponding to the initial data $\psi(x)$, assuming that $\psi,\psi_x\in L^1(\R)$, and determined by the scattering data $\mathcal{D}(t)$ is obtained from $M(z;x,t)$ through the $z \to \infty$ normalization:
\begin{eqnarray}
M(z;x,t) = I + \frac{1}{2 i z} \pmtwo{-\int_{x}^{\infty} |\psi(s,t)|^{2}\di s}{\psi(x,t)}{\overline{\psi(x,t)}}{\int_{x}^{\infty} |\psi(s,t)|^{2}\di s}
+ \mathcal{O} \left(\frac{1}{z^{2}} \right).
\end{eqnarray}
We mention that the RHP approach to solve non-linear integrable PDEs is not just a relevant analytical tool, but it is also numerical one. Indeed, several authors \cite{Bilman2023,Bilman2020,gkogkou2024,Trogdon2012}  employed this method to numerically analyze various non-linear integrable PDEs.

\noindent{\bf Notation}:  We will use $\mathcal{D}$ to denote the scattering data at $t=0$ (i.e. $\mathcal{D}= \mathcal{D}(0)$).  In order to emphasize the dependence on the scattering data, we will occasionally write
\begin{eqnarray}
\psi(x,t;\mathcal{D})
\end{eqnarray}
to represent the solution to the NLS equation determined by the initial scattering data $\mathcal{D}$.  

\subsection{\texorpdfstring{$N$}{N}-soliton solutions and soliton gasses}
\label{sec:introsolgas}
In this paper, we are interested in $N$-soliton solutions, for which $r(z) \equiv 0$.  Under this assumption, the above Riemann-Hilbert problem simplifies.  Indeed, the jump relation becomes trivial:  $M_{+}(z) = M_{-}(z)$ for $z \in \mathbb{R}$, and this implies that now $M$ is a meromorphic function of $z$.  The Riemann-Hilbert problem is referred to as a meromorphic Riemann-Hilbert problem, for which asymptotic analysis started in \cite{Kamvissis2003} where the authors studied a semi-classical scaling of the NLS equation.  (See also  \cite{BaikKriechMcLMiller}, where the authors studied the asymptotic behavior of discrete orthogonal polynomials, through the analysis of meromorphic Riemann-Hilbert problems).

The original set of kinetic equations in \cite{Z71} were derived under the loose assumption of a dilute gas of solitons, described by an evolving density function $f(z,x,t)$.  This density function yields (in some sense) the number of solitons in a small volume element around $(z,x,t)$.  Of course the notion of a soliton in a small space-time window, let alone the idea that there might be a large number of solitons with slightly different eigenvalue parameters in that same space-time window, while at the same time supposing there is enough space between solitons to consider them to be dilute, was quite challenging from the point of view of analysis.

One experimental prediction in \cite{Z71} is that a ``tracer soliton'', a soliton whose amplitude is much bigger than all the other solitons, launched into this gas, should interact with the gas in such away that its velocity \textit{deforms} as it propagates, and the explicit computation of this time-dependent velocity, as a function of density of the gas as well as the tracer's initial eigenvalue parameter, yielded an experimental prediction which could in principle be tested.  About 30 years later, El \cite{El2003} extended the kinetic theory to a dense gas, in which a coupled system of equations was derived, with the evolving density coupled to the evolving velocity of a tracer.  Both \cite{Z71} and \cite{El2003} focused on the KdV equation, but this was extended by El and Kamchatnov in \cite{El05} to other integrable nonlinear partial differential equations, including the NLS equation.  Significant progress regarding the understanding of the kinetic equations arising in different settings, as well as more detailed analysis of the kinetic equations themselves, has continued, for example in \cite{El2020,Kuijlaars2021,Tovbis2022,Tovbis2024}.

However, only recently has there been any validation of the kinetic theory, in the form of robust numerical \cite{carbone2016macroscopic,congy2021soliton,roberti2021numerical} and experimental \cite{maiden2016observation,mossman2024observation} evidence (see the review manuscript \cite{suret2024soliton}).  There have also been a few limited analytical results, particularly \cite{Girotti2021,Girotti2023,girotti2024law,jenkins2024approximation}.

In \cite{Girotti2021} the authors considered the $N\to\infty$ limit of an $N$-soliton solution for the KdV equation, with normalization constants uniformly of size $\mathcal{O}\left(N^{-1} \right)$.  The $N \to \infty$ analysis was carried out under the assumption that $x$ and $t$ were bounded.  The fact that the normalization constants were \textit{vanishingly small as $N \to \infty$} provided a necessary foothold for the rigorous analysis.  This can be understood (intuitively) by considering the connection, in the case of a single soliton, between the normalization constant and the initial location $x_{0}$ of the soliton's peak:  
\begin{eqnarray}
x_{0} = \frac{1}{2 \eta}\log{ \left| \frac{c}{2 \eta} \right| }
\end{eqnarray}
(here $c$ is the normalization constant, and $\lambda = \xi + i \eta$ is the eigenvalue), so if $c \approx N^{-1}$, then $x_{0}$ is shifted a distance $\log{N}$ to the left.  One could then imagine that if one has an $N$-soliton solution with all the normalization constants of size $\mathcal{O}\left(N^{-1}\right)$, all the solitons are shifted far to the left, and what remains for finite values of $x$ is the accumulated tails of all these solitons.  The results in \cite{Girotti2021} are consistent with this intuition.  They prove convergence of the $N$-soliton solutions (as $N \to \infty$) to a ``soliton gas solution'' of the KdV equation, but only for $x$ and $t$ in compact sets.  They subsequently carry out an asymptotic analysis of the soliton gas solution for $|x| \to \infty$, and also for $t \to \infty$ with $\frac{|x|}{t}$ bounded.  In the large-time analysis of the soliton gas solution, they established the existence of an evolving density function, depending on the spectral parameter, space, and time, which describes the local, evolving solitonic content of the solution, providing a first analytical corroboration of the kinetic theory, and in the process showing that in the long-time regime, what emerges is a type of dense soliton gas, called a \textit{condensate}.  However, in \cite{Girotti2021} they did not consider the behavior of the original $N$-soliton solution on space-time scales that would overlap with their asymptotic analysis of the soliton gas solution.  

More recently, in \cite{Orsatti23,bertola2024,falqui2024} the authors considered the soliton gas and the breather gas for the NLS equation under the assumptions that the eigenvalues $\lambda_j$ accumulate on a domain $D\subset \C_+$ (and its complex conjugate $\wo D$) and norming constants uniformly $\mathcal{O}(N^{-1})$.

In \cite{Girotti2023}, the authors considered the mKdV equation and studied the interaction of a tracer soliton with a similar soliton gas with vanishingly small normalization constants.  In \cite{girotti2024law} the authors developed a probabilistic analysis of a soliton gas for the NLS equation, again under the assumption that the normalization constants are uniformly $\mathcal{O}\left( N^{-1} \right)$.  In other words, the ``bulk'' asymptotic behavior of $N$-soliton solutions, with normalization constants chosen to be bounded (away from both $0$ and $\infty$), has not been considered in the context of the kinetic theory of solitons.  (However, the works \cite{Kamvissis2003, JenkinsMcLNLS} concerning the semi-classically scaled NLS equation, could be thought of as examples of this scaling.)

\begin{remark}
Similarly, if $c$ is \textit{large}, then $x_{0}$ is shifted far to the right, and so the above intuitive suggests that if all the normalization constants are $\mathcal{O}\left( N \right)$, then again for $x$ and $t$ on compact sets, one only sees the accumulated tails of all the solitons.
\end{remark}

\subsection{\texorpdfstring{$N$}{N}-soliton solutions: condensate scaling}

The present manuscript considers the asymptotic behavior of $N$-soliton solutions of the NLS equation, with eigenvalues accumulating on a horizontal line in $\mathbb{C}_{+}$ (and its complex conjugate contour in $\mathbb{C}_{-}$), with normalization constants that are bounded, and bounded away from $0$.  Assuming the limit $N \to \infty$, such a configuration has been referred to (see, for example, \cite{suret2024soliton}) as a \textit{soliton condensate}, accompanied by the conjecture that soliton condensates achieve maximal soliton density.  

We will refer to the sequence of such $N$-soliton solutions as an $N$-soliton condensate, and we will refer to the associated sequence of scattering data as an \textbf{$N$-soliton condensate scattering data}. Specifically, we consider the following assumptions.
\begin{Assumption}[$N$-soliton condensate scattering data] 
\label{ass:sym}
For a positive integer $N$, we generate the scattering data as follows:
\begin{enumerate}
\item  There are two positive constants $a$ and $b$, which define a line segment $ \eta_{1} = [- a + i b, a + i b]$, which we orient from right to left,  and its complex conjugate $\eta_{2} = [-a-ib, a - ib]$, which we orient from left to right.  We define
\begin{eqnarray}
A = a + i b
\end{eqnarray}
to be the right-most endpoint of $\eta_{1}$.
\item There is a density function $\rho(z)$ which is analytic in a neighborhood $\cU$ of the domain enclosed by the line segment $\eta_{1}$ and the arc of the circle of radius $|A|$ centered at $0$ connecting $A$ with $\overline{A}$ ($\rho$ is also assumed to be analytic in the Schwarz reflection of $\cU$).  The density $\rho$ satisfies
    \begin{eqnarray}
    \label{eq:rhoConds1}
    &&
\rho(z)>0 \mbox{ for } z = t + ib, \ t\in(-a,a), \\
&&
\int_{-a}^a \rho(t+ib)\di t = 1, \\
\label{eq:rhoConds4}
&&\rho(\wo z) = \wo{\rho(z)}.
\end{eqnarray}

\item The eigenvalues $\{\lambda_j\}_{j=1}^{N}$ on $\eta_{1}$ are chosen to satisfy
\begin{eqnarray}
\label{eq:lamsamp}
\int_{-a}^{\Re(\lambda_j)} \rho(x + ib)\di x = \frac{2j-1}{2N}, \quad \Im(\lambda_j) = b, \ j = 1, \ldots, N.
\end{eqnarray}
See Figure \ref{fig:poles} for an example with uniform density.
\item There is a ``normalization constant sampling function'' $h(z)$ which is analytic in $\cU$ and $\overline{\cU}$ satisfying $ h(\wo{z}) = \wo{h(z)}$, and $h(z)\neq0$ for $z\in \eta_1\cup\eta_2$.

\item For each positive integer $N$, the normalization constants $\{c_{j} \}_{j=1}^{N}$ are chosen to satisfy \begin{eqnarray}
\label{eq:normsamp}
         c_j= h(\lambda_j), \ \ j = 1, \ldots, N \ ,
\end{eqnarray}

\item There are two contours, $\Gamma_{1} \in \mathbb{C}_{+}$ and $\Gamma_{2} \in \mathbb{C}_{-}$, which are Schwarz reflections of each other (and so they are determined by $\Gamma_{1}$).  The contour $\Gamma_{1}$ is a simple, closed, analytic curve that encircles the interval $\eta_{1}$.    See Figure \ref{fig:contour}.
\end{enumerate}

\end{Assumption}

\begin{figure}
		\centering
	\tikzset{every picture/.style={line width=0.75pt}}         

\begin{tikzpicture}[x=0.5pt,y=0.5pt,yscale=-1,xscale=1]

\draw    (100,119.33) -- (555,120.33) ;

\draw  [color={rgb, 255:red, 208; green, 2; blue, 27 }  ,draw opacity=1 ][fill={rgb, 255:red, 208; green, 2; blue, 27 }  ,fill opacity=1 ] (111.07,114.09) .. controls (113.96,114.13) and (116.28,116.5) .. (116.24,119.39) .. controls (116.2,122.28) and (113.82,124.6) .. (110.93,124.56) .. controls (108.04,124.52) and (105.73,122.14) .. (105.77,119.25) .. controls (105.81,116.36) and (108.18,114.05) .. (111.07,114.09) -- cycle ;

\draw  [color={rgb, 255:red, 208; green, 2; blue, 27 }  ,draw opacity=1 ][fill={rgb, 255:red, 208; green, 2; blue, 27 }  ,fill opacity=1 ] (142.27,114.49) .. controls (145.16,114.53) and (147.48,116.9) .. (147.44,119.79) .. controls (147.4,122.68) and (145.02,125) .. (142.13,124.96) .. controls (139.24,124.92) and (136.93,122.54) .. (136.97,119.65) .. controls (137.01,116.76) and (139.38,114.45) .. (142.27,114.49) -- cycle ;

\draw  [color={rgb, 255:red, 208; green, 2; blue, 27 }  ,draw opacity=1 ][fill={rgb, 255:red, 208; green, 2; blue, 27 }  ,fill opacity=1 ] (176.67,114.89) .. controls (179.56,114.93) and (181.88,117.3) .. (181.84,120.19) .. controls (181.8,123.08) and (179.42,125.4) .. (176.53,125.36) .. controls (173.64,125.32) and (171.33,122.94) .. (171.37,120.05) .. controls (171.41,117.16) and (173.78,114.85) .. (176.67,114.89) -- cycle ;

\draw  [color={rgb, 255:red, 208; green, 2; blue, 27 }  ,draw opacity=1 ][fill={rgb, 255:red, 208; green, 2; blue, 27 }  ,fill opacity=1 ] (208.67,114.49) .. controls (211.56,114.53) and (213.88,116.9) .. (213.84,119.79) .. controls (213.8,122.68) and (211.42,125) .. (208.53,124.96) .. controls (205.64,124.92) and (203.33,122.54) .. (203.37,119.65) .. controls (203.41,116.76) and (205.78,114.45) .. (208.67,114.49) -- cycle ;

\draw  [color={rgb, 255:red, 208; green, 2; blue, 27 }  ,draw opacity=1 ][fill={rgb, 255:red, 208; green, 2; blue, 27 }  ,fill opacity=1 ] (241.87,114.89) .. controls (244.76,114.93) and (247.08,117.3) .. (247.04,120.19) .. controls (247,123.08) and (244.62,125.4) .. (241.73,125.36) .. controls (238.84,125.32) and (236.53,122.94) .. (236.57,120.05) .. controls (236.61,117.16) and (238.98,114.85) .. (241.87,114.89) -- cycle ;

\draw  [color={rgb, 255:red, 208; green, 2; blue, 27 }  ,draw opacity=1 ][fill={rgb, 255:red, 208; green, 2; blue, 27 }  ,fill opacity=1 ] (279.87,114.49) .. controls (282.76,114.53) and (285.08,116.9) .. (285.04,119.79) .. controls (285,122.68) and (282.62,125) .. (279.73,124.96) .. controls (276.84,124.92) and (274.53,122.54) .. (274.57,119.65) .. controls (274.61,116.76) and (276.98,114.45) .. (279.87,114.49) -- cycle ;

\draw  [color={rgb, 255:red, 208; green, 2; blue, 27 }  ,draw opacity=1 ][fill={rgb, 255:red, 208; green, 2; blue, 27 }  ,fill opacity=1 ] (314.67,114.49) .. controls (317.56,114.53) and (319.88,116.9) .. (319.84,119.79) .. controls (319.8,122.68) and (317.42,125) .. (314.53,124.96) .. controls (311.64,124.92) and (309.33,122.54) .. (309.37,119.65) .. controls (309.41,116.76) and (311.78,114.45) .. (314.67,114.49) -- cycle ;

\draw  [color={rgb, 255:red, 208; green, 2; blue, 27 }  ,draw opacity=1 ][fill={rgb, 255:red, 208; green, 2; blue, 27 }  ,fill opacity=1 ] (348.67,114.49) .. controls (351.56,114.53) and (353.88,116.9) .. (353.84,119.79) .. controls (353.8,122.68) and (351.42,125) .. (348.53,124.96) .. controls (345.64,124.92) and (343.33,122.54) .. (343.37,119.65) .. controls (343.41,116.76) and (345.78,114.45) .. (348.67,114.49) -- cycle ;

\draw  [color={rgb, 255:red, 208; green, 2; blue, 27 }  ,draw opacity=1 ][fill={rgb, 255:red, 208; green, 2; blue, 27 }  ,fill opacity=1 ] (380.67,115.29) .. controls (383.56,115.33) and (385.88,117.7) .. (385.84,120.59) .. controls (385.8,123.48) and (383.42,125.8) .. (380.53,125.76) .. controls (377.64,125.72) and (375.33,123.34) .. (375.37,120.45) .. controls (375.41,117.56) and (377.78,115.25) .. (380.67,115.29) -- cycle ;

\draw  [color={rgb, 255:red, 208; green, 2; blue, 27 }  ,draw opacity=1 ][fill={rgb, 255:red, 208; green, 2; blue, 27 }  ,fill opacity=1 ] (409.87,115.29) .. controls (412.76,115.33) and (415.08,117.7) .. (415.04,120.59) .. controls (415,123.48) and (412.62,125.8) .. (409.73,125.76) .. controls (406.84,125.72) and (404.53,123.34) .. (404.57,120.45) .. controls (404.61,117.56) and (406.98,115.25) .. (409.87,115.29) -- cycle ;
 
\draw  [color={rgb, 255:red, 208; green, 2; blue, 27 }  ,draw opacity=1 ][fill={rgb, 255:red, 208; green, 2; blue, 27 }  ,fill opacity=1 ] (439.87,114.89) .. controls (442.76,114.93) and (445.08,117.3) .. (445.04,120.19) .. controls (445,123.08) and (442.62,125.4) .. (439.73,125.36) .. controls (436.84,125.32) and (434.53,122.94) .. (434.57,120.05) .. controls (434.61,117.16) and (436.98,114.85) .. (439.87,114.89) -- cycle ;

\draw  [color={rgb, 255:red, 208; green, 2; blue, 27 }  ,draw opacity=1 ][fill={rgb, 255:red, 208; green, 2; blue, 27 }  ,fill opacity=1 ] (471.87,114.49) .. controls (474.76,114.53) and (477.08,116.9) .. (477.04,119.79) .. controls (477,122.68) and (474.62,125) .. (471.73,124.96) .. controls (468.84,124.92) and (466.53,122.54) .. (466.57,119.65) .. controls (466.61,116.76) and (468.98,114.45) .. (471.87,114.49) -- cycle ;

\draw  [color={rgb, 255:red, 208; green, 2; blue, 27 }  ,draw opacity=1 ][fill={rgb, 255:red, 208; green, 2; blue, 27 }  ,fill opacity=1 ] (501.47,114.89) .. controls (504.36,114.93) and (506.68,117.3) .. (506.64,120.19) .. controls (506.6,123.08) and (504.22,125.4) .. (501.33,125.36) .. controls (498.44,125.32) and (496.13,122.94) .. (496.17,120.05) .. controls (496.21,117.16) and (498.58,114.85) .. (501.47,114.89) -- cycle ;

\draw  [color={rgb, 255:red, 208; green, 2; blue, 27 }  ,draw opacity=1 ][fill={rgb, 255:red, 208; green, 2; blue, 27 }  ,fill opacity=1 ] (529.47,114.89) .. controls (532.36,114.93) and (534.68,117.3) .. (534.64,120.19) .. controls (534.6,123.08) and (532.22,125.4) .. (529.33,125.36) .. controls (526.44,125.32) and (524.13,122.94) .. (524.17,120.05) .. controls (524.21,117.16) and (526.58,114.85) .. (529.47,114.89) -- cycle ;

\draw    (100.4,209.73) -- (555.4,210.73) ;

\draw  [color={rgb, 255:red, 208; green, 2; blue, 27 }  ,draw opacity=1 ][fill={rgb, 255:red, 208; green, 2; blue, 27 }  ,fill opacity=1 ] (111.47,204.49) .. controls (114.36,204.53) and (116.68,206.9) .. (116.64,209.79) .. controls (116.6,212.68) and (114.22,215) .. (111.33,214.96) .. controls (108.44,214.92) and (106.13,212.54) .. (106.17,209.65) .. controls (106.21,206.76) and (108.58,204.45) .. (111.47,204.49) -- cycle ;

\draw  [color={rgb, 255:red, 208; green, 2; blue, 27 }  ,draw opacity=1 ][fill={rgb, 255:red, 208; green, 2; blue, 27 }  ,fill opacity=1 ] (142.67,204.89) .. controls (145.56,204.93) and (147.88,207.3) .. (147.84,210.19) .. controls (147.8,213.08) and (145.42,215.4) .. (142.53,215.36) .. controls (139.64,215.32) and (137.33,212.94) .. (137.37,210.05) .. controls (137.41,207.16) and (139.78,204.85) .. (142.67,204.89) -- cycle ;

\draw  [color={rgb, 255:red, 208; green, 2; blue, 27 }  ,draw opacity=1 ][fill={rgb, 255:red, 208; green, 2; blue, 27 }  ,fill opacity=1 ] (177.07,205.29) .. controls (179.96,205.33) and (182.28,207.7) .. (182.24,210.59) .. controls (182.2,213.48) and (179.82,215.8) .. (176.93,215.76) .. controls (174.04,215.72) and (171.73,213.34) .. (171.77,210.45) .. controls (171.81,207.56) and (174.18,205.25) .. (177.07,205.29) -- cycle ;

\draw  [color={rgb, 255:red, 208; green, 2; blue, 27 }  ,draw opacity=1 ][fill={rgb, 255:red, 208; green, 2; blue, 27 }  ,fill opacity=1 ] (209.07,204.89) .. controls (211.96,204.93) and (214.28,207.3) .. (214.24,210.19) .. controls (214.2,213.08) and (211.82,215.4) .. (208.93,215.36) .. controls (206.04,215.32) and (203.73,212.94) .. (203.77,210.05) .. controls (203.81,207.16) and (206.18,204.85) .. (209.07,204.89) -- cycle ;

\draw  [color={rgb, 255:red, 208; green, 2; blue, 27 }  ,draw opacity=1 ][fill={rgb, 255:red, 208; green, 2; blue, 27 }  ,fill opacity=1 ] (242.27,205.29) .. controls (245.16,205.33) and (247.48,207.7) .. (247.44,210.59) .. controls (247.4,213.48) and (245.02,215.8) .. (242.13,215.76) .. controls (239.24,215.72) and (236.93,213.34) .. (236.97,210.45) .. controls (237.01,207.56) and (239.38,205.25) .. (242.27,205.29) -- cycle ;
 
\draw  [color={rgb, 255:red, 208; green, 2; blue, 27 }  ,draw opacity=1 ][fill={rgb, 255:red, 208; green, 2; blue, 27 }  ,fill opacity=1 ] (280.27,204.89) .. controls (283.16,204.93) and (285.48,207.3) .. (285.44,210.19) .. controls (285.4,213.08) and (283.02,215.4) .. (280.13,215.36) .. controls (277.24,215.32) and (274.93,212.94) .. (274.97,210.05) .. controls (275.01,207.16) and (277.38,204.85) .. (280.27,204.89) -- cycle ;

\draw  [color={rgb, 255:red, 208; green, 2; blue, 27 }  ,draw opacity=1 ][fill={rgb, 255:red, 208; green, 2; blue, 27 }  ,fill opacity=1 ] (315.07,204.89) .. controls (317.96,204.93) and (320.28,207.3) .. (320.24,210.19) .. controls (320.2,213.08) and (317.82,215.4) .. (314.93,215.36) .. controls (312.04,215.32) and (309.73,212.94) .. (309.77,210.05) .. controls (309.81,207.16) and (312.18,204.85) .. (315.07,204.89) -- cycle ;

\draw  [color={rgb, 255:red, 208; green, 2; blue, 27 }  ,draw opacity=1 ][fill={rgb, 255:red, 208; green, 2; blue, 27 }  ,fill opacity=1 ] (349.07,204.89) .. controls (351.96,204.93) and (354.28,207.3) .. (354.24,210.19) .. controls (354.2,213.08) and (351.82,215.4) .. (348.93,215.36) .. controls (346.04,215.32) and (343.73,212.94) .. (343.77,210.05) .. controls (343.81,207.16) and (346.18,204.85) .. (349.07,204.89) -- cycle ;

\draw  [color={rgb, 255:red, 208; green, 2; blue, 27 }  ,draw opacity=1 ][fill={rgb, 255:red, 208; green, 2; blue, 27 }  ,fill opacity=1 ] (381.07,205.69) .. controls (383.96,205.73) and (386.28,208.1) .. (386.24,210.99) .. controls (386.2,213.88) and (383.82,216.2) .. (380.93,216.16) .. controls (378.04,216.12) and (375.73,213.74) .. (375.77,210.85) .. controls (375.81,207.96) and (378.18,205.65) .. (381.07,205.69) -- cycle ;

\draw  [color={rgb, 255:red, 208; green, 2; blue, 27 }  ,draw opacity=1 ][fill={rgb, 255:red, 208; green, 2; blue, 27 }  ,fill opacity=1 ] (410.27,205.69) .. controls (413.16,205.73) and (415.48,208.1) .. (415.44,210.99) .. controls (415.4,213.88) and (413.02,216.2) .. (410.13,216.16) .. controls (407.24,216.12) and (404.93,213.74) .. (404.97,210.85) .. controls (405.01,207.96) and (407.38,205.65) .. (410.27,205.69) -- cycle ;

\draw  [color={rgb, 255:red, 208; green, 2; blue, 27 }  ,draw opacity=1 ][fill={rgb, 255:red, 208; green, 2; blue, 27 }  ,fill opacity=1 ] (440.27,205.29) .. controls (443.16,205.33) and (445.48,207.7) .. (445.44,210.59) .. controls (445.4,213.48) and (443.02,215.8) .. (440.13,215.76) .. controls (437.24,215.72) and (434.93,213.34) .. (434.97,210.45) .. controls (435.01,207.56) and (437.38,205.25) .. (440.27,205.29) -- cycle ;

\draw  [color={rgb, 255:red, 208; green, 2; blue, 27 }  ,draw opacity=1 ][fill={rgb, 255:red, 208; green, 2; blue, 27 }  ,fill opacity=1 ] (472.27,204.89) .. controls (475.16,204.93) and (477.48,207.3) .. (477.44,210.19) .. controls (477.4,213.08) and (475.02,215.4) .. (472.13,215.36) .. controls (469.24,215.32) and (466.93,212.94) .. (466.97,210.05) .. controls (467.01,207.16) and (469.38,204.85) .. (472.27,204.89) -- cycle ;

\draw  [color={rgb, 255:red, 208; green, 2; blue, 27 }  ,draw opacity=1 ][fill={rgb, 255:red, 208; green, 2; blue, 27 }  ,fill opacity=1 ] (501.87,205.29) .. controls (504.76,205.33) and (507.08,207.7) .. (507.04,210.59) .. controls (507,213.48) and (504.62,215.8) .. (501.73,215.76) .. controls (498.84,215.72) and (496.53,213.34) .. (496.57,210.45) .. controls (496.61,207.56) and (498.98,205.25) .. (501.87,205.29) -- cycle ;

\draw  [color={rgb, 255:red, 208; green, 2; blue, 27 }  ,draw opacity=1 ][fill={rgb, 255:red, 208; green, 2; blue, 27 }  ,fill opacity=1 ] (529.87,205.29) .. controls (532.76,205.33) and (535.08,207.7) .. (535.04,210.59) .. controls (535,213.48) and (532.62,215.8) .. (529.73,215.76) .. controls (526.84,215.72) and (524.53,213.34) .. (524.57,210.45) .. controls (524.61,207.56) and (526.98,205.25) .. (529.87,205.29) -- cycle ;

\draw    (70.2,160.6) -- (583.4,161) ;

\draw (110.13,158.99) node [anchor=south east] [inner sep=0.75pt]    {$\mathbb{R}$};

\draw (348.67,111.09) node [anchor=south] [inner sep=0.75pt]    {$\lambda _{j}$};
\end{tikzpicture}

		\caption{Poles $\lambda_j$ and their conjugates}
		\label{fig:poles}
	\end{figure}

	\begin{figure}
		\centering

\tikzset{every picture/.style={line width=0.75pt}}          

\begin{tikzpicture}[x=0.5pt,y=0.5pt,yscale=-1,xscale=1]

\draw    (100,119.33) -- (555,120.33) ;

\draw  [color={rgb, 255:red, 208; green, 2; blue, 27 }  ,draw opacity=1 ][fill={rgb, 255:red, 208; green, 2; blue, 27 }  ,fill opacity=1 ] (111.07,114.09) .. controls (113.96,114.13) and (116.28,116.5) .. (116.24,119.39) .. controls (116.2,122.28) and (113.82,124.6) .. (110.93,124.56) .. controls (108.04,124.52) and (105.73,122.14) .. (105.77,119.25) .. controls (105.81,116.36) and (108.18,114.05) .. (111.07,114.09) -- cycle ;

\draw  [color={rgb, 255:red, 208; green, 2; blue, 27 }  ,draw opacity=1 ][fill={rgb, 255:red, 208; green, 2; blue, 27 }  ,fill opacity=1 ] (142.27,114.49) .. controls (145.16,114.53) and (147.48,116.9) .. (147.44,119.79) .. controls (147.4,122.68) and (145.02,125) .. (142.13,124.96) .. controls (139.24,124.92) and (136.93,122.54) .. (136.97,119.65) .. controls (137.01,116.76) and (139.38,114.45) .. (142.27,114.49) -- cycle ;

\draw  [color={rgb, 255:red, 208; green, 2; blue, 27 }  ,draw opacity=1 ][fill={rgb, 255:red, 208; green, 2; blue, 27 }  ,fill opacity=1 ] (176.67,114.89) .. controls (179.56,114.93) and (181.88,117.3) .. (181.84,120.19) .. controls (181.8,123.08) and (179.42,125.4) .. (176.53,125.36) .. controls (173.64,125.32) and (171.33,122.94) .. (171.37,120.05) .. controls (171.41,117.16) and (173.78,114.85) .. (176.67,114.89) -- cycle ;

\draw  [color={rgb, 255:red, 208; green, 2; blue, 27 }  ,draw opacity=1 ][fill={rgb, 255:red, 208; green, 2; blue, 27 }  ,fill opacity=1 ] (208.67,114.49) .. controls (211.56,114.53) and (213.88,116.9) .. (213.84,119.79) .. controls (213.8,122.68) and (211.42,125) .. (208.53,124.96) .. controls (205.64,124.92) and (203.33,122.54) .. (203.37,119.65) .. controls (203.41,116.76) and (205.78,114.45) .. (208.67,114.49) -- cycle ;

\draw  [color={rgb, 255:red, 208; green, 2; blue, 27 }  ,draw opacity=1 ][fill={rgb, 255:red, 208; green, 2; blue, 27 }  ,fill opacity=1 ] (241.87,114.89) .. controls (244.76,114.93) and (247.08,117.3) .. (247.04,120.19) .. controls (247,123.08) and (244.62,125.4) .. (241.73,125.36) .. controls (238.84,125.32) and (236.53,122.94) .. (236.57,120.05) .. controls (236.61,117.16) and (238.98,114.85) .. (241.87,114.89) -- cycle ;

\draw  [color={rgb, 255:red, 208; green, 2; blue, 27 }  ,draw opacity=1 ][fill={rgb, 255:red, 208; green, 2; blue, 27 }  ,fill opacity=1 ] (279.87,114.49) .. controls (282.76,114.53) and (285.08,116.9) .. (285.04,119.79) .. controls (285,122.68) and (282.62,125) .. (279.73,124.96) .. controls (276.84,124.92) and (274.53,122.54) .. (274.57,119.65) .. controls (274.61,116.76) and (276.98,114.45) .. (279.87,114.49) -- cycle ;

\draw  [color={rgb, 255:red, 208; green, 2; blue, 27 }  ,draw opacity=1 ][fill={rgb, 255:red, 208; green, 2; blue, 27 }  ,fill opacity=1 ] (314.67,114.49) .. controls (317.56,114.53) and (319.88,116.9) .. (319.84,119.79) .. controls (319.8,122.68) and (317.42,125) .. (314.53,124.96) .. controls (311.64,124.92) and (309.33,122.54) .. (309.37,119.65) .. controls (309.41,116.76) and (311.78,114.45) .. (314.67,114.49) -- cycle ;

\draw  [color={rgb, 255:red, 208; green, 2; blue, 27 }  ,draw opacity=1 ][fill={rgb, 255:red, 208; green, 2; blue, 27 }  ,fill opacity=1 ] (348.67,114.49) .. controls (351.56,114.53) and (353.88,116.9) .. (353.84,119.79) .. controls (353.8,122.68) and (351.42,125) .. (348.53,124.96) .. controls (345.64,124.92) and (343.33,122.54) .. (343.37,119.65) .. controls (343.41,116.76) and (345.78,114.45) .. (348.67,114.49) -- cycle ;
 
\draw  [color={rgb, 255:red, 208; green, 2; blue, 27 }  ,draw opacity=1 ][fill={rgb, 255:red, 208; green, 2; blue, 27 }  ,fill opacity=1 ] (380.67,115.29) .. controls (383.56,115.33) and (385.88,117.7) .. (385.84,120.59) .. controls (385.8,123.48) and (383.42,125.8) .. (380.53,125.76) .. controls (377.64,125.72) and (375.33,123.34) .. (375.37,120.45) .. controls (375.41,117.56) and (377.78,115.25) .. (380.67,115.29) -- cycle ;
 
\draw  [color={rgb, 255:red, 208; green, 2; blue, 27 }  ,draw opacity=1 ][fill={rgb, 255:red, 208; green, 2; blue, 27 }  ,fill opacity=1 ] (409.87,115.29) .. controls (412.76,115.33) and (415.08,117.7) .. (415.04,120.59) .. controls (415,123.48) and (412.62,125.8) .. (409.73,125.76) .. controls (406.84,125.72) and (404.53,123.34) .. (404.57,120.45) .. controls (404.61,117.56) and (406.98,115.25) .. (409.87,115.29) -- cycle ;

\draw  [color={rgb, 255:red, 208; green, 2; blue, 27 }  ,draw opacity=1 ][fill={rgb, 255:red, 208; green, 2; blue, 27 }  ,fill opacity=1 ] (439.87,114.89) .. controls (442.76,114.93) and (445.08,117.3) .. (445.04,120.19) .. controls (445,123.08) and (442.62,125.4) .. (439.73,125.36) .. controls (436.84,125.32) and (434.53,122.94) .. (434.57,120.05) .. controls (434.61,117.16) and (436.98,114.85) .. (439.87,114.89) -- cycle ;

\draw  [color={rgb, 255:red, 208; green, 2; blue, 27 }  ,draw opacity=1 ][fill={rgb, 255:red, 208; green, 2; blue, 27 }  ,fill opacity=1 ] (471.87,114.49) .. controls (474.76,114.53) and (477.08,116.9) .. (477.04,119.79) .. controls (477,122.68) and (474.62,125) .. (471.73,124.96) .. controls (468.84,124.92) and (466.53,122.54) .. (466.57,119.65) .. controls (466.61,116.76) and (468.98,114.45) .. (471.87,114.49) -- cycle ;

\draw  [color={rgb, 255:red, 208; green, 2; blue, 27 }  ,draw opacity=1 ][fill={rgb, 255:red, 208; green, 2; blue, 27 }  ,fill opacity=1 ] (501.47,114.89) .. controls (504.36,114.93) and (506.68,117.3) .. (506.64,120.19) .. controls (506.6,123.08) and (504.22,125.4) .. (501.33,125.36) .. controls (498.44,125.32) and (496.13,122.94) .. (496.17,120.05) .. controls (496.21,117.16) and (498.58,114.85) .. (501.47,114.89) -- cycle ;

\draw  [color={rgb, 255:red, 208; green, 2; blue, 27 }  ,draw opacity=1 ][fill={rgb, 255:red, 208; green, 2; blue, 27 }  ,fill opacity=1 ] (529.47,114.89) .. controls (532.36,114.93) and (534.68,117.3) .. (534.64,120.19) .. controls (534.6,123.08) and (532.22,125.4) .. (529.33,125.36) .. controls (526.44,125.32) and (524.13,122.94) .. (524.17,120.05) .. controls (524.21,117.16) and (526.58,114.85) .. (529.47,114.89) -- cycle ;
 
\draw    (100.4,209.73) -- (555.4,210.73) ;

\draw  [color={rgb, 255:red, 208; green, 2; blue, 27 }  ,draw opacity=1 ][fill={rgb, 255:red, 208; green, 2; blue, 27 }  ,fill opacity=1 ] (111.47,204.49) .. controls (114.36,204.53) and (116.68,206.9) .. (116.64,209.79) .. controls (116.6,212.68) and (114.22,215) .. (111.33,214.96) .. controls (108.44,214.92) and (106.13,212.54) .. (106.17,209.65) .. controls (106.21,206.76) and (108.58,204.45) .. (111.47,204.49) -- cycle ;

\draw  [color={rgb, 255:red, 208; green, 2; blue, 27 }  ,draw opacity=1 ][fill={rgb, 255:red, 208; green, 2; blue, 27 }  ,fill opacity=1 ] (142.67,204.89) .. controls (145.56,204.93) and (147.88,207.3) .. (147.84,210.19) .. controls (147.8,213.08) and (145.42,215.4) .. (142.53,215.36) .. controls (139.64,215.32) and (137.33,212.94) .. (137.37,210.05) .. controls (137.41,207.16) and (139.78,204.85) .. (142.67,204.89) -- cycle ;

\draw  [color={rgb, 255:red, 208; green, 2; blue, 27 }  ,draw opacity=1 ][fill={rgb, 255:red, 208; green, 2; blue, 27 }  ,fill opacity=1 ] (177.07,205.29) .. controls (179.96,205.33) and (182.28,207.7) .. (182.24,210.59) .. controls (182.2,213.48) and (179.82,215.8) .. (176.93,215.76) .. controls (174.04,215.72) and (171.73,213.34) .. (171.77,210.45) .. controls (171.81,207.56) and (174.18,205.25) .. (177.07,205.29) -- cycle ;

\draw  [color={rgb, 255:red, 208; green, 2; blue, 27 }  ,draw opacity=1 ][fill={rgb, 255:red, 208; green, 2; blue, 27 }  ,fill opacity=1 ] (209.07,204.89) .. controls (211.96,204.93) and (214.28,207.3) .. (214.24,210.19) .. controls (214.2,213.08) and (211.82,215.4) .. (208.93,215.36) .. controls (206.04,215.32) and (203.73,212.94) .. (203.77,210.05) .. controls (203.81,207.16) and (206.18,204.85) .. (209.07,204.89) -- cycle ;

\draw  [color={rgb, 255:red, 208; green, 2; blue, 27 }  ,draw opacity=1 ][fill={rgb, 255:red, 208; green, 2; blue, 27 }  ,fill opacity=1 ] (242.27,205.29) .. controls (245.16,205.33) and (247.48,207.7) .. (247.44,210.59) .. controls (247.4,213.48) and (245.02,215.8) .. (242.13,215.76) .. controls (239.24,215.72) and (236.93,213.34) .. (236.97,210.45) .. controls (237.01,207.56) and (239.38,205.25) .. (242.27,205.29) -- cycle ;

\draw  [color={rgb, 255:red, 208; green, 2; blue, 27 }  ,draw opacity=1 ][fill={rgb, 255:red, 208; green, 2; blue, 27 }  ,fill opacity=1 ] (280.27,204.89) .. controls (283.16,204.93) and (285.48,207.3) .. (285.44,210.19) .. controls (285.4,213.08) and (283.02,215.4) .. (280.13,215.36) .. controls (277.24,215.32) and (274.93,212.94) .. (274.97,210.05) .. controls (275.01,207.16) and (277.38,204.85) .. (280.27,204.89) -- cycle ;

\draw  [color={rgb, 255:red, 208; green, 2; blue, 27 }  ,draw opacity=1 ][fill={rgb, 255:red, 208; green, 2; blue, 27 }  ,fill opacity=1 ] (315.07,204.89) .. controls (317.96,204.93) and (320.28,207.3) .. (320.24,210.19) .. controls (320.2,213.08) and (317.82,215.4) .. (314.93,215.36) .. controls (312.04,215.32) and (309.73,212.94) .. (309.77,210.05) .. controls (309.81,207.16) and (312.18,204.85) .. (315.07,204.89) -- cycle ;

\draw  [color={rgb, 255:red, 208; green, 2; blue, 27 }  ,draw opacity=1 ][fill={rgb, 255:red, 208; green, 2; blue, 27 }  ,fill opacity=1 ] (349.07,204.89) .. controls (351.96,204.93) and (354.28,207.3) .. (354.24,210.19) .. controls (354.2,213.08) and (351.82,215.4) .. (348.93,215.36) .. controls (346.04,215.32) and (343.73,212.94) .. (343.77,210.05) .. controls (343.81,207.16) and (346.18,204.85) .. (349.07,204.89) -- cycle ;
 
\draw  [color={rgb, 255:red, 208; green, 2; blue, 27 }  ,draw opacity=1 ][fill={rgb, 255:red, 208; green, 2; blue, 27 }  ,fill opacity=1 ] (381.07,205.69) .. controls (383.96,205.73) and (386.28,208.1) .. (386.24,210.99) .. controls (386.2,213.88) and (383.82,216.2) .. (380.93,216.16) .. controls (378.04,216.12) and (375.73,213.74) .. (375.77,210.85) .. controls (375.81,207.96) and (378.18,205.65) .. (381.07,205.69) -- cycle ;

\draw  [color={rgb, 255:red, 208; green, 2; blue, 27 }  ,draw opacity=1 ][fill={rgb, 255:red, 208; green, 2; blue, 27 }  ,fill opacity=1 ] (410.27,205.69) .. controls (413.16,205.73) and (415.48,208.1) .. (415.44,210.99) .. controls (415.4,213.88) and (413.02,216.2) .. (410.13,216.16) .. controls (407.24,216.12) and (404.93,213.74) .. (404.97,210.85) .. controls (405.01,207.96) and (407.38,205.65) .. (410.27,205.69) -- cycle ;

\draw  [color={rgb, 255:red, 208; green, 2; blue, 27 }  ,draw opacity=1 ][fill={rgb, 255:red, 208; green, 2; blue, 27 }  ,fill opacity=1 ] (440.27,205.29) .. controls (443.16,205.33) and (445.48,207.7) .. (445.44,210.59) .. controls (445.4,213.48) and (443.02,215.8) .. (440.13,215.76) .. controls (437.24,215.72) and (434.93,213.34) .. (434.97,210.45) .. controls (435.01,207.56) and (437.38,205.25) .. (440.27,205.29) -- cycle ;

\draw  [color={rgb, 255:red, 208; green, 2; blue, 27 }  ,draw opacity=1 ][fill={rgb, 255:red, 208; green, 2; blue, 27 }  ,fill opacity=1 ] (472.27,204.89) .. controls (475.16,204.93) and (477.48,207.3) .. (477.44,210.19) .. controls (477.4,213.08) and (475.02,215.4) .. (472.13,215.36) .. controls (469.24,215.32) and (466.93,212.94) .. (466.97,210.05) .. controls (467.01,207.16) and (469.38,204.85) .. (472.27,204.89) -- cycle ;

\draw  [color={rgb, 255:red, 208; green, 2; blue, 27 }  ,draw opacity=1 ][fill={rgb, 255:red, 208; green, 2; blue, 27 }  ,fill opacity=1 ] (501.87,205.29) .. controls (504.76,205.33) and (507.08,207.7) .. (507.04,210.59) .. controls (507,213.48) and (504.62,215.8) .. (501.73,215.76) .. controls (498.84,215.72) and (496.53,213.34) .. (496.57,210.45) .. controls (496.61,207.56) and (498.98,205.25) .. (501.87,205.29) -- cycle ;

\draw  [color={rgb, 255:red, 208; green, 2; blue, 27 }  ,draw opacity=1 ][fill={rgb, 255:red, 208; green, 2; blue, 27 }  ,fill opacity=1 ] (529.87,205.29) .. controls (532.76,205.33) and (535.08,207.7) .. (535.04,210.59) .. controls (535,213.48) and (532.62,215.8) .. (529.73,215.76) .. controls (526.84,215.72) and (524.53,213.34) .. (524.57,210.45) .. controls (524.61,207.56) and (526.98,205.25) .. (529.87,205.29) -- cycle ;

\draw    (70.2,160.6) -- (583.4,161) ;

\draw  [color={rgb, 255:red, 74; green, 144; blue, 226 }  ,draw opacity=1 ][line width=1.5]  (49.35,104.95) .. controls (49.35,99.51) and (53.76,95.1) .. (59.2,95.1) -- (570,95.1) .. controls (575.44,95.1) and (579.85,99.51) .. (579.85,104.95) -- (579.85,134.5) .. controls (579.85,139.94) and (575.44,144.35) .. (570,144.35) -- (59.2,144.35) .. controls (53.76,144.35) and (49.35,139.94) .. (49.35,134.5) -- cycle ;

\draw  [color={rgb, 255:red, 74; green, 144; blue, 226 }  ,draw opacity=1 ][fill={rgb, 255:red, 74; green, 144; blue, 226 }  ,fill opacity=1 ] (505.63,234.69) -- (513.71,227.29) -- (513.84,241.95) -- cycle ;

\draw  [color={rgb, 255:red, 74; green, 144; blue, 226 }  ,draw opacity=1 ][fill={rgb, 255:red, 74; green, 144; blue, 226 }  ,fill opacity=1 ] (300.13,235.59) -- (308.21,228.18) -- (308.34,242.85) -- cycle ;

\draw  [color={rgb, 255:red, 74; green, 144; blue, 226 }  ,draw opacity=1 ][fill={rgb, 255:red, 74; green, 144; blue, 226 }  ,fill opacity=1 ] (110.64,235.95) -- (118.71,228.55) -- (118.84,243.21) -- cycle ;

\draw  [color={rgb, 255:red, 74; green, 144; blue, 226 }  ,draw opacity=1 ][fill={rgb, 255:red, 74; green, 144; blue, 226 }  ,fill opacity=1 ] (453.68,185.73) -- (445.61,193.14) -- (445.47,178.48) -- cycle ;

\draw  [color={rgb, 255:red, 74; green, 144; blue, 226 }  ,draw opacity=1 ][fill={rgb, 255:red, 74; green, 144; blue, 226 }  ,fill opacity=1 ] (319.68,186.49) -- (311.61,193.9) -- (311.47,179.24) -- cycle ;

\draw  [color={rgb, 255:red, 74; green, 144; blue, 226 }  ,draw opacity=1 ][fill={rgb, 255:red, 74; green, 144; blue, 226 }  ,fill opacity=1 ] (126.18,186.86) -- (118.11,194.27) -- (117.97,179.61) -- cycle ;

\draw  [color={rgb, 255:red, 74; green, 144; blue, 226 }  ,draw opacity=1 ][line width=1.5]  (65.85,195.95) .. controls (65.85,190.51) and (70.26,186.1) .. (75.7,186.1) -- (586.5,186.1) .. controls (591.94,186.1) and (596.35,190.51) .. (596.35,195.95) -- (596.35,225.5) .. controls (596.35,230.94) and (591.94,235.35) .. (586.5,235.35) -- (75.7,235.35) .. controls (70.26,235.35) and (65.85,230.94) .. (65.85,225.5) -- cycle ;

\draw  [color={rgb, 255:red, 74; green, 144; blue, 226 }  ,draw opacity=1 ][fill={rgb, 255:red, 74; green, 144; blue, 226 }  ,fill opacity=1 ] (488.96,143.72) -- (497.03,136.32) -- (497.17,150.98) -- cycle ;

\draw  [color={rgb, 255:red, 74; green, 144; blue, 226 }  ,draw opacity=1 ][fill={rgb, 255:red, 74; green, 144; blue, 226 }  ,fill opacity=1 ] (283.46,144.62) -- (291.54,137.21) -- (291.67,151.88) -- cycle ;

\draw  [color={rgb, 255:red, 74; green, 144; blue, 226 }  ,draw opacity=1 ][fill={rgb, 255:red, 74; green, 144; blue, 226 }  ,fill opacity=1 ] (93.96,144.98) -- (102.04,137.58) -- (102.17,152.24) -- cycle ;

\draw  [color={rgb, 255:red, 74; green, 144; blue, 226 }  ,draw opacity=1 ][fill={rgb, 255:red, 74; green, 144; blue, 226 }  ,fill opacity=1 ] (437.01,94.76) -- (428.94,102.18) -- (428.79,87.51) -- cycle ;

\draw  [color={rgb, 255:red, 74; green, 144; blue, 226 }  ,draw opacity=1 ][fill={rgb, 255:red, 74; green, 144; blue, 226 }  ,fill opacity=1 ] (303.01,95.52) -- (294.94,102.93) -- (294.79,88.27) -- cycle ;

\draw  [color={rgb, 255:red, 74; green, 144; blue, 226 }  ,draw opacity=1 ][fill={rgb, 255:red, 74; green, 144; blue, 226 }  ,fill opacity=1 ] (109.51,95.89) -- (101.44,103.3) -- (101.29,88.64) -- cycle ;

\draw (110.13,158.99) node [anchor=south east] [inner sep=0.75pt]    {$\mathbb{R}$};

\draw (348.67,111.09) node [anchor=south] [inner sep=0.75pt]    {$\lambda _{j}$};

\draw (100.12,92.15) node [anchor=south west] [inner sep=0.75pt]    {$\Gamma _{1}$};

\draw (176.62,238.95) node [anchor=north west][inner sep=0.75pt]    {$\Gamma _{2}$};

\draw (344.28,248.13) node [anchor=north west][inner sep=0.75pt]  [rotate=-179.48]  {$+$};

\draw (158.3,251.32) node [anchor=north west][inner sep=0.75pt]  [rotate=-179.48]  {$+$};

\draw (405.19,183.08) node [anchor=north west][inner sep=0.75pt]  [rotate=-179.48]  {$+$};

\draw (168.68,183.22) node [anchor=north west][inner sep=0.75pt]  [rotate=-179.48]  {$+$};

\draw (425.94,160.89) node [anchor=north west][inner sep=0.75pt]  [rotate=-179.48]  {$+$};

\draw (179.94,156.62) node [anchor=north west][inner sep=0.75pt]  [rotate=-179.48]  {$+$};

\draw (447.84,88.68) node [anchor=north west][inner sep=0.75pt]  [rotate=-179.48]  {$+$};

\draw (148.37,93.4) node [anchor=north west][inner sep=0.75pt]  [rotate=-179.48]  {$+$};

\draw (424.79,143.89) node [anchor=north west][inner sep=0.75pt]  [rotate=-179.48]  {$-$};

\draw (447.02,108.69) node [anchor=north west][inner sep=0.75pt]  [rotate=-179.48]  {$-$};

\draw (179.81,142.12) node [anchor=north west][inner sep=0.75pt]  [rotate=-179.48]  {$-$};

\draw (149.5,108.39) node [anchor=north west][inner sep=0.75pt]  [rotate=-179.48]  {$-$};

\draw (405.32,197.58) node [anchor=north west][inner sep=0.75pt]  [rotate=-179.48]  {$-$};

\draw (167.82,202.14) node [anchor=north west][inner sep=0.75pt]  [rotate=-179.48]  {$-$};

\draw (157.66,235.32) node [anchor=north west][inner sep=0.75pt]  [rotate=-179.48]  {$-$};

\draw (341.62,230.15) node [anchor=north west][inner sep=0.75pt]  [rotate=-179.48]  {$-$};

\end{tikzpicture}

		\caption{Example of contour $\Gamma_1,\Gamma_2$}
		\label{fig:contour}
	\end{figure}

We therefore have, for every positive integer $N$, our $N$-soliton scattering data 
\begin{eqnarray}
    \mathcal{D}_{N} = \left\{r \equiv 0,
\left\{  
\lambda_{j}, c_{j}
\right\}_{j=1}^{N}
    \right\},
\end{eqnarray}
where $\lambda_{j}$ is defined in \eqref{eq:lamsamp} and the normalzation constants $c_{j}$ are defined in \eqref{eq:normsamp}.
For each positive integer $N$, these scattering data correspond to an $N$-soliton solution to the NLS equation, which we will denote by 
\begin{eqnarray}
\psi_{N}(x,t) = \psi(x,t;\mathcal{D}_{N}).
\end{eqnarray}
In what follows, unless otherwise specified, we will assume that the scattering data under consideration is a sequence of $N$-soliton condensate scattering data, with $N$ large and increasing towards $\infty$.  We will refer to this collection of assumptions by saying, ``assuming a sequence of $N$-soliton condensate scattering data'', or ``under the $N$-soliton condensate scattering data assumption''.

\subsection{Riemann-Hilbert problem for \texorpdfstring{$\psi_{N}(x,t)$}{psi}}
Given the scattering data $\mathcal{D}_{N}$, the Riemann-Hilbert problem \ref{rhp:reflection_RHP} simplifies significantly, since $r(z) \equiv 0$ for all real $z$.  Indeed, since the jump relation across the real axis \eqref{eq:MJump} becomes the trivial relation $M_{+}(z) = M_{-}(z)$ for all real $z$,  Morera's Theorem tells us that $M(z)$ is then analytic across the real axis, and hence $M(z)$ is meromorphic in $\mathbb{C}$.  We will make one standard transformation, from $M(z)$ to a new matrix $A(z)$, following (for example) \cite{Girotti2021}.  We use the contours $\Gamma_{1}$ and $\Gamma_{2}$ which are provided to us under the $N$-soliton condensate scattering assumption, $\Gamma_{1}$ encircling the interval $\eta_{1}$ and $\Gamma_{2}$ its Schwarz reflection.  Each of these contours is taken to be oriented clockwise, as shown in Fig. \ref{fig:contour}.  We then define $\rhp(z)$ as follows:
	\begin{equation}
		\rhp(z) = \begin{cases}
			M(z)\begin{pmatrix}
				1& 0 \\
				- \sum_{j=1}^N \frac{c_j}{z-\lambda_j}e^{2i\theta(\lambda_j)} & 1 
			\end{pmatrix}, & z \textit{ inside } \Gamma_1 \\ \\[2pt]
			M(z)\begin{pmatrix}
			1 & 	\sum_{j=1}^N \frac{\wo c_j}{z-\wo \lambda_j}e^{-2i\theta(\wo{\lambda}_j)} \\
		   0& 1 
		\end{pmatrix}, & z \textit{ inside } \Gamma_2 \\
		\\[2pt]
			M(z), & \textit{otherwise}.
		\end{cases}
	\end{equation}
Now $M(z)$ has simple poles at the eigenvalues $\lambda_{j}$ and $\overline{\lambda_{j}}$, but the definition above is chosen so that $\rhp(z)$ is analytic inside the contours $\Gamma_{1}$ and $\Gamma_{2}$.  Indeed, the reader may check from the residue conditions \eqref{eq:rescondCP} and \eqref{eq:rescondCM} that (for example) the pole of $M_{11}(z)$ at $\lambda_{k}$ is canceled since 
\begin{eqnarray}
    A_{N;11}(z) = M_{11}(z) - \frac{c_{k} M_{12}(z)e^{2i\theta(\lambda_k)}}{z - \lambda_{k}} + \left<\mbox{analytic at $\lambda_{k}$} \right>,
\end{eqnarray}
and the first two terms have the same exact residue at $\lambda_{k}$.  Similar arguments show that the matrix $\rhp(z)$ is analytic inside $\Gamma_{1}$ and inside $\Gamma_{2}$, and it is also obviously analytic for $z$ outside the two contours.  Finally, $\rhp(z)$ clearly has boundary values as $z$ approaches the contours $\Gamma_{1}$ and $\Gamma_{2}$ (with different boundary values as $z$ approaches each contour from inside or outside).

\textbf{Notation:  boundary values for $z$ on a contour}:  Following standard convention, for $z$ in the contour $\Gamma_{1}$ (or $\Gamma_{2}$), we will let $A_{N,+}(z)$ denote the boundary value of $\rhp(z)$ obtained by letting $z'$ approach $z$ from outside $\Gamma_{1}$ (or outside $\Gamma_{2}$):
\begin{eqnarray}
    A_{N,+}(z) = \lim_{\substack{z' \to z \\ z' \in \mbox{ \small exterior of }  \Gamma_{j}}} \rhp(z'), \ \ z \in \Gamma_{j},
\end{eqnarray}
and we will let $A_{N,-}(z)$ denote the boundary value from inside $\Gamma_{1}$ (or $\Gamma_{2}$).  (Note that this is consistent with the usual convention that the $+$ side of a contour is on the left of the contour as one traverses it according to its orientation).

The new matrix-valued function $\rhp(z)$ satisfies a Riemann-Hilbert problem, which we will refer to as the \textit{$N$-soliton condensate Riemann-Hilbert problem}.  

For this Riemann-Hilbert problem, the poles are accumulating on the segments $\eta_1,\eta_2$ which we parametrize for future convenience:
\begin{equation}
\label{eq:sym_contour}
\eta_1 = \left\{z\in\C\,\Big \vert \, z = a-2at+ib,\; t\in(0,1) \right\}\,,\quad \eta_2 = \left\{z\in\C\,\Big \vert \, z = -a +2ta - ib,\; t\in(0,1) \right\}\,.
\end{equation}
In particular, $\eta_1$ is the horizontal segment from $A$ to $-\wo A$, and $\eta_2$ is the one from $-A$ to $\wo A$. 

  Under the $N$-soliton scattering assumption, we have the following Riemann-Hilbert problem.  
	\begin{RHP}[$N$-soliton Condensate Riemann-Hilbert Problem]
 \label{rhp:A}Find a $2 \times 2 $ matrix function $A_N = A_N(z) = A_N(z; x, t)$ such that
	
		\begin{itemize}
			\item $A_N$ is analytic on $\mathbb{C} \setminus \{ \Gamma_1 \cup \Gamma_2\}$
			\item $A_{N,+}(z) = A_{N,-}(z)J_{A}(z)$ where the jump matrix $J_{A_N}(z)$ is given by
   \begin{equation}
   \label{eq:AJumps}
       J_{A_N}(z) =
           \begin{cases}
                 \begin{pmatrix}
					1 & 0 \\
					\sum_{j=1}^N \frac{c_j}{z-\lambda_j}e^{2i\theta(\lambda_j)} & 1
				\end{pmatrix}, \quad & z\in \Gamma_1 \\[4pt] \\
                \begin{pmatrix}
					1 & -\sum_{j=1}^N \frac{\wo c_j}{z-\wo \lambda_j}e^{-2i\theta(\wo{\lambda}_j)}\\
					 0& 1
				\end{pmatrix}, \quad & z\in \Gamma_2
           \end{cases}
   \end{equation}
   \item 
$\displaystyle 
A_N(z) = I + \frac{1}{2 i z} \pmtwo{-\int_{x}^{\infty} |\psi_{N}(s,t)|^{2}\di s}{\psi_{N}(x,t)}{\overline{\psi_{N}(x,t)}}{\int_{x}^{\infty} |\psi_{N}(s,t)|^{2}\di s}
+ \mathcal{O} \left(\frac{1}{z^{2}} \right), \ \mbox{ as } z \to \infty$.

	\end{itemize}
	\end{RHP}
Note that the above $z \to \infty$ asymptotic condition describes exactly how to extract the $N$-soliton condensate solution $\psi_{N}(x,t)$ from the solution $A_N$ of the Riemann-Hilbert problem.  Usually the asymptotic normalization condition is described in the weak form $A_N = I + \mathcal{O} \left( \frac{1}{z} \right)$, but in this case, since the contours $\Gamma_{1}$ and $\Gamma_{2}$ are bounded, $A$ possesses a complete asymptotic expansion in powers of $\frac{1}{z}$.

\subsection{Riemann-Hilbert problem for the soliton gas}

The $N$-soliton condensate Riemann-Hilbert problem behaves singularly as the number of eigenvalues grows to infinity, because the sums appearing in the jump relationships \eqref{eq:AJumps} do not converge as $N \to \infty$. However, their asymptotic behavior can be determined:
\begin{eqnarray}\label{eq:eqerror}
&&
\sum_{j=1}^{N} \frac{c_{j}e^{ 2 i \theta(\lambda_{j})}
}{z-\lambda_{j}} = - N \int_{\eta_{1}} \frac{h(\lambda) \rho(\lambda) e^{ 2 i \theta(\lambda)} d\lambda}{z- \lambda} + \mathcal{O}\left( \frac{1}{N} \right) \ , \\
&&
\sum_{j=1}^N \frac{\wo c_j}{z-\wo \lambda_j}e^{-2i\theta(\wo{\lambda}_j)} =  N \int_{\eta_{2}} \frac{{h(\lambda)} \rho(\lambda) e^{ - 2 i \theta(\lambda)}\di s}{z - \lambda} + \mathcal{O}\left( \frac{1}{N} \right)\,,
\end{eqnarray}
where we recall that the contour $\eta_{1}$ is oriented from right to left, while $\eta_{2}$ is oriented from left to right, and in both cases the error term $\mathcal{O} \left( \frac{1}{N} \right)$ is uniform for all $z$ in the respective contour $\Gamma_{1}$ or $\Gamma_{2}$.  Indeed, under the $N$-soliton condensate assumption 3, a complete asymptotic expansion for this sum can be calculated, because the eigenvalues $\{\lambda_{j}\}_{j=1}^{N}$ have been selected to arrive at a uniform mid-point approximation, which is known to converge with an error $\cO(N^{-2})$.  The next term in the expansion can be computed explicitly, we perform such computation in Appendix \ref{appendix:A}.

If we replace the finite sum by its leading order asymptotic behavior, we arrive at a Riemann-Hilbert problem which still contains the large parameter $N$.  It is this Riemann-Hilbert problem which we study as $N \to \infty$.  The associated solution to the NLS equation, $\psi_{SG}=\psi_{SG}(x,t) = \psi_{SG}(x,t;N)$, is the soliton gas (or soliton condensate) solution.  

\begin{RHP}[Soliton gas condensate Riemann-Hilbert problem] 
 \label{rhp:Atil}Find a $2 \times 2$ matrix function $\tilde{A} = \tilde{A}(z) = \tilde{A}(z; x, t, N)$ such that
	
		\begin{itemize}
			\item $\tilde{A}$ is analytic on $\mathbb{C} \setminus \{ \Gamma_1 \cup \Gamma_2\}$
			\item $\tilde{A}_{+}(z) = \tilde{A}_{-}(z)J_{\tilde{A}}(z)$ where the jump matrix $J_{\tilde{A}}(z)$ is given by
   \begin{equation}
   \label{eq:AtJumps}
       J_{\tilde{A}}(z) =
           \begin{cases}
                \pmtwo{1}{0}
                {N \int_{\eta_{1}} \frac{h(\lambda) \rho(\lambda) e^{ 2 i \theta(\lambda)} d\lambda}{\lambda-z} }{1}, 
                \quad & z\in \Gamma_1 \\[4pt] \\
                \pmtwo{1}{N \int_{\eta_{2}} \frac{{h(\lambda)} \rho(\lambda) e^{ - 2 i \theta(\lambda)}\di s}{\lambda-z}}{0}{1}, \quad & z\in \Gamma_2
           \end{cases}
   \end{equation}
   \item 
$\displaystyle 
\tilde{A}(z) = I + \frac{1}{2 i z} \pmtwo{-\int_{x}^{\infty} |\psi_{SG}(s,t;N)|^{2}\di s}{\psi_{SG}(x,t;N)}{\overline{\psi_{SG}(x,t;N)}}{\int_{x}^{\infty} |\psi_{SG}(s,t;N)|^{2}\di s}
+ \mathcal{O} \left(\frac{1}{z^{2}} \right), \ \mbox{ as } z \to \infty$.

	\end{itemize}
	\end{RHP}

This Riemann-Hilbert problem arises formally, by replacing the jump matrices \eqref{eq:AJumps} with \eqref{eq:AtJumps}.  Since it is not related to Riemann-Hilbert problem \ref{rhp:A} by explicit transformations, the existence and uniqueness of a solution must be established independently.  Note that this statement extends to the soliton gas condensate solution $\psi_{SG}(x,t;N)$.  Uniqueness for the Riemann-Hilbert problem \ref{rhp:Atil} (regardless of whether or not the solution is known to exist) is completely straightforward since $\mbox{det}J_{\tilde{A}} \equiv 1$.  See, for example, the first paragraph of the proof of Theorem 3.1 on P. 1503 of \cite{DKMVZ1}.  Regarding existence, the crucial fact is that the jump matrices obey the symmetry 
\begin{eqnarray}J_{\tilde{A}}(z) = \pmtwo{0}{1}{1}{0} 
\overline{J_{\tilde{A}}(\overline{z})} \pmtwo{0}{1}{1}{0}.
\end{eqnarray}
This permits the direct application of an existence theory via the connection between Riemann-Hilbert problems and Fredholm singular integral operators (see Theorems 9.1, 9.2, and 9.3 in \cite{Zhou89}, or the proof of Theorem 5.3 in \cite{DKMVZ3}). Regarding the asymptotic behaviour of the solution, using a standard dressing argument, see for example \cite{Bilman2020Duke,Ken_NLS}, one can verify that the $z^{-1}$ coefficient in the large $z$ expansion for $\wt A(z)$ has the given form. Thus we state without further discussion the following result:
\begin{theorem}
    Under the N -soliton condensate scattering data assumption, the Riemann-Hilbert problem \ref{rhp:Atil} possess a unique solution, and determines $\psi_{SG}(x,t;N)$, a solution of the NLS equation.
\end{theorem}

\subsection{Statement of the results}

In Sections \ref{sec:dressing} - \ref{sec:local} we determine the asymptotic behavior of the solution $\tilde{A}$ of the Riemann-Hilbert problem \ref{rhp:Atil}.  The main application of this work is to determine the asymptotic behavior of the soliton gas condensate $\psi_{SG}(x,t;N)$, which is described in the following theorem.  

\begin{theorem}
\label{thm:main_sym}
    Under the $N$-soliton condensate scattering data assumption, for all $(x,t)$ in a compact set $\fK$, there is $N_{0}$ so that for all $N > N_{0}$, $\psi_{SG}(x,t;N)$ satisfies the asymptotic description \eqref{eq:solution_limit}.  Specifically, for $\psi_{SG} $, we have 

\begin{equation}
\label{eq:solution_main}
\begin{split}
            &\psi_{SG}(x,t;N)  =\\ &  -2 i \frac{\Re(A)\Im(A)}{|A|} \sd\left( -2|A|x + \frac{2K(\cos(\theta))}{\pi}\ln(N) -  \frac{2|A|}{\pi}\Re\left(\int_{\eta_1} \frac{\log(2\pi h(s)\rho(s))}{R_+(s)}\di  s\right); \cos(\theta)\right) e^{it(A^2+\wo A^2) + i \phi_0} \\
            \empty & +\mathcal{O} \left(\frac{1}{\log{N}}\right),
\end{split}
\end{equation}
where the error term $\mathcal{O}\left( \frac{1}{\log{N} } \right)$ is uniform for all $(x,t)$ in the compact set $\fK$.  Here $\sd$ is the Jacobi elliptic function $\sd$ \cite[Chap. 22]{dlmf},  
and $R(z)$ and $\phi_0$ are given by

\begin{align}
\label{eq:R_def}
    &R(z) = (z-A)^{\frac{1}{2}}(z+A)^{\frac{1}{2}}(z-\wo A)^{\frac{1}{2}}(z+\wo A)^{\frac{1}{2}}\,,\\
    \label{eq:phi_0_def}
    &\phi_0 = \ln(N) \frac{K(\cos(\theta))}{K(\sin(\theta))} + \Re\left( \int_{\eta_1} \frac{s\ln(2\pi h(s)\rho(s))}{R_+(s)}\di s\right)\,, \quad \theta = \arg(A)\
\end{align}
where the function $K(k)$ is the complete elliptic integral of the fist kind \cite[Chap. 19]{dlmf}, the branch-cuts are the canonical ones, and the integration is performed on the straight line connecting the endpoints.
\end{theorem}

\begin{remark}
It is interesting to observe that the leading order asymptotic behavior of the modulus of the soliton gas condensate solution $\psi_{SG}(x,t;N)$ \eqref{eq:solution_main} does not depend on time.  The explanation for this fact is that our assumption on the scattering data, the poles are accumulating on symmetric contours. Essentially, the soliton gas condensate is in a sort of equilibrium, and can be interpreted as an elliptic wave with zero velocity.  This could also be interpreted as another form of soliton-shielding, similar to  \cite{Orsatti23,bertola2024}. In particular, we want to emphasize that the solution \eqref{eq:solution_main} and the solution \cite[eq. (4.44)]{bertola2024} are similar, indeed they both arise as solutions of similar  RHPs, the two main differences are that in our case the contours where the poles are accumulating are horizontal, and the norming constants $c_j$ are bounded away from zero, while in \cite{bertola2024} the contours are vertical and the constants are uniformly $\mathcal{O}(N^{-1})$.

We notice that, in view of the Galilean invariance of the NLS equation, for $v\in \R$,  $\psi(x-vt,t)e^{ivx-i\frac{v^2t}{2}}$ solves the NLS equations \eqref{eq:NLS} whenever $\psi(x,t)$ does. Furthermore, if the scattering maps gives $\psi(x,0)\to\{r(z),\{\lambda_j,c_j\}_{j=1}^N\}$ then $\psi(x,0)e^{ivx}\to\{r\left(z+\frac{v}{2}\right),\{\lambda_j-\frac{v}{2},c_j\}_{j=1}^N\}$; therefore, in the setting where the two line segments $\eta_{1} $ and $\eta_{2}$ are shifted so that they are not symmetric with respect to the origin, one can always reduce to the symmetric situation up to a simple scaling. Finally, it is worth noticing that the solution depends just on the product of $h(z),\rho(z)$ and not on each function separately.  Fig. \ref{fig:solution} shows plots of the leading order asymptotic of the solution of the NLS equation for different values of $N$.
\end{remark}

\begin{Proposition}
\label{prop:last_detail}
    Let $A_N(z),\wt A(z)$ be the solutions of RHP \ref{rhp:A} and RHP \ref{rhp:Atil} respectively.
Under the $N$-soliton condensate scattering data assumption, there is $N_{0}$ so that for all  $N > N_{0}$, we have  
    \begin{equation}
    \label{eq:relation_A_A_tildenew}
        A_N(z) = \left(I + \mathcal{O} \left( \frac{1}{N \left( 1 + |z| \right) } \right) \right) \wt A(z)\,.
    \end{equation}
    where the error term, $\mathcal{O} \left( \frac{1}{N \left( 1 + |z| \right) } \right)$ is uniform in the entire complex plane, and $(x,t)$ in a compact set.
\end{Proposition}
The proof of this proposition is presented in Appendix  \ref{appendix:proof2}, and relies on the complete asymptotic description of the solution $\tilde{A}$ to the Riemann-Hilbert problem \ref{rhp:Atil}, which as mentioned above is developed in Sections \ref{sec:dressing} - \ref{sec:local}.

It is an immediate consequence of this uniform approximation that the $N$-soliton solution is extremely well approximated by the soliton gas condensate solution.

\begin{Corollary}
\label{cor:NSol}
    Under the $N$-soliton condensate scattering data assumption, for all $(x,t)$ in a compact set, there is $N_{0}$ and a constant $c_{0}$ so that for all $N > N_{0}$, 
    we have
    \begin{eqnarray}
      \left|  \psi_{N}(x,t) - \psi_{SG}(x,t;N)  \right| \le \frac{c_{0}}{N}.
      \end{eqnarray}   
\end{Corollary}

In the next four sections, we develop the asymptotic analysis of the solution of the Riemann-Hilbert problem \ref{rhp:Atil}, to establish, in Section \ref{sec:solution},  the asymptotic behavior of the soliton gas condensate solution $\psi_{SG}(x,t;N)$.  In Section \ref{sec:kinetic} we consider a more complex solution which contains one additional soliton, and using our analysis we provide a rigorous validation of the kinetic theory of solitons in this setting.  We prove Proposition \ref{prop:last_detail}, Corollary \ref{cor:NSol} and other technical results in Appendix \ref{appendix:A}.

\begin{figure}[ht]
    \centering
    \includegraphics[width=0.45\linewidth]{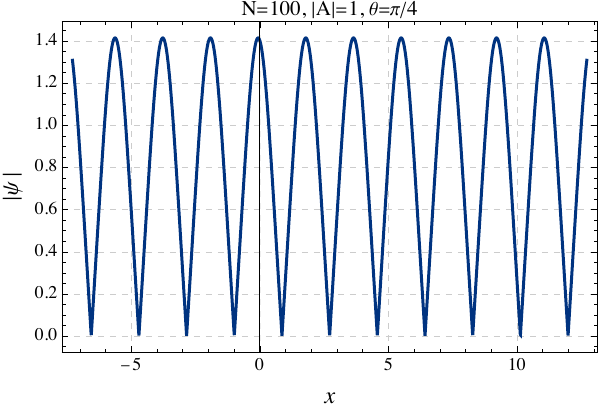}
    \includegraphics[width=0.45\linewidth]{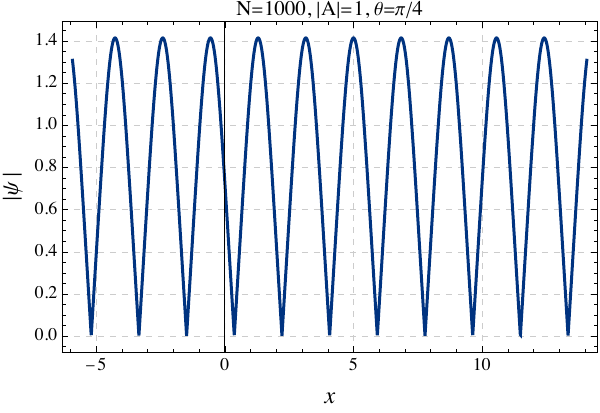}

    \includegraphics[width=0.45\linewidth]{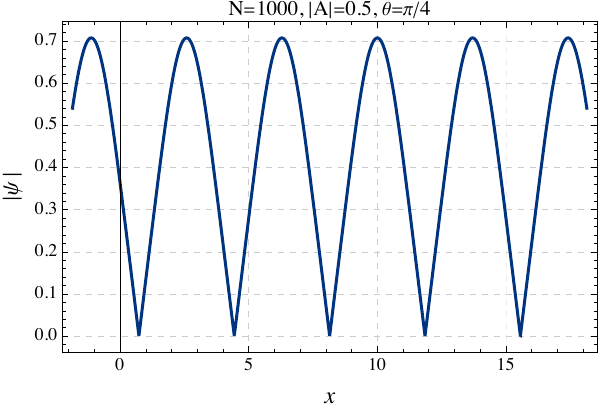}
    \includegraphics[width=0.45\linewidth]{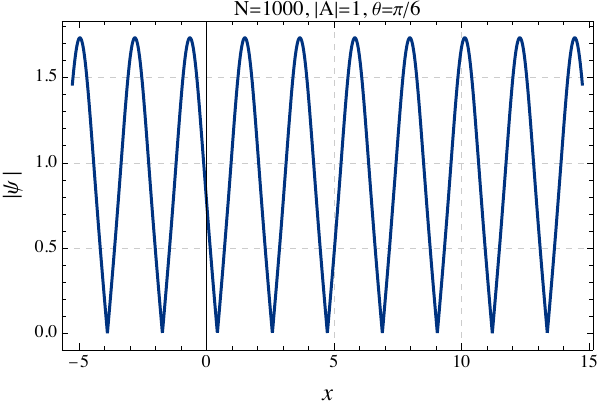}
    
    \caption{Solution to the NLS equation \eqref{eq:NLS} in assumptions \ref{ass:sym}. Here $A,N$ and $\theta$ are specified in the plots}
    \label{fig:solution}
\end{figure}

 	\section{Dressing transformations}
	\label{sec:dressing}
In this section, we make two transformations, starting with  RHP \ref{rhp:Atil}, which yield a new Riemann-Hilbert problem. This new Riemann-Hilbert problem is not yet in the class of small-norm problems, but is a stepping-stone toward that objective.  Recall that we have already turned the residue conditions of $M$ at $\lambda_j$ into suitable jump conditions for the function $\wt A$ across $\Gamma_1 \cup \Gamma_2$ (see RHP \ref{rhp:Atil}). 

	\begin{figure}
		\centering

\tikzset{every picture/.style={line width=0.75pt}}        

\begin{tikzpicture}[x=0.5pt,y=0.5pt,yscale=-1,xscale=1]

\draw [color={rgb, 255:red, 208; green, 2; blue, 27 }  ,draw opacity=1 ][line width=1.5]    (120,124.63) -- (575,125.63) ;

\draw [color={rgb, 255:red, 208; green, 2; blue, 27 }  ,draw opacity=1 ][line width=1.5]    (120.4,215.03) -- (575.4,216.03) ;

\draw    (90.2,165.9) -- (603.4,166.3) ;

\draw  [color={rgb, 255:red, 74; green, 144; blue, 226 }  ,draw opacity=1 ][line width=1.5]  (69.35,110.25) .. controls (69.35,104.81) and (73.76,100.4) .. (79.2,100.4) -- (590,100.4) .. controls (595.44,100.4) and (599.85,104.81) .. (599.85,110.25) -- (599.85,139.8) .. controls (599.85,145.24) and (595.44,149.65) .. (590,149.65) -- (79.2,149.65) .. controls (73.76,149.65) and (69.35,145.24) .. (69.35,139.8) -- cycle ;

\draw  [color={rgb, 255:red, 74; green, 144; blue, 226 }  ,draw opacity=1 ][fill={rgb, 255:red, 74; green, 144; blue, 226 }  ,fill opacity=1 ] (525.63,239.99) -- (533.71,232.59) -- (533.84,247.25) -- cycle ;

\draw  [color={rgb, 255:red, 74; green, 144; blue, 226 }  ,draw opacity=1 ][fill={rgb, 255:red, 74; green, 144; blue, 226 }  ,fill opacity=1 ] (320.13,240.89) -- (328.21,233.48) -- (328.34,248.15) -- cycle ;

\draw  [color={rgb, 255:red, 74; green, 144; blue, 226 }  ,draw opacity=1 ][fill={rgb, 255:red, 74; green, 144; blue, 226 }  ,fill opacity=1 ] (130.64,241.25) -- (138.71,233.84) -- (138.84,248.51) -- cycle ;

\draw  [color={rgb, 255:red, 74; green, 144; blue, 226 }  ,draw opacity=1 ][fill={rgb, 255:red, 74; green, 144; blue, 226 }  ,fill opacity=1 ] (473.68,191.03) -- (465.61,198.44) -- (465.47,183.78) -- cycle ;

\draw  [color={rgb, 255:red, 74; green, 144; blue, 226 }  ,draw opacity=1 ][fill={rgb, 255:red, 74; green, 144; blue, 226 }  ,fill opacity=1 ] (339.68,191.79) -- (331.61,199.2) -- (331.47,184.53) -- cycle ;

\draw  [color={rgb, 255:red, 74; green, 144; blue, 226 }  ,draw opacity=1 ][fill={rgb, 255:red, 74; green, 144; blue, 226 }  ,fill opacity=1 ] (146.18,192.16) -- (138.11,199.57) -- (137.97,184.91) -- cycle ;

\draw  [color={rgb, 255:red, 74; green, 144; blue, 226 }  ,draw opacity=1 ][line width=1.5]  (85.85,201.25) .. controls (85.85,195.81) and (90.26,191.4) .. (95.7,191.4) -- (606.5,191.4) .. controls (611.94,191.4) and (616.35,195.81) .. (616.35,201.25) -- (616.35,230.8) .. controls (616.35,236.24) and (611.94,240.65) .. (606.5,240.65) -- (95.7,240.65) .. controls (90.26,240.65) and (85.85,236.24) .. (85.85,230.8) -- cycle ;

\draw  [color={rgb, 255:red, 74; green, 144; blue, 226 }  ,draw opacity=1 ][fill={rgb, 255:red, 74; green, 144; blue, 226 }  ,fill opacity=1 ] (508.96,149.02) -- (517.03,141.62) -- (517.17,156.28) -- cycle ;

\draw  [color={rgb, 255:red, 74; green, 144; blue, 226 }  ,draw opacity=1 ][fill={rgb, 255:red, 74; green, 144; blue, 226 }  ,fill opacity=1 ] (303.46,149.92) -- (311.54,142.51) -- (311.67,157.18) -- cycle ;

\draw  [color={rgb, 255:red, 74; green, 144; blue, 226 }  ,draw opacity=1 ][fill={rgb, 255:red, 74; green, 144; blue, 226 }  ,fill opacity=1 ] (113.96,150.28) -- (122.04,142.88) -- (122.17,157.54) -- cycle ;

\draw  [color={rgb, 255:red, 74; green, 144; blue, 226 }  ,draw opacity=1 ][fill={rgb, 255:red, 74; green, 144; blue, 226 }  ,fill opacity=1 ] (457.01,100.06) -- (448.94,107.47) -- (448.79,92.81) -- cycle ;

\draw  [color={rgb, 255:red, 74; green, 144; blue, 226 }  ,draw opacity=1 ][fill={rgb, 255:red, 74; green, 144; blue, 226 }  ,fill opacity=1 ] (323.01,100.82) -- (314.94,108.23) -- (314.79,93.57) -- cycle ;

\draw  [color={rgb, 255:red, 74; green, 144; blue, 226 }  ,draw opacity=1 ][fill={rgb, 255:red, 74; green, 144; blue, 226 }  ,fill opacity=1 ] (129.51,101.19) -- (121.44,108.6) -- (121.29,93.94) -- cycle ;

\draw  [color={rgb, 255:red, 208; green, 2; blue, 27 }  ,draw opacity=1 ][fill={rgb, 255:red, 208; green, 2; blue, 27 }  ,fill opacity=1 ] (216.87,124.85) -- (224.79,117.28) -- (225.22,131.94) -- cycle ;

\draw  [color={rgb, 255:red, 208; green, 2; blue, 27 }  ,draw opacity=1 ][fill={rgb, 255:red, 208; green, 2; blue, 27 }  ,fill opacity=1 ] (230.01,215.78) -- (221.77,223.01) -- (221.96,208.35) -- cycle ;

\draw  [color={rgb, 255:red, 208; green, 2; blue, 27 }  ,draw opacity=1 ][fill={rgb, 255:red, 208; green, 2; blue, 27 }  ,fill opacity=1 ] (433.37,125.85) -- (441.29,118.28) -- (441.72,132.94) -- cycle ;

\draw  [color={rgb, 255:red, 208; green, 2; blue, 27 }  ,draw opacity=1 ][fill={rgb, 255:red, 208; green, 2; blue, 27 }  ,fill opacity=1 ] (447.01,215.69) -- (438.95,223.11) -- (438.79,208.44) -- cycle ;

\draw (157.13,162.29) node [anchor=south east] [inner sep=0.75pt]    {$\mathbb{R}$};

\draw (120.12,97.45) node [anchor=south west] [inner sep=0.75pt]    {$\Gamma _{1}$};

\draw (196.62,244.25) node [anchor=north west][inner sep=0.75pt]    {$\Gamma _{2}$};

\draw (400.28,213.43) node [anchor=north west][inner sep=0.75pt]  [rotate=-179.48]  {$+$};

\draw (188.68,215.02) node [anchor=north west][inner sep=0.75pt]  [rotate=-179.48]  {$+$};

\draw (413.94,145.92) node [anchor=north west][inner sep=0.75pt]  [rotate=-179.48]  {$+$};

\draw (250.87,142.2) node [anchor=north west][inner sep=0.75pt]  [rotate=-179.48]  {$+$};

\draw (414.31,124.42) node [anchor=north west][inner sep=0.75pt]  [rotate=-179.48]  {$-$};

\draw (249,124.69) node [anchor=north west][inner sep=0.75pt]  [rotate=-179.48]  {$-$};

\draw (188.82,234.44) node [anchor=north west][inner sep=0.75pt]  [rotate=-179.48]  {$-$};

\draw (399.62,235.45) node [anchor=north west][inner sep=0.75pt]  [rotate=-179.48]  {$-$};

\draw (118,121.23) node [anchor=south east] [inner sep=0.75pt]    {$\eta _{1}$};

\draw (118.4,211.63) node [anchor=south east] [inner sep=0.75pt]    {$\eta _{2}$};
\end{tikzpicture}
		\caption{Contours $\Gamma_1,\Gamma_2,\eta_1,\eta_2$}
		\label{fig:double_contour}
	\end{figure}
Following the procedure highlighted in Proposition 2.3 and Lemma 2.4 of \cite{Girotti2021}, we move the jump condition from $\Gamma_1,\Gamma_2$ to $\eta_1,\eta_2$. To accomplish this, we define the function
	\begin{equation}\label{e:eqfenB}
		B(z) = \begin{cases}
			\wt A(z) \begin{pmatrix}
				1 & 0 \\
				N\int_{\eta_1} \frac{h(\lambda)\rho(\lambda)}{\lambda - z}\red{e^{2i\theta(\lambda)}}\di \lambda & 1
			\end{pmatrix} \,, \quad z \text{ inside } \Gamma_1 \\ \\[2pt]
			\wt A(z)\begin{pmatrix}
				1 & N \int_{\eta_2} \frac{ h(\lambda) \rho(\lambda)}{\lambda-z}\red{e^{-2i\theta( \lambda)}} \di \lambda \\
				0& 1
			\end{pmatrix}, \,\quad z \text{ inside } \Gamma_2\\ \\[2pt]
			\wt A(z),\,\quad \text{otherwise}
		\end{cases}
	\end{equation}
 
	The reader may verify that $B(z)$ has no jumps across $\Gamma_1,\Gamma_2$, but it has jumps across $\eta_1,\eta_2$. Specifically, the matrix $B(z)$ solves the following RHP.
	
	\begin{RHP}
		\label{rhp:almost_constant}
	We look for a square matrix function $B(z)$ such that:
	
	\begin{itemize}
		\item $B(z)$ is analytic in $\mathbb{C} \setminus \left( \eta_1\cup \eta_2\right)$.  It possesses smooth boundary values $B_{+}$ and $B_{-}$ as $z$ approaches either contour from the $+$ or $-$ side.
		\item The boundary values satisfy the jump relation $B(z)_+ = B(z)_-J_B(z)$ across the contour $\eta_1,\eta_2$, where

                \begin{equation}
                    J_B(z) = \begin{cases}
                        \begin{pmatrix}
				1 & 0 \\
				2\pi i N h(z)\rho(z) \red{e^{2i\theta(z)}}& 1
			\end{pmatrix}, \quad & z\in \eta_1 \\[4pt] \\
                \begin{pmatrix}
				1 & 2\pi i N h(z)\rho(z)\red{e^{-2i\theta( z)}} \\
				0& 1
			\end{pmatrix}, \quad & z\in \eta_2
                    \end{cases}
                \end{equation}
	\item $B(z)= I + O(1/z)$ as $z\to\infty$.
	\end{itemize}
\end{RHP}
The next classical step in the analysis of the RHP is to modify the previous one to control the $N$ dependence of the jumps, in order to derive the $N \to \infty$ asymptotic behavior of $ \psi_{SG}(x,t;N)$
\cite{Girotti2021,Girotti2023}.   First, we introduce a vertical line segment $\gamma$ which connects the points $-\wo A$ and $-A$, as shown in Figure \ref{fig:K_contour}, and then we introduce two new functions $g(z)$ and $y(z)$, both being analytic in  $ \mathbb{C}\setminus\{\gamma \cup \eta_1 \cup \eta_2\}$.  As usual, these functions achieve their boundary values $g_{\pm}$ and $y_{\pm}$ in the sense of smooth functions on the contours $\gamma, \eta_{1}$ and $\eta_{2}$.  Then, we define the function $C(z)$ as follows
\begin{gather}\label{eq:eqdfnC}
C(z) = e^{(\ln(N) g_\infty + y_\infty) \sigma_3} B(z) e^{-(\ln(N)g(z)+y(z))\sigma_3}\,,
\end{gather}

where  $g_\infty, y_\infty$ are the values of $g(z)$ and $y(z)$ at $\infty$, which we assume are finite (this assumption will be verified later, since these two quantities will be explicitly constructed).  The function $C(z)$ satisfies the following RHP. 
\begin{RHP}
\label{rhp:K_intial}
		We look for a function $C(z)\in \text{Mat}(\C,2)$ such that:
		\begin{itemize}
			\item $C(z)$ is analytic on $\mathbb{C}\setminus\{\eta_1\cup \eta_2\cup \gamma\}$, and achieves its boundary values $C_{\pm}$ on $\eta_{1}$ and $\eta_{2}$ in the sense of smooth functions.
			\item The boundary values satisfy $C_{+}(z) = C_{-}(z) J_{C}(z)$, where the jump matrix $J_{C}(z)$ is given as follows:
\begin{gather*}
J_{C}(z)=  \begin{cases}
\begin{pmatrix}
     e^{\ln(N)\big( g_{-}(z) - g_{+}(z)\big) + y_-(z) - y_+(z)} & 0\\ \\[2pt]
2 \pi i N  h(z) \rho (z)\red{e^{2i\theta(z)}}e^{- \ln (N) \big( g_{-}(z) + g_{+}(z) \big) - (y_-(z) + y_+(z))}    &  e^{ \ln(N)\big( g_{+}(z) -  g_{-}(z) \big) + y_+(z) - y_-(z) }
    \end{pmatrix},  & z \in \eta_1 \\ \\[4pt]
    \begin{pmatrix}
     e^{\ln(N)\big( g_{-}(z) - g_{+}(z)\big) + \left(y_-(z) - y_+(z)\right)} & 2 \pi i N  h(z) \rho (z)\red{e^{-2i\theta( z)}} e^{\ln (N) \big( g_{-}(z) + g_{+}(z) \big) + y_+(z) + y_-(z)} \\ \\[2pt]
0   &  e^{ \ln(N)\big( g_{+}(z) -  g_{-}(z) \big) + y_+(z) - y_-(z)}
    \end{pmatrix},  & z \in \eta_2 \\ \\[4pt]
\begin{pmatrix}
 e^{\ln(N)\big( g_{-}(z) - g_{+}(z)\big) + \left(y_-(z) - y_+(z)\right)}& 0\\
0 & e^{ \ln(N)\big( g_{+}(z) -  g_{-}(z) \big) + y_+(z) - y_-(z)}
    \end{pmatrix},  & z \in \gamma
\end{cases}
\end{gather*}
\item $C(z) = I + \mathcal{O}(1/z)$, as $z \to \infty$.
\end{itemize}
\end{RHP}
\begin{figure}
    \centering

\tikzset{every picture/.style={line width=0.75pt}}        

\begin{tikzpicture}[x=0.5pt,y=0.5pt,yscale=-1,xscale=1]

\draw [color={rgb, 255:red, 208; green, 2; blue, 27 }  ,draw opacity=1 ][line width=1.5]    (125.33,93.11) -- (179.8,93.23) -- (503.8,93.94) -- (580.33,94.11) ;
\draw [shift={(144.27,93.15)}, rotate = 0.13] [fill={rgb, 255:red, 208; green, 2; blue, 27 }  ,fill opacity=1 ][line width=0.08]  [draw opacity=0] (11.61,-5.58) -- (0,0) -- (11.61,5.58) -- cycle    ;
\draw [shift={(333.5,93.57)}, rotate = 0.13] [fill={rgb, 255:red, 208; green, 2; blue, 27 }  ,fill opacity=1 ][line width=0.08]  [draw opacity=0] (11.61,-5.58) -- (0,0) -- (11.61,5.58) -- cycle    ;
\draw [shift={(533.77,94.01)}, rotate = 0.13] [fill={rgb, 255:red, 208; green, 2; blue, 27 }  ,fill opacity=1 ][line width=0.08]  [draw opacity=0] (11.61,-5.58) -- (0,0) -- (11.61,5.58) -- cycle    ;

\draw [color={rgb, 255:red, 208; green, 2; blue, 27 }  ,draw opacity=1 ][line width=1.5]    (125.73,183.51) -- (205.8,183.69) -- (484.2,184.3) -- (580.73,184.51) ;
\draw [shift={(172.57,183.61)}, rotate = 180.13] [fill={rgb, 255:red, 208; green, 2; blue, 27 }  ,fill opacity=1 ][line width=0.08]  [draw opacity=0] (11.61,-5.58) -- (0,0) -- (11.61,5.58) -- cycle    ;
\draw [shift={(351.8,184.01)}, rotate = 180.13] [fill={rgb, 255:red, 208; green, 2; blue, 27 }  ,fill opacity=1 ][line width=0.08]  [draw opacity=0] (11.61,-5.58) -- (0,0) -- (11.61,5.58) -- cycle    ;
\draw [shift={(539.27,184.42)}, rotate = 180.13] [fill={rgb, 255:red, 208; green, 2; blue, 27 }  ,fill opacity=1 ][line width=0.08]  [draw opacity=0] (11.61,-5.58) -- (0,0) -- (11.61,5.58) -- cycle    ;
 
\draw [color={rgb, 255:red, 65; green, 117; blue, 5 }  ,draw opacity=1 ][line width=1.5]    (125.33,93.11) -- (125.73,183.51) ;
\draw [shift={(125.56,145.11)}, rotate = 269.75] [fill={rgb, 255:red, 65; green, 117; blue, 5 }  ,fill opacity=1 ][line width=0.08]  [draw opacity=0] (11.61,-5.58) -- (0,0) -- (11.61,5.58) -- cycle    ;

\draw (405.61,181.91) node [anchor=north west][inner sep=0.75pt]  [rotate=-179.48]  {$+$};

\draw (194.02,183.5) node [anchor=north west][inner sep=0.75pt]  [rotate=-179.48]  {$+$};

\draw (419.28,114.4) node [anchor=north west][inner sep=0.75pt]  [rotate=-179.48]  {$+$};

\draw (253.67,113.08) node [anchor=north west][inner sep=0.75pt]  [rotate=-179.48]  {$+$};

\draw (419.64,92.9) node [anchor=north west][inner sep=0.75pt]  [rotate=-179.48]  {$-$};

\draw (254.34,93.17) node [anchor=north west][inner sep=0.75pt]  [rotate=-179.48]  {$-$};

\draw (194.15,202.92) node [anchor=north west][inner sep=0.75pt]  [rotate=-179.48]  {$-$};

\draw (404.95,203.93) node [anchor=north west][inner sep=0.75pt]  [rotate=-179.48]  {$-$};

\draw (324,89.05) node [anchor=south east] [inner sep=0.75pt]    {$\eta _{1}$};

\draw (326.4,178.78) node [anchor=south east] [inner sep=0.75pt]    {$\eta _{2}$};

\draw (105.33,128.4) node [anchor=north west][inner sep=0.75pt]    {$\gamma $};

\end{tikzpicture}

    \caption{Non-analyticity contour for the function $C(z)$}
    \label{fig:K_contour}
\end{figure}
As mentioned, we introduce the function $g(z)$ to control the $N$ dependence in the jump matrices. We require that $g(z)$ satisfies the following RHP.
\begin{RHP}
\label{RHP:g}
		We look for a scalar function $g(z)$ such that:
		\begin{itemize}
			\item $g(z)$ is analytic on $\mathbb{C}\setminus\{\eta_1\cup \eta_2\cup \gamma\}$
                \item $g(z)$ satisfies the conditions:
\begin{align}
\label{e:eq314a}
&g_{+}(z) + g_{-}(z) = 1, \quad z \in \eta_1 \\ 
\label{e:eq314b}
&g_{+}(z) + g_{-}(z) = - 1, \quad z \in \eta_2\\
\label{e:eq314c}
&g_{+}(z) - g_{-}(z) = c_1, \quad z \in \gamma
\end{align}
\item $g(z) = g_\infty + \mathcal{O}(1/z)$, as $z \to \infty$,
\end{itemize}
where $ c_1,\,g_\infty$ are two constants independent of $z$ still to be determined.
\end{RHP}
Moreover, we wish to make the off-diagonal entries of the jump matrix $J_{C}$  constants (for $z$ on the contours $\eta_{1}$ and $\eta_{2}$). This is accomplished by requiring that the function $y(z)$ satisfies the following RHP.
\begin{RHP}
\label{rhp:y}
We look for a scalar function $y(z)$ such that:
		\begin{itemize}
			\item $y(z)$ is analytic on $\mathbb{C}\setminus\{\eta_1\cup \eta_2\cup \gamma\}$
\item $y(z)$ satisfies the conditions:
\begin{equation}\label{e:eq3.2}
\begin{split}
y_{+}(z) + y_{-}(z) = \ln(2 \pi h(z) \rho(z)) + 2i\theta(z), & \quad z \in \eta_1  \\ 
y_{+}(z) + y_{-}(z)  = - \ln(2 \pi h(z) \rho(z)) + 2i\theta(z), &\quad z \in \eta_2  \\
y_+(z) -y_-(z) = \Delta, &\quad z \in \gamma 
\end{split}
\end{equation}
\item $y(z) = y_\infty + \mathcal{O}(1/z)$, as $z \to \infty$,
\end{itemize}
where $y_\infty$ is a constant still to be determined.
	\end{RHP}
Assuming that such $g(z),y(z)$ exist, then the jump matrix $J_{C}(z)$ reduces into:
\begin{gather}
\label{eq:new_jump_C}
J_{C}(z)=  \begin{cases}
\begin{pmatrix}
    e^{\ln(N)\big( g_{-}(z) - g_{+}(z)\big) + y_-(z) - y_+(z)} & 0\\ \\[2pt]
i & e^{ \ln(N)\big( g_{+}(z) - g_{-}(z) \big) + y_+(z) - y_-(z)}
    \end{pmatrix},  & z \in \eta_1 \\ \\[4pt]
    \begin{pmatrix}
     e^{\ln(N)\big( g_{-}(z) - g_{+}(z)\big) + y_-(z) - y_+(z)} & i \\ \\[2pt]
0   &   e^{\ln(N)\big( g_{+}(z) - g_{-}(z) \big) + y_+(z) - y_-(z) }
    \end{pmatrix},  & z \in \eta_2 \\ \\[4pt]
\begin{pmatrix}
e^{-c_1\ln (N) - \Delta} & 0\\
0 & e^{c_1 \ln (N)  + \Delta}
    \end{pmatrix},  & z \in \gamma
\end{cases}
\end{gather}

\subsection{The \texorpdfstring{$g(z)$}{g(z)} function}

To solve RHP \ref{RHP:g}, we proceed as follows. First, consider the RHP for $g'(z)$. Directly from the RHP for $g(z)$, we get the following. 
\begin{RHP}
   We look for a scalar function $g'(z)$ such that

    \begin{itemize}
			\item $g'(z)$ is analytic on $\mathbb{C}\setminus\{\eta_1\cup \eta_2\}$
                \item $g'(z)$ satisfies the conditions:
\begin{align}
&g'_{+}(z) + g'_{-}(z) = 0, \quad z \in \eta_1 \\ 
&g'_{+}(z) + g'_{-}(z) = 0, \quad z \in \eta_2
\end{align}
\item $g'(z) = \mathcal{O}(z^{-2})$, as $z \to \infty$.
\end{itemize}
\end{RHP}
The previous RHP can be solved explicitly as
\begin{equation}
    g'(z) = \frac{c_0}{ R(z)}, \quad R(z) = (z-A)^{\frac{1}{2}}(z+\wo{A})^{\frac{1}{2}}(z+A)^{\frac{1}{2}}(z-\wo{A})^{\frac{1}{2}}\,,
\end{equation}
for some $c_0 \in \C$, where for each factor $(\cdot )^{1/2}$ we use the principal branch, so that $R(z)$ is analytic in $\mathbb{C} \setminus \{ \eta_1 \cup \eta_2$\}, and is normalized such that 
\[ R(z) \sim z^2\,, \quad z\to\infty \,.\]
Define the function 
\begin{equation}\label{eq:fung}
    g(z) = c_0 \int_A^z \frac{1}{R(s)}\di s + \frac{1}{2}\,,
\end{equation}
then it is easy to check that $g$ satisfies the following conditions:
\begin{align}
&g_{+}(z) + g_{-}(z) = 1, \quad z \in \eta_1 \\ 
&g_{+}(z) + g_{-}(z) =  2c_0 \int_{\gamma} \frac{1}{R(s)}\di s + 1 , \quad z \in \eta_2\\
&g_{+}(z) - g_{-}(z) = 2c_0 \int_{\eta_1} \frac{1}{R_+(s)}\di s, \quad z \in \gamma
\end{align}
Imposing the conditions
\begin{equation}
    c_0= - \left(\int_\gamma \frac{1}{R(s)}\di s\right)^{-1}\, \mbox{and} \quad c_1 = 2c_0 \int_{\eta_1}\frac{1}{R_{+}(s)}\di s \, ,
    \end{equation}
one may verify that $g(z)$ in \eqref{eq:fung} solves RHP \ref{RHP:g}. 

For later usage, we list a series of properties that the function $R$ and its integrals satisfy. Since the proof is technical, we defer it to the appendix \ref{appendix:A}.
\begin{Proposition}
\label{prop:sym_R}
    The function $R$ defined as
    \begin{equation}
    \label{eq:def_R}
        R(z)=(z-A)^{\frac{1}{2}}(z+\wo{A})^{\frac{1}{2}}(z+A)^{\frac{1}{2}}(z-\wo{A})^{\frac{1}{2}}\,,
    \end{equation}
    where for each factor $(\cdot )^{1/2}$ we use the principle branch, satisfies the following properties

\begin{enumerate}
\item $R(z) = z^{2} - \frac{A^{2} + \wo{A}^{2}}{2} + \mathcal{O} \left( \frac{1}{z^{2}}\right)$, as $z \to \infty$,
    \item $R(\wo z) = \wo R(z)$,
    \item $R(- z) = R(z)$,
\end{enumerate}
Furthermore, 
\begin{equation}
    \label{eq:elliptic_integral}
    \int_{\eta_1} \frac{1}{R_+(z)}\di z = -\frac{K(\cos(\theta))}{|A|}\,,\quad \int_{-\wo A}^{-A}\frac{1}{R(s)}\di s=-i\frac{K(\sin(\theta))}{|A|}\,,
\end{equation}
where $ \theta=\arg(A)$, and  $K(k)$ is the complete elliptic integral of the first kind \cite[Chap. 19]{dlmf}.
\end{Proposition}

\begin{remark}
\label{rem:c_tau_rem}
For later usage, we define 
\begin{equation}
    \label{eq:c_tau_def}
    c= \frac{1}{2\int_{\eta_1} \frac{1}{R_+(z)}\di z}=- \frac{|A|}{2K(\cos(\theta))}\,,\quad \tau = \frac{\int_{-\wo A}^{-A} \frac{1}{R(z)}\di z}{\int_{\eta_1} \frac{1}{R_+(z)}\di z}=2i\frac{K(\cos(\theta))}{K(\sin(\theta))}\,,
\end{equation} and we notice that $c\in \R_-$, $\tau\in i\R_+$, and $\frac{2c}{\tau}\in i \R_+$. In this notation

\begin{equation}
    c_0 = - \frac{2c}{\tau}\,, \qquad c_1 = - \frac{2}{\tau}\,,
\end{equation}
therefore, one can rewrite $g(z)$ as

\begin{equation}
    \label{eq:g_def}
    g(z) = -\frac{2c}{\tau}\int_A^z\frac{\di s}{R(s)} + \frac{1}{2}
\end{equation}
Furthermore, as a corollary of Proposition \ref{prop:sym_R} the following holds:
\begin{equation}
\label{eq:IntegralEval}
 \frac{2c}{\tau}\int_{-\wo A}^{- A} \frac{1}{R(s)} \di s = 1, \quad \frac{2c}{\tau}\int_A^{-\wo A} \frac{1}{R_+(s)} \di s = \frac{1}{\tau}, \quad 
    -\frac{2c}{\tau}\int_A^\infty \frac{1}{R(s)} \di s = \frac{1}{2\tau} - \frac{1}{2}\,.
\end{equation}
\end{remark}

\subsection{The \texorpdfstring{$y(z)$}{y(z)} function}

We now find the explicit solution to RHP \ref{rhp:y}.
\begin{Proposition}\label{prop:prop210}
    The solution  $y \,:\, \C\to \C$ to the Riemann-Hilbert Problem \ref{rhp:y} is given by 
\begin{equation}
\label{eq:ydef00}
    y(z) = \frac{R(z)}{2 \pi i} \left( \int_{\eta_1} \frac{\log \left( 2 \pi h(s) \rho(s) \right) \red{+ 2i\theta(s,x,t)}}{R_{+}(s)(s-z)}\di s - \int_{\eta_2} \frac{\log \left( 2 \pi h(s) \rho(s) \right) \red{- 2i\theta(s,x,t)}}{R_{+}(s)(s-z)}\di s + \int_{\gamma} \frac{\Delta(x)}{R(s)(s-z)}\di s\right) \,,
\end{equation}
with $\Delta(x)$
\begin{equation}
\label{eq:Delta_def_new}
    \Delta(x) = -\frac{2c}{\tau} \left( 2\pi x  +2\Re\left( \int_{\eta_1} \frac{\log \,( 2 \pi h(s) \rho(s))}{R_{+}(s)} \di s\right)\right)\,.
\end{equation}
Lastly, $y_\infty = \lim_{z\to\infty}y(z)$ is finite. 
\end{Proposition}
\begin{remark}
    We notice that $\Delta(x)\in i\R$.
\end{remark}
\begin{proof}

By applying the Cauchy-Plemelj-Sokhotski formula we obtain that the function $y(z)$ has the following form

\begin{equation}
    y(z) = \frac{R(z)}{2 \pi i} \left( \int_{\eta_1} \frac{\log \left( 2 \pi h(s) \rho(s) \right) \red{+ 2i\theta(s,x,t)}}{R_{+}(s)(s-z)}\di s - \int_{\eta_2} \frac{\log \left( 2 \pi h(s) \rho(s) \right) \red{- 2i\theta(s,x,t)}}{R_{+}(s)(s-z)}\di s + \int_{\gamma} \frac{\wh \Delta(x,t)}{R(s)(s-z)}\di s\right) \,,
\end{equation}
with $\wh \Delta(x,t)$ defined as
\begin{gather}
\label{eq:delta_t}
\wh \Delta(x,t) = \frac{2c}{\tau} \left( - \int_{\eta_1} \frac{\log \, (2 \pi h(s) \rho(s)) + \red{2i\theta(s,x,t)}}{R_{+}(s)}\di s + \int_{\eta_2} \frac{\log (2 \pi h(s) \rho(s)) -  \red{2i\theta(s,x,t)}  }{R_{+}(s)}\di s \right)\,.
\end{gather}

It is straightforward to verify, using the asymptotic behavior $R(z) = z^{2} \left( 1 + \mathcal{O}\left( \frac{1}{z^{2}}\right) \right)$ as $z \to \infty$, that 
\begin{equation}
\begin{split}
        y(z) = \frac{z^2}{2\pi i} \Big( &-\frac{1}{z}\int_{\eta_1} \frac{\log \left( 2 \pi h(s) \rho(s) \right) + \red{2i\theta(s,x,t)} }{R_{+}(s)} \di s + \frac{1}{z} \int_{\eta_2} \frac{\log \left( 2 \pi h(s) \rho(s) \right) - 2i\theta(x,t,s)}{R_{+}(s)}\di s\\ & - \frac{1}{z} \int_{\gamma} \frac{\wh \Delta(x,t)}{R(s)}\di s + \mathcal{O}(z^{-2})\Big)
\end{split}
\end{equation}
and \eqref{eq:delta_t} guarantees that  $y_{\infty}:=\lim_{z\to\infty}y(z)$ exists. 
To conclude, we must show that $\wh \Delta(x,t) = \Delta(x)$ in \eqref{eq:Delta_def_new}.
First, we notice that $\wh \Delta(x,t)$ is independent of time. Indeed, the coefficient of $t$ in $\wh \Delta(x,t)$ is given by
\begin{equation}
  - 2i \Bigg( \int_{\eta_1} \frac{s^2}{R_{+}(s)}\di s + \int_{\eta_2} \frac{s^2}{R_{+}(s)}\di s \Bigg)
\end{equation}
which can be shown to be zero by residue calculation and by leveraging the symmetry of the endpoints of $\eta_1$. 
Therefore, \eqref{eq:delta_t} reads
\begin{equation}
    \wh\Delta(x,t) = \left( \int_{\gamma} \frac{1}{R(s)}\di s \right)^{-1} \left( - \int_{\eta_1} \frac{\log \, (2 \pi h(s) \rho(s)) + 2ixs}{R_{+}(s)}\di s + \int_{\eta_2} \frac{\log (2 \pi h(s) \rho(s)) -  2ix s }{R_{+}(s)}\di s \right).
\end{equation}
Moreover, by residue calculation, one can show that

\begin{equation}
    -2ix \left( \int_{\eta_1} \frac{s}{R(s)}\di s + \int_{\eta_1} \frac{s}{R(s)}\di s \right) = -2\pi x\,.
\end{equation}
Finally, applying Proposition \ref{prop:sym_R}, we know that $\int_\gamma \frac{1}{R(s)} \di s \in i \R$; so to conclude it remains to show that
\begin{equation}
    \label{eq:equiv_cond}
        \left( - \int_{\eta_1} \frac{\log \,( 2 \pi h(s) \rho(s)) }{R_{+}(s)} \di s + \int_{\eta_2} \frac{\log \left(2 \pi h(s) \rho(s)\right)}{R_{+}(s)} \di s \right) = -2\Re \left( \int_{\eta_1} \frac{\log \,( 2 \pi h(s) \rho(s)) }{R_{+}(s)} \di s \right)\,.
    \end{equation}

    Setting $z=\wt x -ib$, the second integral can be rewritten as

    \begin{equation}
    \begin{split}
        \int_{-a}^a \frac{\log \left(2 \pi h(\wt x-ib) \rho(\wt x-ib)\right)  }{R_+(\wt x-ib)}\di \wt x &  =\wo{ \int_{-a}^a \frac{\log \left(2 \pi h(\wt x+ib) \rho(\wt x+ib)\right) }{R_+(\wt x+ib)}\di \wt x }\\
        & \stackrel{s=\wt x + ib}{=} \wo{\int_{-\wo A}^A\frac{\log \left(2 \pi h(s) \rho(s)\right) }{R_+(s)}\di s}\\ &  = - \wo{\int_{\eta_1}\frac{\log \left(2 \pi h(s) \rho(s)\right)  }{R_+(s)}\di s}\,,
    \end{split}
    \end{equation}
    from which it follows that $\wh\Delta(x,t)\equiv\Delta(x)$, which completes the proof.
\end{proof} 

\section{Opening lenses}
\label{sec:Lenses}

We next introduce two further transformations.  The first involves identifying a contour deformation, so that the jump matrices (on the new contours) have entries that are bounded, but exhibit rapid oscillations.  The second transformation is referred to as ``opening lenses'', whose objective is to deform the Riemann-Hilbert problem so that rapidly oscillatory terms become exponentially decreasing.  As with the transformations from the previous section, these further transformations are invertible, so that knowledge of the solution of this new Riemann-Hilbert problem will be equivalent to knowledge of the original Riemann-Hilbert problem \ref{rhp:K_intial}.

To achieve the first transformation, we aim to find two contours $\eta_3,\eta_4$ such that the analytic extension of $g_+(z) - g_-(z)$ is purely imaginary for $z\in \eta_3,\eta_4$.
The existence of these two contours is a result of the following Lemma.
\begin{lemma}
\label{lem:immaginary_part}
    Let the function $g(z)$ be the unique solution to RHP \ref{RHP:g}. Define the function $u(z)$ for $z\in \C_+\setminus(\eta_1\cup\gamma)$ as  
 \begin{equation}
    \label{eq:udef}
        u(z)=-2g(z)+1= \frac{4c}{\tau} \int_A^z\frac{1}{R(s)}\di s , \ \qquad z\in \mathbb{C}_{+} \setminus\left( \eta_{1} \cup \gamma \right),
    \end{equation}
    where the path of integration is chosen to avoid $\eta_{1}\cup \eta_{2}\cup\gamma $.
        Let $\eta_{3}$ be the arc of the circle of radius $|A|$ centered at $0$ connecting $A$ and $-\overline{A}$.  We have  
      \begin{equation}
         u_{-}(z) = g_{+}(z) - g_{-}(z), \quad z \in \eta_{1}\,, \qquad u(z) \in i\mathbb{R}, \quad z\in \eta_3
    \end{equation}

As with \eqref{eq:udef}, we define
\begin{equation} \label{eq:utdef}
    \tilde{u}(z) = \frac{4c}{\tau} \int_{\wo A}^z\frac{1}{R(s)}\di s , \ \qquad z\in \mathbb{C}_{-} \setminus\left( \eta_{2} \cup \gamma \right),
\end{equation}
where again the path is chosen to avoid $\eta_{1} \cup \eta_{2} \cup \gamma$.
Let $\eta_{4}$ be the arc of the circle of radius $|A|$ centered at $0$ connecting $-A$ and $\overline{A}$.  We have  
    \begin{equation}
         \wt u_{-}(z) = g_{+}(z) - g_{-}(z), \quad z \in \eta_{2}\,, \qquad u(z) \in i\mathbb{R}, \quad z\in \eta_4
    \end{equation}
     Clearly  $\eta_4$ lies below $\eta_2$ and we take it to be oriented from $-A$ to $\wo A$ (like $\eta_2$).

Finally, define $\Omega_{3}^{+}$ to be the open region enclosed by $\eta_1,\eta_3$ and $\Omega_2^+$ the open region enclosed by $\eta_2,\eta_4$. Then we have

\begin{equation}
\label{eq:inequality}
    \begin{cases}
        \Re(u(z))<  0 \quad z \in \Omega_3^+ \\
        \Re(u(z))>0 \quad z \in \mathbb{C}_{+}\setminus\overline{\Omega_3^+}
    \end{cases}\,, \quad \begin{cases}
        \Re(\wt u(z))> 0 \quad z \in \Omega_2^+ \\
        \Re(\wt u(z))<0 \quad z \in \mathbb{C}_{-}\setminus \overline{\Omega_2^+}
    \end{cases}
\end{equation} 
where $\overline{\Omega_{3}^{+}}$ and $\overline{\Omega_{2}^{+}}$ denote the closure of those sets.
\end{lemma}
The contours $\eta_3,\eta_4,\gamma$ are displayed in figure \ref{fig:contour_imaginary}. Since the proof is technical, we defer it to the appendix \ref{appendix:A}.

\begin{remark}
\label{rem:sym_u_ut}
For future reference, note that
    \begin{eqnarray}&&
          \wt u(z) = +
 \frac{2}{\tau} - u(-z)\, , \quad z \in\ \ \fbox{\rule{0pt}{0.8ex}\hspace{0.8ex}} \\
 && \wt u(z) = -
 \frac{2}{\tau} - u(-z)\, , \quad z \not\in\ \ \fbox{\rule{0pt}{0.8ex}\hspace{0.8ex}} \\
    \end{eqnarray}
    where $\fbox{\rule{0pt}{0.8ex}\hspace{0.8ex}}$ is the rectangle formed by the real axis and the points  $\overline{A}$
and $-A$.  \end{remark}
Define the function $\widetilde{y}$ as follows: 
\begin{gather*}
\widetilde{y}(z) = \begin{cases}
-y(z) + \log \big( 2 \pi h(z) \rho (z) \big) + 2 i \theta (z), & \quad z \in \Omega_{3}^{+}\\[2pt]
-y(z) - \log \big( 2 \pi h(z) \rho (z) \big) + 2 i \theta (z), & \quad z \in \Omega_{2}^{+}\\[2pt]
y(z), & \quad z \, \text{elsewhere}
\end{cases}
\end{gather*}
which is analytic in $\mathbb{C} \setminus \left( \eta_{3} \cup \gamma \cup \eta_{4} \right)$ and satisfies the following conditions:
\begin{subequations}\label{eq:eqytilde}
\begin{gather}
\widetilde{y}_{+}(z) + \widetilde{y}_{-}(z) = \log \big( 2 \pi h(z) \rho (z) \big) + 2 i \theta (z), \quad z \in \eta_3\\[2pt]
\widetilde{y}_{+}(z) + \widetilde{y}_{-}(z) = -\log \big( 2 \pi h(z) \rho (z) \big) + 2 i \theta (z), \quad z \in \eta_4\\[2pt]
\widetilde{y}_{+}(z) - \widetilde{y}_{-}(z) = \Delta(x), \quad z \in \gamma
\end{gather}
\end{subequations}
Moreover, define the matrix function $D$ as follows
\begin{equation}
    D(z) = \begin{cases}
    C(z), \quad & z \not\in \Omega_{3}^{+} \cup \Omega_{2}^{+}\\
    C(z) J_C(z), \quad & z \in \Omega_{3}^{+} \cup \Omega_{2}^{+}.
    \end{cases}
\end{equation}
Here the matrix $J_{C}$ in \eqref{eq:new_jump_C}, which was originally defined only on the contours $\eta_{1}$, $\eta_{2}$, and $\gamma$, possesses an analytic extension to the domains $\Omega_{3}^{+}$ and $\Omega_{2}^{+}$, using $\widetilde{y}$, $u$, and $\tilde{u}$ where needed.

Applying Lemma \ref{lem:immaginary_part} and using the definition of $D$, we can check that $D$ satisfies the following RHP.
\begin{RHP}
\label{RHP:immaginary_jump}
    Find a matrix function $D \in \mathrm{Mat}(2,\mathbb{C})$ holomorphic in $\mathbb{C}\setminus\{\eta_3\cup\eta_4\cup\gamma\}$ such that 
 \begin{enumerate}
        \item $D_{+}(z) = D_{-}(z) J_{D}(z)$ 
\begin{gather*}
J_{D}(z)=  \begin{cases}
\begin{pmatrix}
   e^{-\ln(N)u(z)  + \widetilde{y}_{-}(z) - \widetilde{y}_{+}(z)} & 0\\ \\[2pt]
i &  e^{ \ln(N)u(z) + \widetilde{y}_{+}(z) - \widetilde{y}_{-}(z)}
    \end{pmatrix},  & z \in \eta_3 \\ \\[4pt]
    \begin{pmatrix}
     e^{-\ln(N)\tilde{u}(z) +  \widetilde{y}_{-}(z) - \widetilde{y}_{+}(z)} & i\\ \\[2pt]
0   &  e^{\ln(N)\tilde{u}(z)  +  \widetilde{y}_{+}(z) - \widetilde{y}_{-}(z)}
    \end{pmatrix},  & z \in \eta_4 \\ \\[4pt]
\begin{pmatrix}
 e^{\frac{2}{\tau}\ln(N) - \Delta(x)} & 0\\
0 &  e^{-\frac{2}{\tau}\ln(N) +\Delta(x)}
    \end{pmatrix},  & z \in \gamma
\end{cases}
\end{gather*}
\item $D(z) = I + \mathcal{O}(1/z)$, as $z \to \infty$.
\item $u(z)\in i\R$ for $z\in \eta_3$, and $\tilde{u}(z) \in i \R $ for $ z \in \eta_4$.
    \end{enumerate}
\end{RHP}
\begin{figure}
    \centering

\tikzset{every picture/.style={line width=0.75pt}}      

\begin{tikzpicture}[x=0.5pt,y=0.5pt,yscale=-1,xscale=1]

\draw  [dash pattern={on 4.5pt off 4.5pt}]  (132,140.5) -- (469,140) ;

\draw  [dash pattern={on 4.5pt off 4.5pt}]  (131.8,346.5) -- (468.8,346) ;

\draw [color={rgb, 255:red, 65; green, 117; blue, 5 }  ,draw opacity=1 ][line width=2.25]    (132,140.5) -- (131.8,346.5) ;
\draw [shift={(131.89,252.1)}, rotate = 270.06] [fill={rgb, 255:red, 65; green, 117; blue, 5 }  ,fill opacity=1 ][line width=0.08]  [draw opacity=0] (14.29,-6.86) -- (0,0) -- (14.29,6.86) -- cycle    ;

\draw  [draw opacity=0][line width=2.25]  (132,140.5) .. controls (167.16,82.19) and (231.4,43.57) .. (304.31,44.55) .. controls (374.52,45.49) and (435.71,82.92) .. (470.12,138.58) -- (301.65,242.88) -- cycle ; \draw [color={rgb, 255:red, 208; green, 2; blue, 27 }  ,draw opacity=1 ][line width=2.25]    (132,140.5) .. controls (167.16,82.19) and (231.4,43.57) .. (304.31,44.55) .. controls (372.77,45.47) and (432.64,81.07) .. (467.49,134.44) ; \draw [shift={(132,140.5)}, rotate = -57.] [fill={rgb, 255:red, 208; green, 2; blue, 27 }, fill opacity=1][line width=0.08][draw opacity=0] (11.79,-5.66) -- (0,0) -- (11.79,5.66) -- (7.83,0) -- cycle ;

\draw  [draw opacity=0][line width=2.25]  (468.8,346) .. controls (433.87,402.13) and (371.72,439.51) .. (300.82,439.58) .. controls (229.51,439.66) and (166.96,401.96) .. (132.03,345.35) -- (300.62,241.23) -- cycle ; \draw [color={rgb, 255:red, 208; green, 2; blue, 27 }  ,draw opacity=1 ][line width=2.25]    (465.58,351) .. controls (430.14,404.34) and (369.59,439.51) .. (300.82,439.58) .. controls (229.51,439.66) and (166.96,401.96) .. (132.03,345.35) ;  \draw [shift={(468.8,346)}, rotate = 123.89] [fill={rgb, 255:red, 208; green, 2; blue, 27 }  ,fill opacity=1 ][line width=0.08]  [draw opacity=0] (14.29,-6.86) -- (0,0) -- (14.29,6.86) -- cycle    ;

\draw (101.33,218.4) node [anchor=north west][inner sep=0.75pt]  [font=\Large]  {$\gamma $};

\draw (286,147.4) node [anchor=north west][inner sep=0.75pt]  [font=\Large]  {$\eta _{1}$};

\draw (286.67,303.07) node [anchor=north west][inner sep=0.75pt]  [font=\Large]  {$\eta _{2}$};

\draw (289.33,-0.27) node [anchor=north west][inner sep=0.75pt]  [font=\Large]  {$\eta _{3}$};

\draw (304.21,468.82) node [anchor=south] [inner sep=0.75pt]  [font=\Large]  {$\eta _{4}$};

\draw (283.33,361.4) node [anchor=north west][inner sep=0.75pt]  [font=\Large]  {$\Omega _{2}^{+}$};

\draw (282.67,80.73) node [anchor=north west][inner sep=0.75pt]  [font=\Large]  {$\Omega _{3}^{+}$};

\end{tikzpicture}

    \caption{Jump contours for RHP \ref{RHP:immaginary_jump} and for $\wt y(z)$}
    \label{fig:contour_imaginary}
\end{figure}
\subsection{Jump matrix factorization}

On $\eta_{3}$, the jump matrix $J_{D}$ factors as follows
\begin{gather}
 J_{D}(z) = \begin{pmatrix}
e^{-\ln(N)u(z)  + \widetilde{y}_{-}(z) - \widetilde{y}_{+}(z)}& 0\\ \\[2pt]
i & e^{ \ln(N)u(z) +  \widetilde{y}_{+}(z) - \widetilde{y}_{-}(z)}
    \end{pmatrix} \\
    = \begin{pmatrix}
1 & -i e^{-\ln(N) u(z) +  \widetilde{y}_{-}(z) - \widetilde{y}_{+}(z)} \\
0 & 1
    \end{pmatrix} \begin{pmatrix}
    0 & i \\
    i & 0
    \end{pmatrix} \begin{pmatrix}
    1 & -i e^{\ln(N)u(z)  + \widetilde{y}_{+}(z) - \widetilde{y}_{-}(z)}  \\
    0 & 1
    \end{pmatrix}
\end{gather}
and on $\eta_4$ as
\begin{gather*}
 J_{D}(z) = \begin{pmatrix}
e^{-\ln(N)\left( \tilde{u}(z)  \right)  + \widetilde{y}_{-}(z) - \widetilde{y}_{+}(z)}  & i \\ \\[2pt]
0   & e^{\ln(N)\left( \tilde{u}(z) \right)  + \widetilde{y}_{+}(z) - \widetilde{y}_{-}(z)}  
    \end{pmatrix}\\
= \begin{pmatrix}
1 & 0\\
-i e^{\ln(N)\tilde{u}(z)   + \widetilde{y}_{+}(z) - \widetilde{y}_{-}(z)}  & 1
\end{pmatrix} \begin{pmatrix}
0 & i\\
i & 0
\end{pmatrix} \begin{pmatrix}
1 & 0\\
-i e^{-\ln(N) \tilde{u}(z)  + \widetilde{y}_{-}(z) - \widetilde{y}_{+}(z)}  & 1
\end{pmatrix}.
\end{gather*}
Thanks to the conditions that the function $\widetilde{y}$ satisfies on $\eta_3$ and $\eta_4$, we can rewrite the jump for $D$:
\begin{gather*}
J_{D}(z)= \begin{cases}
\begin{pmatrix}
1 & -i e^{-\ln(N) u(z)} \frac{e^{2 \widetilde{y}_{-}(z)}}{2 \pi h(z) \rho (z) e^{2 i \theta(z)}} \\
0 & 1
\end{pmatrix} \begin{pmatrix}
    0 & i\\
    i & 0
\end{pmatrix} \begin{pmatrix}
1 & -i e^{\ln(N) u(z)} \frac{e^{2 \widetilde{y}_{+}(z)}}{2 \pi h(z) \rho (z) e^{2 i \theta(z)}} \\ 
0 & 1
\end{pmatrix}, \quad & z \in \eta_3 \\ \\[4pt]
\begin{pmatrix}
1 & 0 \\
-i e^{\ln(N) \tilde{u}(z)} \frac{e^{-2 \widetilde{y}_{-}(z)}}{2 \pi h(z) \rho (z) e^{-2 i \theta(z)}} & 1
\end{pmatrix}
\begin{pmatrix}
0 & i \\
i & 0
\end{pmatrix}\begin{pmatrix}
1 & 0 \\
-i e^{-\ln(N) \tilde{u}(z)} \frac{e^{-2 \widetilde{y}_{+}(z)}}{2 \pi h(z) \rho (z) e^{-2 i \theta(z)}} & 1
\end{pmatrix}, \quad & z \in \eta_4\\ \\[4pt]
\begin{pmatrix}
e^{\frac{2}{\tau} \ln(N) -\Delta(x) } & 0 \\
0 & e^{-\frac{2}{\tau} \ln(N) + \Delta(x) }
\end{pmatrix}, \quad & z \in \gamma
\end{cases}
\end{gather*}
We define the matrices
\begin{equation}
\label{eq:jump_factor_1}
\begin{split}
&\mathfrak{M}^{-}(z) = \begin{pmatrix}
1 & -i e^{-\ln(N) u(z)} \frac{e^{2 \widetilde{y}(z)}}{2 \pi h(z) \rho(z) e^{2 i \theta (z)}} \\
0 & 1
\end{pmatrix}, \quad \mathfrak{M}^{+}(z) =\begin{pmatrix}
1 & -i e^{\ln(N) u(z)} \frac{e^{2 \widetilde{y}(z)}}{2 \pi h(z) \rho(z) e^{2 i \theta (z)}} \\
0 & 1
\end{pmatrix}\\
& \mathfrak{N}^{-}(z) = \begin{pmatrix}
1 & 0\\
-i e^{-\ln(N) \tilde{u}(z) } \frac{e^{-2 \widetilde{y}(z)}}{2 \pi h(z) \rho(z) e^{-2i \theta(z)}} & 1
\end{pmatrix}, \quad \mathfrak{N}^{+}(z) = \begin{pmatrix}
1 & 0\\
-i e^{\ln(N)\tilde{u}(z) } \frac{e^{-2 \widetilde{y}(z)}}{2 \pi h(z) \rho(z) e^{-2i \theta(z)}} & 1
\end{pmatrix}\,.
\end{split}
\end{equation}
The off-diagonal entries of the matrices $ \mathfrak{M}^{\pm}$ are piecewise analytic functions with jumps across the contours $\eta_{1}$ and $\eta_{3}$.  Notice that because $u$ changes signs across the contour $\eta_{1}$, we have
\begin{equation}
    \left(\mathfrak{M}^{+}\right)_{-}(z) = \left(\mathfrak{M}^{-}\right)_{+}(z), \ z \in \eta_{1}.
\end{equation}

Analogously, the off-diagonal entries of $\mathfrak{N}^{\pm}$ are piecewise analytic, with jumps across the contours $\eta_{2}$ and $\eta_{4}$.  Notice that because $\tilde{u}$ changes signs across $\eta_{2}$, we have
\begin{equation}
    \left(\mathfrak{N}^{-}\right)_{-}(z) = \left(\mathfrak{N}^{+}\right)_{+}(z), \ z \in \eta_{2}.
\end{equation}

 We can rewrite the jump matrix $J_{D}$ for the matrix $D$ as follows.

\begin{gather*}
J_{D}(z) = \begin{cases} \left(\mathfrak{M}^{-}\right)_{-}(z) \begin{pmatrix}
0 & i\\
i & 0
\end{pmatrix} \left( \mathfrak{M}^{+}\right)_{+}(z), \quad & z \in \eta_3\\ \\[2pt]
\left( \mathfrak{N}^{+}\right)_{-}(z) \begin{pmatrix}
0 & i\\
i & 0
\end{pmatrix} \left( \mathfrak{N}^{-}\right)_{+}(z), \quad & z \in \eta_4\\ \\[2pt]
\begin{pmatrix}
e^{\frac{2}{\tau} \ln(N) -\Delta(x) } & 0 \\
0 & e^{-\frac{2}{\tau} \ln(N) + \Delta(x) }
\end{pmatrix}, \quad & z \in \gamma
\end{cases}
\end{gather*}

The next transformation is from $D(z)$ to $E(z)$, the so-called ``opening of lenses'' transformation.  Our aim is to arrive at a new RHP for a function $E(z)$ with constant jumps on the contours $\eta_3,\eta_4,\gamma$ and ``small'' jumps on the other 4 contours $\omega_1,\omega_2,\omega_3,\omega_4$ that we construct according to the following proposition.

\begin{Proposition}
\label{prop:small_norm}
Fix $\min(\Re(A),\Im(A))> 3\varepsilon > 0$, let $\cU_A,\cU_{\wo A},\cU_{-A},\cU_{-\wo A}$ be 4 disks  of radius $\varepsilon$ centered at $A,\wo A, -A, -\wo A$ respectively and define $\cV = \C \setminus \{\cU_A\cup \cU_{\wo A}\cup\cU_{-A}\cup\cU_{-\wo A}\}$. Consider the matrices $\mathfrak{M}^{-},\mathfrak{N}^{+}$ \eqref{eq:jump_factor_1}, then there exist four contours $\omega_1,\omega_2,\omega_3,\omega_4$, a $\kappa>0$ and a $N_0$ such that for all $N>N_0$ 
\begin{equation}
        \mathfrak{N}^{+}(z) = \begin{pmatrix}
            1 & 0 \\
            0 & 1
        \end{pmatrix} + O(N^{-\kappa}),  \quad z\in \left(\omega_{1} \cup\omega_2\right) \cap \cV; \quad \mathfrak{M}^{-}(z) = \begin{pmatrix}
            1 & 0 \\
            0 & 1
        \end{pmatrix} + O(N^{-\kappa}), \quad z\in \left(\omega_3\cup\omega_4 \right)\cap \cV\,.
    \end{equation}
      In particular, the contours $\omega_{1},\omega_2\subset \C_-\setminus\wo \Omega_2^+$ lie on opposing sides of $\eta_{4}$, and $\omega_{3},\omega_4\subset \C_+\setminus\wo \Omega_3^+$ lie on opposing sides of $\eta_{3}$. An example of these contours is shown in Fig. \ref{fig:small_norm}.
\end{Proposition}
\begin{proof}
We focus on $\mathfrak{M}^-$, the proof in the other situation is analogous. First, we notice that in view of the definition \eqref{eq:jump_factor_1}, we must show that there exist two contours $\omega_3,\omega_4$ connecting $A$ with $-\wo A$, a $\kappa>0$ and a $N_0$ such that for all $N>N_0$

\begin{equation}
    -i e^{-\ln(N) u(z)} \frac{e^{2 \widetilde{y}(z)}}{2 \pi h(z) \rho(z) e^{2 i \theta (z)}} = O(N^{-\kappa})\,, \quad z\in \left(\omega_3\cup \omega_4\right)\cap \cV\,.
\end{equation}

Consider the set $\C_+\setminus(\wo \Omega_3^+ \cup \gamma)$. We notice that in these regions the quantities $h(z)$, $\rho(z)$, $\theta(z)$ and $\wt y(z)$ are all analytic and bounded.  Moreover, \eqref{eq:inequality} holds.
(See Assumptions \ref{ass:sym} and Proposition \ref{prop:prop210}, and Lemma \ref{lem:immaginary_part} to verify these properties.) 
Therefore, we can choose $\omega_3,\omega_4\subset \C_+\setminus(\wo \Omega_3^+ \cup \gamma)$ (it is convenient to pick $\omega_{3},\omega_4$ to be analytic arcs).  Finally, since $\Re(u(z))$ is continuous on $(\omega_3\cup\omega_4)\cap \cV$, there exists $\kappa>0$ such that
\[ \min_{z\in\omega_3\cap\cV} \Re(u(z)) = \kappa\,,\]
so there exists $N_0$ such that for all $N>N_0$ 
\begin{equation}
    -i e^{-\ln(N) u(z)} \frac{e^{2 \widetilde{y}(z)}}{2 \pi h(z) \rho(z) e^{2 i \theta (z)}} = O(N^{-\kappa})\,, \quad z\in \omega_3\cap \cV\,,
\end{equation}
proving the result.
\end{proof}
Given these 4 contours, we can define a new function $E(z)$ as \begin{gather}\label{eq:eqdfnE}
E(z) = \begin{cases}
D(z) \, \mathfrak{N}^{+}(z), \quad & z \in \Omega_1\\
D(z) \, \left( \mathfrak{N}^{-}(z)\right)^{-1}, \quad & z \in \Omega_2^{+}\\
D(z) \, \left(\mathfrak{N}^{+}(z)\right)^{-1}, \quad & z \in \Omega_2^{-}\\
D(z) \, \left(\mathfrak{M}^{+}(z)\right)^{-1}, \quad & z \in \Omega_3^{+}\\
D(z) \, \left(\mathfrak{M}^{-}(z)\right)^{-1}, \quad & z \in \Omega_3^{-}\\
D(z) \, \mathfrak{M}^{-}(z), \quad & z \in \Omega_4\\
D(z), \quad & \text{elsewhere}
\end{cases}
\end{gather}
where the sets $\Omega_2^{+},\Omega_2^{-},\Omega_3^{+},\Omega_3^{-}$ are as in figure \ref{fig:small_norm} and $\mathfrak{M}^{+},\mathfrak{M}^{-},\mathfrak{N}^{+},\mathfrak{N}^{-}$ are defined in \eqref{eq:jump_factor_1}. The matrix $ E(z)$ solves the following RHP.
\begin{RHP}
\label{rhp:small_norm}
Find a matrix function $E(z)\in \mathrm{Mat}(2,\mathbb{C})$ analytic in $\mathbb{C}\setminus\{\eta_3\cup\eta_4\cup\gamma\cup \omega_1\cup \omega_2\cup \omega_3\cup \omega_4\}$ such that 
\begin{enumerate}
        \item $E_{+}(z) = E_{-}(z) J_E(z)$ 
\begin{gather*}
J_{E}(z) = \begin{cases}
\begin{pmatrix}
0 & i\\
i & 0
\end{pmatrix}, \quad & z \in \eta_3 \cup \eta_4\\ \\[4pt]
\begin{pmatrix}
e^{\frac{2}{\tau} \ln(N)-\Delta(x)} & 0 \\
0 & e^{-\frac{2}{\tau} \ln(N)+\Delta(x) }
\end{pmatrix}, \quad & z \in \gamma\\ \\[4pt]
\mathfrak{M}^{-}(z), \quad & z \in \omega_3 \cup \omega_4\\ \\[4pt]
\mathfrak{N}^{+}(z), \quad & z \in \omega_1 \cup \omega_2
\end{cases}
\end{gather*}
\item $E(z) = I + \mathcal{O}(1/z)$, as $z \to \infty$.
\end{enumerate}
\end{RHP}

\begin{figure}
    \centering

\tikzset{every picture/.style={line width=0.75pt}} 

\begin{tikzpicture}[x=0.4pt,y=0.4pt,yscale=-1,xscale=1]

\draw  [dash pattern={on 4.5pt off 4.5pt}]  (120.93,240.75) -- (542.33,240.15) ;
 
\draw  [dash pattern={on 4.5pt off 4.5pt}]  (120.68,487.88) -- (542.08,487.28) ;

\draw [color={rgb, 255:red, 65; green, 117; blue, 5 }  ,draw opacity=1 ][line width=2.25]    (120.93,240.75) -- (120.68,487.88) ;
\draw [shift={(120.8,372.92)}, rotate = 270.06] [fill={rgb, 255:red, 65; green, 117; blue, 5 }  ,fill opacity=1 ][line width=0.08]  [draw opacity=0] (14.29,-6.86) -- (0,0) -- (14.29,6.86) -- cycle    ;

\draw  [draw opacity=0][line width=2.25]  (120.93,240.75) .. controls (164.79,170.4) and (245.36,123.75) .. (336.84,124.93) .. controls (426.22,126.09) and (503.9,172.6) .. (546.48,241.37) -- (333.52,362.87) -- cycle ; \draw [color={rgb, 255:red, 208; green, 2; blue, 27 }  ,draw opacity=1 ][line width=2.25]    (124.29,235.52) .. controls (168.83,168.1) and (247.64,123.78) .. (336.84,124.93) .. controls (426.22,126.09) and (503.9,172.6) .. (546.48,241.37) ;  \draw [shift={(120.93,240.75)}, rotate = 303.15] [fill={rgb, 255:red, 208; green, 2; blue, 27 }  ,fill opacity=1 ][line width=0.08]  [draw opacity=0] (14.29,-6.86) -- (0,0) -- (14.29,6.86) -- cycle    ;

\draw  [draw opacity=0][line width=2.25]  (545.56,485.18) .. controls (502.5,555.13) and (423.24,602.07) .. (332.54,602.17) .. controls (242.81,602.26) and (164.19,556.49) .. (120.68,487.88) -- (332.29,364.23) -- cycle ; \draw [color={rgb, 255:red, 208; green, 2; blue, 27 }  ,draw opacity=1 ][line width=2.25]    (542.26,490.39) .. controls (498.53,557.45) and (420.97,602.07) .. (332.54,602.17) .. controls (242.81,602.26) and (164.19,556.49) .. (120.68,487.88) ;  \draw [shift={(545.56,485.18)}, rotate = 123.11] [fill={rgb, 255:red, 208; green, 2; blue, 27 }  ,fill opacity=1 ][line width=0.08]  [draw opacity=0] (14.29,-6.86) -- (0,0) -- (14.29,6.86) -- cycle    ;

\draw [line width=1.5]    (120.93,240.75) .. controls (122,354.5) and (548,332.5) .. (546.48,241.37) ;
\draw [shift={(325.1,318.06)}, rotate = 359.11] [fill={rgb, 255:red, 0; green, 0; blue, 0 }  ][line width=0.08]  [draw opacity=0] (13.4,-6.43) -- (0,0) -- (13.4,6.44) -- (8.9,0) -- cycle    ;

\draw [line width=1.5]    (120.68,487.88) .. controls (122,367.5) and (545.16,412.72) .. (546.23,488.5) ;
\draw [shift={(336.73,414.81)}, rotate = 183.09] [fill={rgb, 255:red, 0; green, 0; blue, 0 }  ][line width=0.08]  [draw opacity=0] (13.4,-6.43) -- (0,0) -- (13.4,6.44) -- (8.9,0) -- cycle    ;

\draw [line width=1.5]    (120.68,487.88) .. controls (121,771) and (554,768) .. (546.23,488.5) ;
\draw [shift={(342.93,699.07)}, rotate = 179.09] [fill={rgb, 255:red, 0; green, 0; blue, 0 }  ][line width=0.08]  [draw opacity=0] (13.4,-6.43) -- (0,0) -- (13.4,6.44) -- (8.9,0) -- cycle    ;

\draw [line width=1.5]    (120.93,240.75) .. controls (122,-42.5) and (545,-19) .. (546.48,241.37) ;
\draw [shift={(324.01,37.12)}, rotate = 0.16] [fill={rgb, 255:red, 0; green, 0; blue, 0 }  ][line width=0.08]  [draw opacity=0] (13.4,-6.43) -- (0,0) -- (13.4,6.44) -- (8.9,0) -- cycle    ;

\draw (84.59,336.62) node [anchor=north west][inner sep=0.75pt]  [font=\Large]  {$\gamma $};

\draw (316.75,251.74) node [anchor=north west][inner sep=0.75pt]  [font=\Large]  {$\eta _{1}$};

\draw (317.59,438.49) node [anchor=north west][inner sep=0.75pt]  [font=\Large]  {$\eta _{2}$};

\draw (309.92,163.59) node [anchor=north west][inner sep=0.75pt]  [font=\Large]  {$\eta _{3}$};

\draw (329.27,590.91) node [anchor=south] [inner sep=0.75pt]  [font=\Large]  {$\eta _{4}$};

\draw (314.8,508.97) node [anchor=north west][inner sep=0.75pt]  [font=\Large]  {$\Omega _{2}^{+}$};

\draw (346.96,195.27) node [anchor=north west][inner sep=0.75pt]  [font=\Large]  {$\Omega _{3}^{+}$};

\draw (206.8,431.97) node [anchor=north west][inner sep=0.75pt]  [font=\Large]  {$\Omega _{2}^{-}$};

\draw (311.8,632.97) node [anchor=north west][inner sep=0.75pt]  [font=\Large]  {$\Omega _{1}$};

\draw (278.8,72.97) node [anchor=north west][inner sep=0.75pt]  [font=\Large]  {$\Omega _{2}$};

\draw (378.96,263.27) node [anchor=north west][inner sep=0.75pt]  [font=\Large]  {$\Omega _{3}^{-}$};

\draw (423,686.4) node [anchor=north west][inner sep=0.75pt]  [font=\Large]  {$\omega _{1}$};

\draw (416,389.4) node [anchor=north west][inner sep=0.75pt]  [font=\Large]  {$\omega _{2}$};

\draw (432,308.4) node [anchor=north west][inner sep=0.75pt]  [font=\Large]  {$\omega _{3}$};

\draw (414,21.4) node [anchor=north west][inner sep=0.75pt]  [font=\Large]  {$\omega _{4}$};

\end{tikzpicture}

    \caption{Contour for RHP \ref{rhp:small_norm} }
    \label{fig:small_norm}
\end{figure}
\section{Global Parametrix}
\label{sec:global}
In this section, we define and solve the model problem that yields the leading order asymptotic behavior for the Riemann-Hilbert problem \ref{rhp:Atil} for $z$ away from the four branch points.  The following RHP comes from the small norm approximation that we have done in the previous section.

\begin{RHP}
    \label{RHP:model_problem}
    Find a matrix function $\Lambda(z)\in \mathrm{Mat}(2,\mathbb{C})$ holomorphic in $\mathbb{C}\setminus\{\eta_3\cup\eta_4\cup\gamma\}$ such that

    \begin{enumerate}
        \item $\Lambda_{+}(z) = \Lambda _{-}(z) J_\Lambda(z)$ 
\begin{equation}
J_{\Lambda}(z)=  \begin{cases}
\begin{pmatrix}
    0 & i \\
    i & 0
\end{pmatrix}\,, & z\in \eta_3\cup\eta_4 \\ \\[4pt]

\begin{pmatrix}
 e^{- \Delta(x)  + \frac{2\ln(N)}{\tau}} & 0\\
0 & e^{\Delta(x) -  \frac{2\ln(N)}{\tau}}
    \end{pmatrix},  & z \in \gamma 
\end{cases}
\end{equation}
\item $\Lambda(z)$ satisfies the asymptotic behavior: $\Lambda(z) = I + \mathcal{O}(1/z)$, as $z \to \infty$
\item $\Lambda$ has $\frac{1}{4}$-root singularities (each entry of $\Lambda$) at $z=A,-A,\wo{A},-\wo{A}$.
\end{enumerate}
\end{RHP}
\begin{remark}
Note that in all previous RHPs, the solution was bounded at each branch point.  Here, local analysis shows that we must have $1/4$-root singularities at each branch point.
\end{remark}
Following the procedure in \cite{Ken_NLS}, we factor the matrix $\Lambda(z)$ as

\begin{equation}
\label{eq:eqLambda}
    \Lambda(z) = \xi(z) e^{i \nu_0 \sigma_3} \Xi(z) e^{i\nu(z)\sigma_3}\,,
\end{equation}
where $\xi,\nu$ are two scalar functions, $\Xi(z)\in\textrm{Mat}(2,\mathbb{C})$, and $\nu_0$ is a constant to be determined. Specifically, $\xi$ is the solution to the following RHP.

\begin{RHP}
    \label{RHP: RHP for ksi} 
    Find a scalar function $\xi(z)$ holomorphic in $\mathbb{C}\setminus\{\eta_3\cup\eta_4\}$ such that
    \begin{enumerate}
        \item $\xi_{+}(z) = i \, \xi_{-}(z), \quad z \in \eta_3\cup\eta_4$
        \item $\xi(z) = 1 + \mathcal{O}(1/z)$, as $z \to \infty$
        \item $\xi$ has $\frac{1}{4}$-root singularities at $z=A,-A$.
    \end{enumerate}
\end{RHP}
Directly by applying the Sokhotski-Plemelj formula we get 
\begin{equation}
\xi (z) = \left( \frac{z +\wo{A}}{z-A} \right)^{1/4} \left( \frac{z -\wo{A}}{z+A} \right)^{1/4}.
\end{equation}

Moreover, the function $\nu(z)$ satisfies the following RHP.

\begin{RHP}
\label{RHP: RHP for nu} Find a scalar function $\nu(z)$ holomorphic in $\mathbb{C}\setminus\{\eta_3\cup\eta_4\cup\gamma\}$ such that
    \begin{enumerate}
        \item $\nu_{+}(z) + \nu_{-}(z) = 0, \quad z \in \eta_3\cup\eta_4$
        \item $\nu_{+}(z) - \nu_{-}(z) = i \left( \Delta(x) - 2\frac{\ln(N)}{\tau}\right), \quad z \in \gamma$.
    \end{enumerate}
  
\end{RHP}
We can solve the previous RHP explicitly. Let $S(z) = (z-A)^{\frac{1}{2}}(z-\wo A)^{\frac{1}{2}}(z+A)^{\frac{1}{2}}(z+\wo A)^{\frac{1}{2}}$ and choose the branch-cuts in such a way that $S$ is analytic in $\mathbb{C} \setminus \{ \eta_3 \cup \eta_4 \}$, and $S(z)\sim z^2$ at infinity. Consider $\nu(z) = S(z)\phi(z)$, where $\phi$ is analytic everywhere except on $\gamma$. Directly from the jump conditions that the functions $S$ and $\nu$ satisfy, we can deduce the jump condition for $\phi$
    \begin{equation}
    \phi_{-}(z) - \phi_{+}(z) = \frac{i \left( \Delta(x) - 2\frac{\ln(N)}{\tau} \right)}{S(z)}, \quad z \in \gamma    
    \end{equation}
which implies that 
\begin{equation}
    \phi(z) = \frac{1}{2 \pi } \int_{\gamma} \frac{1}{S(s)} \frac{\Delta(x) - 2\frac{\ln(N)}{\tau}}{s - z} \, \di s.
\end{equation}
From there, we can deduce the solution to the RHP for $\nu$
\begin{equation}
    \nu(z) = \frac{S(z)}{2 \pi } \int_{\gamma} \frac{1}{S(s)} \frac{\Delta(x) - 2\frac{\ln(N)}{\tau}}{s - z} \, \di s.
\end{equation}
From the last expression we determine the large-$z$ asymptotics for the function $\nu (z)$
    \begin{equation}
    \nu(z) = a_1 z + a_0 + \mathcal{O}(1/z), \quad z \to \infty
    \end{equation}
    \begin{equation}
    \label{eq:def_a1_a0}
a_1 = - \frac{1}{2 \pi } \int_{\gamma} \frac{\Delta(x) - \frac{2\ln(N)}{\tau}}{S(s)} \, \di s = - \frac{\tau\Delta(x)-2\ln(N)}{4c\pi}, \quad a_0 = - \frac{1}{2 \pi } \int_{\gamma} \frac{\Delta(x) - \frac{2\ln(N)}{\tau}}{S(s)} \, s \, \di s.
    \end{equation}

Considering all the above, one can deduce the RHP for the function 
\begin{equation}
\Xi(z) = \frac{1}{\xi (z)} e^{-i \nu_0 \sigma_3} \Lambda(z) e^{-i \nu (z) \sigma_3}.
\end{equation}

\begin{RHP}
\label{RHP: RHP for Xi(z)} Find a square matrix function $\Xi(z)$ analytic in $\mathbb{C}\setminus \{\eta_3 \cup \eta_4\}$ such that
\begin{enumerate}
    \item $\Xi_{+}(z) = \Xi_{-}(z) \begin{pmatrix}
        0 & 1\\
        1 & 0
    \end{pmatrix}, \quad z \in \eta_3 \cup \eta_4$
    \item $\Xi(z) = \left[ I + \mathcal{O}(1/z) \right] e^{-i (a_1 z + a_0 + \nu_0) \sigma_3}, \quad z \to \infty$
    \item $\Xi$ has $\frac{1}{2}$-root singularities (each entry of $\Xi$) at $z=\wo{A},-\wo{A}$.
\end{enumerate}
\end{RHP}
The goal is to find an explicit representation for $\Xi$, which yields an explicit representation for $\Lambda$, since the functions $\xi$ and $\nu$ are known.  Consider now the double-sheeted Riemann surface $\cR$ which consists of two copies of the complex plane, $S_{I}$ and $S_{II}$, which are glued together along the two branch cuts $\eta_3$ and $\eta_4$. Every point $z \in \mathbb{C}$ has an image in $S_{I}$ via some projection map $P_I(z)$, and it has an image in $S_{II}$ via some projection map $P_{II}(z)$. Conversely, any point $p \in S_{I}$ is the image of some point $z \in \mathbb{C}$ through the map $P_{I}(z)$, and there exists another point $p' \in S_{II}$ which is the image of the same point $z$ in $\mathbb{C}$ through $P_{II}(z)$. We introduce a canonical homology basis with the $\alpha$-cycle encircling $\eta_3$ counterclockwise on $S_{I}$, and the $\beta$-cycle going from $\eta_3$ to $\eta_4$ on $S_{I}$ and coming back to $\eta_3$ on $S_{II}$. For more details, see Fig. \ref{fig:riemann_surface}.

\textbf{Notational convention:}  To avoid a flurry of symbols, we will often use $f(z)$ instead of $f(P_{I}(z))$, or $f(P_{II}(z))$, where needed.  Readers beware!
 
\begin{figure}
    \centering

\tikzset{every picture/.style={line width=0.75pt}}     

\begin{tikzpicture}[x=0.6pt,y=0.6pt,yscale=-1,xscale=1]

\draw   (112.4,12.33) -- (344.33,12.33) -- (244.93,160.33) -- (13,160.33) -- cycle ;

\draw   (106.47,222) -- (318.33,222) -- (227.53,367.33) -- (15.67,367.33) -- cycle ;

\draw   (468.47,108.67) -- (680.33,108.67) -- (589.53,254) -- (377.67,254) -- cycle ;
 
\draw    (479,125) -- (379.67,103.67) -- (258.28,74.8) ;
\draw [shift={(256.33,74.33)}, rotate = 13.38] [color={rgb, 255:red, 0; green, 0; blue, 0 }  ][line width=0.75]    (10.93,-3.29) .. controls (6.95,-1.4) and (3.31,-0.3) .. (0,0) .. controls (3.31,0.3) and (6.95,1.4) .. (10.93,3.29)   ;

\draw    (466.33,237) -- (360.33,263) -- (244.94,291.19) ;
\draw [shift={(243,291.67)}, rotate = 346.27] [color={rgb, 255:red, 0; green, 0; blue, 0 }  ][line width=0.75]    (10.93,-3.29) .. controls (6.95,-1.4) and (3.31,-0.3) .. (0,0) .. controls (3.31,0.3) and (6.95,1.4) .. (10.93,3.29)   ;

\draw [color={rgb, 255:red, 74; green, 144; blue, 226 }  ,draw opacity=1 ][line width=1.5]    (167,131.5) -- (192,33) ;
\draw [shift={(177.46,90.29)}, rotate = 284.24] [fill={rgb, 255:red, 74; green, 144; blue, 226 }  ,fill opacity=1 ][line width=0.08]  [draw opacity=0] (11.61,-5.58) -- (0,0) -- (11.61,5.58) -- cycle    ;

\draw    (134.5,52.5) .. controls (174.5,22.5) and (219,34) .. (234.5,52.5) ;

\draw    (124.33,107.5) .. controls (129,138.5) and (220.62,141.83) .. (233,114.5) ;

\draw [color={rgb, 255:red, 74; green, 144; blue, 226 }  ,draw opacity=1 ][line width=1.5]    (150.5,338) -- (175.5,239.5) ;
\draw [shift={(164.67,282.16)}, rotate = 104.24] [fill={rgb, 255:red, 74; green, 144; blue, 226 }  ,fill opacity=1 ][line width=0.08]  [draw opacity=0] (11.61,-5.58) -- (0,0) -- (11.61,5.58) -- cycle    ;

\draw    (118,259) .. controls (158,229) and (202.5,240.5) .. (218,259) ;

\draw    (107.83,314) .. controls (112.5,345) and (204.12,348.33) .. (216.5,321) ;

\draw  [draw opacity=0][line width=1.5]  (194.43,24.23) .. controls (223.89,25.79) and (246.5,34.94) .. (246.5,46) .. controls (246.5,58.15) and (219.19,68) .. (185.5,68) .. controls (151.81,68) and (124.5,58.15) .. (124.5,46) .. controls (124.5,33.85) and (151.81,24) .. (185.5,24) .. controls (188.46,24) and (191.37,24.08) .. (194.21,24.22) -- (185.5,46) -- cycle ; \draw [color={rgb, 255:red, 208; green, 2; blue, 27 }  ,draw opacity=1 ][line width=1.5]    (194.43,24.23) .. controls (223.89,25.79) and (246.5,34.94) .. (246.5,46) .. controls (246.5,58.15) and (219.19,68) .. (185.5,68) .. controls (151.81,68) and (124.5,58.15) .. (124.5,46) .. controls (124.5,33.85) and (151.81,24) .. (185.5,24) .. controls (189.16,24) and (192.75,24.12) .. (196.23,24.34) ; \draw [shift={(191.24,24.1)}, rotate = 0.93] [fill={rgb, 255:red, 208; green, 2; blue, 27 }  ,fill opacity=1 ][line width=0.08]  [draw opacity=0] (11.61,-5.58) -- (0,0) -- (11.61,5.58) -- cycle    ; 
 
\draw    (482.83,141.75) .. controls (522.83,111.75) and (567.33,123.25) .. (582.83,141.75) ;

\draw    (472.67,196.75) .. controls (477.33,227.75) and (568.95,231.08) .. (581.33,203.75) ;

\draw (236.5,55.9) node [anchor=north west][inner sep=0.75pt]    {$A$};

\draw (234.95,117.43) node [anchor=north west][inner sep=0.75pt]  [rotate=-0.76]  {$\overline{A}$};

\draw (132.46,55.87) node [anchor=north east] [inner sep=0.75pt]  [rotate=-0.74]  {$-\overline{A}$};

\draw (122.29,110.37) node [anchor=north east] [inner sep=0.75pt]  [rotate=-0.76]  {$-A$};

\draw (68,15.07) node [anchor=north west][inner sep=0.75pt]    {$S_{I}$};

\draw (60.67,218.07) node [anchor=north west][inner sep=0.75pt]    {$S_{II}$};

\draw (641.33,82.07) node [anchor=north west][inner sep=0.75pt]    {$\mathbb{C}$};

\draw (379.67,100.27) node [anchor=south] [inner sep=0.75pt]    {$P_{I}( z)$};

\draw (360.33,266.4) node [anchor=north] [inner sep=0.75pt]    {$P_{II}( z)$};

\draw (212.88,69.14) node [anchor=north west][inner sep=0.75pt]    {$\alpha $};

\draw (178.91,95.69) node [anchor=north west][inner sep=0.75pt]  [rotate=-0.74]  {$\beta $};

\draw (220,262.4) node [anchor=north west][inner sep=0.75pt]    {$A$};

\draw (218.45,323.93) node [anchor=north west][inner sep=0.75pt]  [rotate=-0.76]  {$\overline{A}$};

\draw (115.96,262.37) node [anchor=north east] [inner sep=0.75pt]  [rotate=-0.74]  {$-\overline{A}$};

\draw (105.79,316.87) node [anchor=north east] [inner sep=0.75pt]  [rotate=-0.76]  {$-A$};

\draw (162.41,302.19) node [anchor=north west][inner sep=0.75pt]  [rotate=-0.74]  {$\beta $};

\draw (584.83,145.15) node [anchor=north west][inner sep=0.75pt]    {$A$};

\draw (583.29,206.68) node [anchor=north west][inner sep=0.75pt]  [rotate=-0.76]  {$\overline{A}$};

\draw (480.79,145.12) node [anchor=north east] [inner sep=0.75pt]  [rotate=-0.74]  {$-\overline{A}$};

\draw (470.62,199.62) node [anchor=north east] [inner sep=0.75pt]  [rotate=-0.76]  {$-A$};

\end{tikzpicture}

    \caption{2-sheeted Riemann Surface $\cR$}
    \label{fig:riemann_surface}
\end{figure}

Let $p$ be a point on $\cR$. We seek a column vector function $W$ on the Riemann surface defined as
\begin{equation}
W(p) = \begin{pmatrix}
w_1 (p)\\
w_2(p)
\end{pmatrix}
\end{equation}
which is used to express the matrix-valued function $\Xi$ as follows:
\begin{equation}\label{e:eqS}
\Xi(z) = \Big( W(P_{I}(z)) \,,\, W(P_{II}(z)) \Big).
\end{equation}
One can check that \eqref{e:eqS} is consistent with the jump condition that $\Xi$ satisfies. Note that since $\Xi$ has $\frac{1}{2}$-root singularities at the points $\wo{A}, -\wo{A}$ in the $z$-plane, the function $W$ has simple poles at the points $\wo{A}, -\wo{A}$ on $\cR$. Moreover, from the large-$z$ asymptotics of the function $\Xi$, we deduce that
\begin{align}
\label{eq:W_conditions}
   & W(p) = \left[ \begin{pmatrix}
        1\\
        0
    \end{pmatrix} + \mathcal{O}(1/p)\right] e^{-i ( a_1 p + a_0 + \nu_0)}, \quad p \to \infty_{1}\\
     &W(p) = \left[ \begin{pmatrix}
        0\\
        1
    \end{pmatrix} + \mathcal{O}(1/p)\right] e^{i ( a_1 p + a_0 + \nu_0)}, \quad p \to \infty_{2} 
\end{align}
where $\infty_{1}$ and $\infty_{2}$ are the images of $\infty$ in $S_{I}$ and $S_{II}$, respectively. 
Following \cite{Ken_NLS}, we solve RHP \ref{RHP: RHP for Xi(z)} in terms of Theta functions.

To this end, we define the function $\rho$ as follows
\begin{equation}
    \rho(p) = \begin{cases}
        S(p), \quad & p \in S_{I}\\
        -S(p), \quad & p \in S_{II}
    \end{cases}
\end{equation}
\\[4pt]
Now, consider the integral $\oint_{\alpha} \frac{c}{\rho(s)} \, \di s$. One can check that for $c$ as in \eqref{eq:c_tau_def}, it follows that $\oint_{\alpha} \frac{c}{\rho(s)} \,\di s = 1$. Furthermore, we can express $\tau$ defined in \eqref{eq:c_tau_def} as an integral on the $\beta$-cycle as 
\begin{equation}
    \tau = \oint_{\beta} \frac{c}{\rho(s)} \, \di s\,,
\end{equation}
and we recall that it is positive imaginary, see Remark \ref{rem:c_tau_rem}.
Next, we fix the asymptotic behavior of $W(p)$ by introducing the function $\delta(p)$ on $\cR$ as
\begin{equation}\label{e:eqdelta}
\delta (p) = \int_{A}^{p} \frac{\wt a_2 s^2 + \wt a_1 s + \wt a_0}{\rho (s)} \, \di s
\end{equation}
where $\wt a_2, \wt a_1, \wt a_0$ have to be determined. Below, we present a proposition that is useful for understanding the analytic properties and the asymptotic behavior of $\delta(p)$.

\begin{Proposition}
\label{prop:asymptotic_for_delta}
Consider the Riemann surface $\cR$ previously defined and let $\delta(p)$ be as in \eqref{e:eqdelta}. Setting
\begin{equation}
\begin{split}
\wt a_0 = -2c a_1  \int_A^{-\wo A} \frac{s^2}{S_+(s)}\di s \,, \quad \wt a_1 = 0\,, \quad \wt a_2 = a_1\,, \\ \nu_0 =\frac{a_1}{2}(1-\tau)\int_A^{-\wo A} \frac{s^2}{S_+(s)}\di s - a_1 A - a_0 + a_1 \int_A^\infty \frac{s^2 - S(s)}{S(s)}\di s\,,
\end{split}
\end{equation}
then $\delta(p)$ satisfies the asymptotic behavior
\begin{equation}
\delta(p) = a_1 p + a_0 + \nu_0 + \mathcal{O} \left( \frac{1}{p} \right), \ \quad p\to\infty_1 
\end{equation}
where $a_1, a_0$ are given in equation \eqref{eq:def_a1_a0}, and $\nu_0$ appears in the asymptotic behavior of $W(p)$ in \eqref{eq:W_conditions}.  Then, $\delta(z)$ is analytic in $\mathbb{C} \setminus\{\eta_3\cup\eta_4\}$ (recalling our slight abuse of notation) and satisfies the following conditions 

\begin{align}
&\delta_{+}(z) + \delta_{-}(z) = 0,  & z \in \eta_3\\
&\delta_{+}(z) + \delta_{-}(z) = \zeta(x,N), & z \in \eta_4
\end{align}

where 

\begin{equation}
    \label{eq:zeta_def_proof}
    \zeta(x,N) = i\left( 4c\pi x + 2\ln(N) + 4c \Re\left( \int_{\eta_1} \frac{\log(2\pi h(s)\rho(s))}{R_+(s)}\di s\right)\right) \in i\R.
\end{equation}

\end{Proposition}

\begin{proof}
    
Let $p \in S_{I}$, then 
\begin{equation}
\rho(p) \equiv S(p) \sim p^2\left( 1 + \mathcal{O}\left(\frac{1}{p^{2}}\right)\right) , \quad p \to \infty_{1}
\end{equation}
Therefore
\begin{align*}
 \frac{\wt a_2 p^2 + \wt a_1 p + \wt a_0}{\rho (p)} &= \frac{\wt a_2 p^2 + \wt a_1 p + \wt a_0}{p^2\left( 1 + \mathcal{O}\left(\frac{1}{p^{2}}\right)\right)}\\
 &= \left( \wt a_2 + \frac{\wt a_1}{p} + \frac{\wt a_0}{p^2} \right)\left( 1 + \mathcal{O}\left(\frac{1}{p^{2}}\right)\right)\\
 &= \wt a_2 + \frac{\wt a_1}{p} +  \mathcal{O}\left(\frac{1}{p^{2}}\right).
\end{align*}
We can then write
\begin{align*}
\int_{A}^{p} \frac{\wt a_2 s^2 + \wt a_1 s + \wt a_0}{\rho (s)} \,\di s &= \int_{A}^{p} \left( \frac{\wt a_2 s^2 + \wt a_1 s + \wt a_0}{\rho (s)} - \wt a_2 - \frac{\wt a_1}{s} \right) \, \di s + \int_{A}^{p} \left( \wt a_2 + \frac{\wt a_1}{s} \right) \, \di s \\
&= \int_{A}^{p} \left( \frac{\wt a_2 s^2 + \wt a_1 s + \wt a_0}{\rho (s)} - \wt a_2 - \frac{\wt a_1}{s} \right) \, \di s  + \left( \wt a_2 p - \wt a_2 A + \wt a_1 \ln(p) - \wt a_1 \ln(A) \right).
\end{align*}
To match the asymptotics of $W(p)$ as $p \to \infty_{1}$ in equation \eqref{eq:W_conditions}, we require 
\begin{equation}
    \wt a_2 = a_1\,, \qquad \wt a_1 =0\,, \qquad \nu_0 = \frac{\tau - 1}{4c}\wt a_0 - a_1 A  - a_0 +  a_1 \int_A^\infty \frac{s^2 - S(s)}{S(s)}\di s\,,
\end{equation}
where we used Lemma \ref{lem:immaginary_part} to compute $\int_A^\infty \frac{1}{S(s)}\di s$. Additionally, if $p \in \gamma$, then
\begin{equation}
\delta_{+}(p) - \delta_{-}(p) = 2 \int_{A}^{- \wo{A}} \frac{a_1 s^2 + \wt a_0}{S_+(s)} \di s
\end{equation}
which we require to be zero, so $\delta(p)$ is analytic in $\gamma$. This fixes $\wt a_0$ as
\begin{equation}
    \wt a_0 = -2c a_1  \int_A^{-\wo A} \frac{s^2}{S_+(s)}\di s\,. 
\end{equation}
Plugging $\wt a_0$ into the previous expression for $\nu_0$, we get
\begin{equation}
    \nu_0 = a_1(1-\tau)\int_A^{-\wo A} \frac{s^2}{S_+(s)}\di s - a_1 A - a_0 + a_1 \int_A^\infty \frac{s^2 - S(s)}{S(s)}\di s\,.
\end{equation}

Finally, a straightforward computation shows that 

\begin{align}
&\delta_{+}(p) + \delta_{-}(p) = 0, & p \in \eta_3\\
&\delta_{+}(p) + \delta_{-}(p) = 2 \int_{-\wo A}^{-A} \frac{a_1 s^2 + \wt a_0}{S(s)} \di s, & p \in \eta_4
\end{align}
where we can rewrite the second jump relation (on $\eta_4$) as follows
    \begin{equation}
    \begin{split}
             \delta_+(p) + \delta_-(p) &= 2 a_1 \left( \int_{-\wo A}^{-A} \frac{s^2}{S(s)}\di s - \tau 
             \int_{\eta_3} \frac{s^2}{S_+(s)}\di s\right) \\ 
             & = - \frac{\tau\Delta(x)-2\ln(N)}{2c\pi}\left( \int_{-\wo A}^{-A} \frac{s^2}{S(s)}\di s - \tau \int_{\eta_3} \frac{s^2}{S_+(s)}\di s\right)\\
             & =- \frac{\tau\Delta(x)-2\ln(N)}{2c\pi}\left( \int_{-\wo A}^{-A} \frac{s^2}{R(s)}\di s - \tau \int_{\eta_1} \frac{s^2}{R_+(s)}\di s\right) \,.
    \end{split}
    \end{equation}
We used Lemma \ref{lem:immaginary_part} to simplify the expression. Applying the Riemann bilinear relation \cite{lawden2013elliptic} one can rewrite the previous jumps as 

\begin{equation}
\begin{split}
    \delta_+(p) + \delta_-(p) &:= - \frac{\tau\Delta(x)-2\ln(N)}{2c\pi}\left( \int_{-\wo A}^{-A} \frac{s^2}{R(s)}\di s - \tau \int_{\eta_1} \frac{s^2}{R_+(s)}\di s\right)\\
    &= i\left( 4c\pi x + 2\ln(N) + 4c \Re\left( \int_{\eta_1} \frac{\log(2\pi h(s)\rho(s))}{R_+(s)}\di s\right)\right)\,=\zeta(x,N)\,,
\end{split}
\end{equation}
from which it follows that $\zeta\in i\R$. This concludes the proof of the proposition.

\end{proof}

\subsection{Solution on the Riemann surface}
Now we can explicitly find $W(p)$ satisfying the conditions \eqref{eq:W_conditions}. To this end, we introduce the Abel map related to $\cR$ as follows  
\begin{equation}
\label{eq:abel_map}
    \mathcal{A}(p) = \int_{A}^{p} \frac{c}{\rho (s)} \, \di s.
\end{equation}
 This is typically multi-valued, and if one wants to create a single-valued function, one may restrict the path of integration to avoid $P_{I}(\gamma) \cup P_{II}(\gamma)\cup P_{I}([\wo{A},\infty]) \cup P_{II}([\wo{A},\infty])$.  

Next, we consider the Riemann Theta function  \cite[21]{dlmf}
\begin{equation}
    \Theta(z;\tau) = \sum_{n \in \mathbb{Z}}  e^{2 \pi i n z + \pi n^2 i \tau}, \quad z \in \mathbb{C}
\end{equation}
which is an even function of $z$ satisfying the following properties
\begin{align}
    &\Theta(z+1;\tau) = \Theta(z;\tau) \\
    &\Theta(z+\tau;\tau) = e^{-\pi i \tau - 2 \pi i z} \Theta(z;\tau)\\
    &\Theta(z + n + m \tau;\tau) = e^{- \pi i m^2 \tau - 2 \pi i m z} \Theta(z;\tau), \quad n, m \in \mathbb{Z}
\end{align}
and has a simple zero at $z=\frac{\tau + 1}{2}$. We notice that the Riemann Theta function is related to the Jacobi $\theta_3$ as $\Theta(z;\tau)=\theta_3(\pi z|\tau)$. We use the $\Theta$-function to construct $W(p)$ for $p \in \cR$, and then through \eqref{e:eqS} we find $\Xi(z)$ solving RHP \ref{RHP: RHP for Xi(z)}.  Recall that the function $W(p)$ has simple poles at the points $\wo{A},-\wo{A}$ and that two entries of $W(p)$ vanish at $\infty_1$ and $\infty_2$. 

Define the functions
\begin{align}
    & w_1(p) = \frac{1}{Z_1}\frac{\Theta \left( \mathcal{A}(p) - \mathcal{A}(\infty_2) + \frac{\tau + 1}{2}; \tau \right)\Theta \left( \mathcal{A}(p) - \mathcal{A}(Q_1) + \frac{\tau + 1}{2}; \tau \right)}{\Theta \left( \mathcal{A}(p) - \mathcal{A}(\wo{A}) + \frac{\tau + 1}{2}; \tau \right)\Theta \left( \mathcal{A}(p) - \mathcal{A}(-\wo{A}) + \frac{\tau + 1}{2}; \tau \right)} e^{-i \delta(p)}\,,\\
    & w_2(p) = \frac{1}{Z_2} \frac{\Theta \left( \mathcal{A}(p) - \mathcal{A}(\infty_1) + \frac{\tau + 1}{2}; \tau \right)\Theta \left( \mathcal{A}(p) - \mathcal{A}(Q_2) + \frac{\tau + 1}{2}; \tau \right)}{\Theta \left( \mathcal{A}(p) - \mathcal{A}(\wo{A}) + \frac{\tau + 1}{2}; \tau \right)\Theta \left( \mathcal{A}(p) - \mathcal{A}(-\wo{A}) + \frac{\tau + 1}{2}; \tau \right)} e^{-i \delta(p)}\,,
\end{align}
where

\begin{equation}
    \label{eq:normalizations}
    \begin{split}
        &Z_1 = \frac{\Theta \left( \mathcal{A}(\infty_1) - \mathcal{A}(\infty_2) + \frac{\tau + 1}{2}; \tau \right)\Theta \left( \mathcal{A}(\infty_1)) - \mathcal{A}(Q_1) + \frac{\tau + 1}{2}; \tau \right)}{\Theta \left( \mathcal{A}(\infty_1)) - \mathcal{A}(\wo{A}) + \frac{\tau + 1}{2}; \tau \right)\Theta \left( \mathcal{A}(\infty_1)) - \mathcal{A}(-\wo{A}) + \frac{\tau + 1}{2}; \tau \right)}\\
        &Z_2 = \frac{\Theta \left( \mathcal{A}(\infty_2) - \mathcal{A}(\infty_1) + \frac{\tau + 1}{2}; \tau \right)\Theta \left( \mathcal{A}(\infty_2) - \mathcal{A}(Q_2) + \frac{\tau + 1}{2}; \tau \right)}{\Theta \left( \mathcal{A}(\infty_2) - \mathcal{A}(\wo{A}) + \frac{\tau + 1}{2}; \tau \right)\Theta \left( \mathcal{A}(\infty_2) - \mathcal{A}(-\wo{A}) + \frac{\tau + 1}{2}; \tau \right)}\,.
    \end{split}
\end{equation}
We notice that defining $W(p) = (w_1(p), w_2(p))^\intercal$, it has the right asymptotic behavior for $p \to \infty$, but, in general, it is not single-valued on the Riemann surface.  Indeed, with the notational convention that $p + \beta$ means adding a $\beta$-cycle, we have

\begin{equation}
\label{eq:sol_rs}
\begin{split}
    & w_1(p+\beta) = \exp\left( 2\pi i \left( \cA(\infty_2) + \cA(Q_1) - \cA(\wo A) - \cA(-\wo A) -  \frac{\zeta(x,N)}{2\pi}  \right) \right)w_1(p)\,, \\
    & w_2(p+\beta) = \exp\left( 2\pi i \left( \cA(\infty_1) + \cA(Q_2) - \cA(\wo A) - \cA(-\wo A) -  \frac{\zeta(x,N)}{2\pi}  \right) \right)w_1(p)\,.
    \end{split}
\end{equation}
Therefore, to have $W(p)$ single valued on the Riemann surface, we must set
\begin{equation}
    \begin{split}
        &\cA(\infty_2) + \cA(Q_1) - \cA(\wo A) - \cA(-\wo A) - \frac{\zeta(x,N)}{2\pi}  = 0 \,,\\
        & \cA(\infty_1) + \cA(Q_2) - \cA(\wo A) - \cA(-\wo A) - \frac{\zeta(x,N)}{2\pi} = 0\,,
    \end{split}
\end{equation}
which implies that

\begin{align}
    & \cA(Q_1) = \cA(\wo A) + \cA(-\wo A) + \frac{\zeta(x,N)}{2\pi} - \cA(\infty_2)\,, \\
    & \cA(Q_2) = \cA(\wo A) + \cA(-\wo A) +  \frac{\zeta(x,N)}{2\pi} - \cA(\infty_1)\,.
\end{align}
Note that this determines the zeros $Q_1, Q_2$ of $w_1(p;\tau)$ and $w_2(p;\tau)$, respectively, although the path of integration from $A$ to $Q_{1}$ or $Q_{2}$ will necessarily involve many traverses of the cycle $\beta$, since $\zeta(x,N)$ is growing in $N$. We can further simplify the expression for $W(p)$ applying the following proposition, which is a Corollary of \eqref{eq:IntegralEval}.

\begin{Proposition}
\label{prop:abel_map_values}
    Consider the Riemann surface $\cR$, and define the Abel map $\cA(p)$ \eqref{eq:abel_map}; then

    \begin{equation}
        \cA(\wo A) = \frac{\tau}{2}\,, \quad \cA(-\wo A)= \frac{1}{2}\,,\quad \cA(\infty_1)= - \cA(\infty_2) =  \frac{\tau - 1}{4}\,.
    \end{equation}
\end{Proposition}
Therefore, we can rewrite the solution $W(p)$ as

\begin{equation}
    W(p) = \begin{pmatrix}
        w_1(p) \\
        w_2(p)
    \end{pmatrix} = \begin{pmatrix}
        \frac{1}{Z_1}\frac{\Theta \left( \mathcal{A}(p) + \frac{1+3\tau}{4} ; \tau \right)\Theta \left( \mathcal{A}(p) -  \frac{\zeta(x,N)}{2\pi} + \frac{1-\tau}{4}; \tau \right)}{\Theta \left( \mathcal{A}(p) +  \frac{ 1}{2}; \tau \right)\Theta \left( \mathcal{A}(p) + \frac{\tau}{2}; \tau \right)} e^{-i \delta(p)}\\
        \frac{1}{Z_2} \frac{\Theta \left( \mathcal{A}(p) + \frac{3+\tau}{4}; \tau \right)\Theta \left( \mathcal{A}(p) -  \frac{\zeta(x,N)}{2\pi} + \frac{\tau - 1}{4}; \tau \right)}{\Theta \left( \mathcal{A}(p) + \frac{ 1}{2}; \tau \right)\Theta \left( \mathcal{A}(p)  + \frac{\tau}{2}; \tau \right)} e^{-i \delta(p)}
    \end{pmatrix}\,.
\end{equation}
 \subsection{Solution on the Complex plane}
 We now have to project the solution from the Riemann surface $\cR$ to the complex plane. Consider the two projections $P_I(z),P_{II}(z)$ defined from the complex plane to the first and the second sheet of $\cR$ respectively, see Fig. \ref{fig:riemann_surface}. Then, one can rewrite the solution of the RHP \ref{RHP: RHP for Xi(z)} as

\begin{equation}
\label{eq:sol_complex_plane}
    \Xi(z) = \begin{pmatrix}
        w_1(P_I(z)) & w_1(P_{II}(z)) \\
        w_2(P_I(z)) & w_2(P_{II}(z)) \\
    \end{pmatrix}=: \begin{pmatrix}
        s_{11}(z;x,t) & s_{12}(z;x,t) \\
        s_{21}(z;x,t) & s_{22}(z;x,t)
    \end{pmatrix}\,,
\end{equation}
from the definition of $w_1(p),w_2(p)$ \eqref{eq:sol_rs} it is easy to see that the previously defined function solves RHP \ref{RHP: RHP for Xi(z)}. 

For future use, we write the previous solution explicitly,  making use of our previously forewarned notational abuse in which  $\cA(z)$ and $\delta(z)$ are considered to be evaluated on the complex plane, or, equivalently, on the first Riemann sheet.

\begin{equation}
\label{eq:solution_ext}
\begin{split}
        &s_{11}(z;x,t) = \frac{1}{Z_1}\frac{\Theta \left( \mathcal{A}(z) + \frac{1+3\tau}{4} ; \tau \right)\Theta \left( \mathcal{A}(z) -  \frac{\zeta(x,N)}{2\pi} + \frac{1-\tau}{4}; \tau \right)}{\Theta \left( \mathcal{A}(z) +  \frac{ 1}{2}; \tau \right)\Theta \left( \mathcal{A}(z) + \frac{\tau}{2}; \tau \right)} e^{-i \delta(z)} \\
    &s_{12}(z;x,t) = \frac{1}{Z_1}\frac{\Theta \left( -\mathcal{A}(z) + \frac{1+3\tau}{4} ; \tau \right)\Theta \left( -\mathcal{A}(z) -  \frac{\zeta(x,N)}{2\pi} + \frac{1-\tau}{4}; \tau \right)}{\Theta \left( -\mathcal{A}(z) +  \frac{ 1}{2}; \tau \right)\Theta \left( -\mathcal{A}(z) + \frac{\tau}{2}; \tau \right)} e^{i \delta(z)} \\
    & s_{21}(z;x,t) = \frac{1}{Z_2} \frac{\Theta \left( \mathcal{A}(z) + \frac{3+\tau}{4}; \tau \right)\Theta \left( \mathcal{A}(z) -  \frac{\zeta(x,N)}{2\pi} + \frac{\tau - 1}{4}; \tau \right)}{\Theta \left( \mathcal{A}(z) + \frac{ 1}{2}; \tau \right)\Theta \left( \mathcal{A}(z)  + \frac{\tau}{2}; \tau \right)} e^{-i \delta(z)}\\
  & s_{22}(z;x,t) = \frac{1}{Z_2} \frac{\Theta \left( -\mathcal{A}(z) + \frac{3+\tau}{4}; \tau \right)\Theta \left( -\mathcal{A}(z) -  \frac{\zeta(x,N)}{2\pi} + \frac{\tau - 1}{4}; \tau \right)}{\Theta \left( -\mathcal{A}(z) + \frac{ 1}{2}; \tau \right)\Theta \left( -\mathcal{A}(z)  + \frac{\tau}{2}; \tau \right)} e^{i \delta(z)}.
\end{split}
\end{equation}

\section{Local parametrices}
\label{sec:local}

In the previous section, we developed a solution to RHP \ref{RHP:model_problem} which serves as an outer approximation to the solution of RHP \ref{rhp:small_norm} when we are not too close to the four branching points $A,\wo{A},-A,-\wo A$. However, in the vicinity of these points, this approximation is not expected to be close to the solution of the full Riemann-Hilbert problem (since the jump matrices are not close to $I$ there). Therefore, we must construct a solution in a small neighborhood around each of these four points to better approximate the solution to RHP \ref{rhp:small_norm} locally. The global solution to RHP \ref{rhp:small_norm} will then be given in terms of the error function that describes the difference between the global solution and the approximate solutions, inside and outside the four neighborhoods, which as we will see in this section, exists since it satisfies a small norm RHP.

In this section, we construct the local parametrices following \cite{Kuijlaars-2004}. We chose to present the explicit construction of the parametrix near $z=-A$ as it is the most delicate case and requires additional techniques. For the points $z= \wo A$ and $z=A$, we verify that the proposed constructions satisfy the required jump conditions. The local parametrix near $z=-\wo A$ follows analogously and has therefore been omitted.

\subsection{Local parametrix near \texorpdfstring{$z= - A$}{-A}}

Recall from proposition \ref{prop:small_norm} that $\mathcal{U}_{-A}$ is a small neighborhood of $z=-A$. Let us denote by $E_{-A}(z)$ a local parametrix near $z=-A$: $E_{-A}(z)$ should satisfy all the Riemann-Hilbert jump relations satisfied by $E(z)$ inside the neighborhood $\mathcal{U}_{-A}$, and it should also behave on the boundary of $\mathcal{U}_{-A}$ so as to asymptotically match the global parametrix.  

Inside $\mathcal{U}_{-A}$, we seek $E_{-A}(z)$ satisfying the Riemann-Hilbert jump relations:
\begin{gather}\label{eq:eqEminA}
E_{-A +}(z) = E_{-A -}(z) J_{E_{-A}}(z), \quad J_{E_{-A}}(z) = \begin{cases}
\begin{pmatrix}
0 & i\\
i & 0
\end{pmatrix}, & \quad z \in \eta_4 \\\\[2pt]
e^{\left( -\Delta(x) + \frac{2}{\tau} \ln(N) \right) \sigma_3}, & \quad z \in \gamma\\\\[2pt]
\mathfrak{N}^{+}(z), & \quad z \in \omega_1 \cup \omega_2
\end{cases}
\end{gather}

Following \cite{Kuijlaars-2004}, we construct $E_{-A}$  in terms of special functions. For this purpose, we define the function $\Psi(\xi)$ as follows:
\begin{gather}\label{eq:eq233}
\Psi(z) = \begin{cases}
\begin{pmatrix}
I_{0}(2 \xi^{1/2}) & \frac{i}{\pi} K_{0}(2 \xi^{1/2}) \\\\[1pt]
2 \pi i \xi^{1/2} I_{0}'(2 \xi^{1/2}) & -2 \xi^{1/2} K_{0}'(2 \xi^{1/2})
\end{pmatrix}, \quad & \arg(\xi) \in (-2\pi/3,2\pi/3) \\\\[1pt]
\begin{pmatrix}
\frac{1}{2} H_{0}^{(1)}(2(-\xi)^{1/2}) & \frac{1}{2} H_{0}^{(2)}(2(-\xi)^{1/2}) \\\\[1pt]
\pi \xi^{1/2} (H_{0}^{(1)})'(2(-\xi)^{1/2}) & \pi \xi^{1/2} (H_{0}^{(2)})'(2(-\xi)^{1/2})
\end{pmatrix}, \quad & \arg(\xi) \in (2\pi/3,\pi) \\\\[1pt]
\begin{pmatrix}
\frac{1}{2} H_{0}^{(2)}(2(-\xi)^{1/2}) & -\frac{1}{2} H_{0}^{(1)}(2(-\xi)^{1/2}) \\\\[1pt]
-\pi \xi^{1/2} (H_{0}^{(2)})'(2(-\xi)^{1/2}) & \pi \xi^{1/2} (H_{0}^{(1)})'(2(-\xi)^{1/2})
\end{pmatrix}, \quad & \arg(\xi) \in (-\pi,-2\pi/3)
\end{cases}
\end{gather}
    where $H_{0}^{(1)}$ and $H_{0}^{(2)}$ are the Hankel functions and $K_{0}$ and $I_{0}$ are the modified Bessel functions, see \cite[Chapter 10]{dlmf}. In \cite{Kuijlaars-2004}, it is shown that $\Psi$ satisfies the jumps given in Fig. \ref{fig Contour Psi(xi)}. For brevity, we introduce the notation $\mathcal{Y}(z) = 2 \pi h(z) \rho(z) e^{-2i \theta(z)}$. 
\begin{figure}[ht]
    \centering
    \begin{tikzpicture}[
        myarrow/.style={
            thick,
            postaction={decorate,
            decoration={markings,mark=at position #1 with {\arrow[scale=1.5]{stealth}}}}
        },
        every matrix/.style={scale=0.4} 
    ]
        \draw[myarrow=0.5]  (-2.5,-2.5) -- (0,0);
        \draw[myarrow=0.5]  (-3.5,0) -- (0,0);
        \draw[myarrow=0.5]  (-2.5,2.5) -- (0,0);

        \matrix (m1) at (-0.5,2) [matrix of math nodes, left delimiter=(, right delimiter=), inner sep=0pt, column sep=10pt] {
            1 & 0\\ \\[2pt]
            1 & 1 \\
        };

        \matrix (m2) at (-2.5,0.5) [matrix of math nodes, left delimiter=(, right delimiter=), inner sep=0pt, column sep=10pt] {
            0 & 1\\ \\[2pt]
            -1 & 0 \\
        };
         \matrix (m3) at (-0.5,-2) [matrix of math nodes, left delimiter=(, right delimiter=), inner sep=0pt, column sep=10pt] {
            1 & 0\\ \\[2pt]
            1 & 1 \\
        };        
        \node (eq) at (0,-0.3) {$0$};
    \end{tikzpicture}
    \caption{Jumps for $\Psi(\xi)$}
    \label{fig Contour Psi(xi)}
\end{figure}

We will arrive at precisely this collection of jump relations through explicit transformations, starting with $E_{-A}(z)$.

We introduce the function $m$:
\begin{gather}\label{eq:eq217}
m(z) = \begin{cases}
E_{-A}(z) \, e^{\left( \widetilde{y}(z) - \frac{1}{\tau} \ln(N) \right) \sigma_3}, \quad & z \in \Omega\\\\[2pt]
E_{-A}(z) \, e^{\left( \widetilde{y}(z) + \frac{1}{\tau} \ln(N) \right) \sigma_3}, \quad & z \in \widetilde{\Omega}.
\end{cases}
\end{gather}

\begin{figure}
    \centering

\tikzset{every picture/.style={line width=0.75pt}}       

\begin{tikzpicture}[x=0.5pt,y=0.5pt,yscale=-1,xscale=1]

\draw [color={rgb, 255:red, 208; green, 2; blue, 27 }  ,draw opacity=1 ][line width=2.25]    (208.6,162.69) .. controls (223.25,166.25) and (269,174.4) .. (299.4,181.6) .. controls (329.8,188.8) and (351,189.75) .. (402.04,190.48) .. controls (453.07,191.21) and (482.5,188.3) .. (513,180.8) ;
\draw [shift={(262.77,173.71)}, rotate = 191.37] [fill={rgb, 255:red, 208; green, 2; blue, 27 }  ,fill opacity=1 ][line width=0.08]  [draw opacity=0] (14.29,-6.86) -- (0,0) -- (14.29,6.86) -- cycle    ;
\draw [shift={(359.32,189.45)}, rotate = 182.82] [fill={rgb, 255:red, 208; green, 2; blue, 27 }  ,fill opacity=1 ][line width=0.08]  [draw opacity=0] (14.29,-6.86) -- (0,0) -- (14.29,6.86) -- cycle    ;
\draw [shift={(466.79,188.58)}, rotate = 175.08] [fill={rgb, 255:red, 208; green, 2; blue, 27 }  ,fill opacity=1 ][line width=0.08]  [draw opacity=0] (14.29,-6.86) -- (0,0) -- (14.29,6.86) -- cycle    ;

\draw [color={rgb, 255:red, 65; green, 117; blue, 5 }  ,draw opacity=1 ][line width=2.25]    (207.98,53.54) -- (208.29,108.76) -- (208.6,162.69) ;
\draw [shift={(208.18,89.75)}, rotate = 269.67] [fill={rgb, 255:red, 65; green, 117; blue, 5 }  ,fill opacity=1 ][line width=0.08]  [draw opacity=0] (14.29,-6.86) -- (0,0) -- (14.29,6.86) -- cycle    ;
\draw [shift={(208.5,144.32)}, rotate = 269.67] [fill={rgb, 255:red, 65; green, 117; blue, 5 }  ,fill opacity=1 ][line width=0.08]  [draw opacity=0] (14.29,-6.86) -- (0,0) -- (14.29,6.86) -- cycle    ;

\draw [line width=1.5]    (208.6,162.69) .. controls (224.06,151.1) and (292.75,126.25) .. (313.25,129.25) .. controls (333.75,132.25) and (474.5,139.75) .. (487.75,137.25) ;
\draw [shift={(266.19,138.17)}, rotate = 162.06] [fill={rgb, 255:red, 0; green, 0; blue, 0 }  ][line width=0.08]  [draw opacity=0] (11.61,-5.58) -- (0,0) -- (11.61,5.58) -- cycle    ;
\draw [shift={(407.94,135.54)}, rotate = 182.79] [fill={rgb, 255:red, 0; green, 0; blue, 0 }  ][line width=0.08]  [draw opacity=0] (11.61,-5.58) -- (0,0) -- (11.61,5.58) -- cycle    ;

\draw [line width=1.5]    (208.6,162.69) .. controls (230.25,184.75) and (230.75,186.75) .. (268.25,222.75) .. controls (305.75,258.75) and (456.5,268.25) .. (515.5,232.25) ;
\draw [shift={(242.92,198.07)}, rotate = 225.09] [fill={rgb, 255:red, 0; green, 0; blue, 0 }  ][line width=0.08]  [draw opacity=0] (11.61,-5.58) -- (0,0) -- (11.61,5.58) -- cycle    ;
\draw [shift={(398.79,254.94)}, rotate = 181.05] [fill={rgb, 255:red, 0; green, 0; blue, 0 }  ][line width=0.08]  [draw opacity=0] (11.61,-5.58) -- (0,0) -- (11.61,5.58) -- cycle    ;

\draw   (117.04,162.69) .. controls (117.04,119.54) and (158.04,84.56) .. (208.6,84.56) .. controls (259.17,84.56) and (300.16,119.54) .. (300.16,162.69) .. controls (300.16,205.84) and (259.17,240.82) .. (208.6,240.82) .. controls (158.04,240.82) and (117.04,205.84) .. (117.04,162.69) -- cycle ;
 
\draw  [fill={rgb, 255:red, 248; green, 231; blue, 28 }  ,fill opacity=0.55 ] (297.23,181.25) .. controls (287.23,215.15) and (251.27,240.25) .. (208.44,240.25) .. controls (157.87,240.25) and (116.88,205.27) .. (116.88,162.12) .. controls (116.88,119.06) and (157.7,84.14) .. (208.11,83.99) -- (208.44,162.12) -- cycle ;

\draw  [fill={rgb, 255:red, 189; green, 16; blue, 224 }  ,fill opacity=0.48 ] (208.6,84.56) .. controls (259.1,84.63) and (300,119.33) .. (300,162.12) .. controls (300,168.79) and (299,175.27) .. (297.13,181.45) -- (208.44,162.12) -- cycle ;

\draw (219.9,49.18) node [anchor=north west][inner sep=0.75pt]  [font=\large]  {$\gamma $};

\draw (305.57,156.92) node [anchor=north west][inner sep=0.75pt]  [font=\large]  {$\eta _{4}$};

\draw (378.7,139.9) node [anchor=north west][inner sep=0.75pt]    {$\omega _{2}$};

\draw (345.2,229.4) node [anchor=north west][inner sep=0.75pt]    {$\omega _{1}$};

\draw (235,111.4) node [anchor=north west][inner sep=0.75pt]    {$\Omega $};

\draw (169.5,163.9) node [anchor=north west][inner sep=0.75pt]    {$\tilde{\Omega }$};

\end{tikzpicture}

    \caption{Contour for RHP \ref{eq:eqEminA} }
    \label{fig:localpar}
\end{figure}
Using the jump of $E_{-A}$ and $\widetilde{y}$ across $\gamma$, we can check that $m$ is analytic across $\gamma$. Moreover, it satisfies the following jump relations across $\eta_4 \cup \omega_1 \cup \omega_2$
\begin{gather}
m_{+}(z) = m_{-}(z) J_{m}(z), 
\end{gather}
where the jump matrix $J_{m}(z)$ can be factored as follows
\begin{gather}\label{eq:eq67n}
J_{m}(z) =  \Bigg( \frac{1}{\mathcal{Y}(z)} \Bigg)^{-\sigma_3/2} \left\{\begin{array}{lr}
\begin{pmatrix}
0 & i\\
i & 0
\end{pmatrix}, \quad & z \in \eta_4\\\\[2pt]
\begin{pmatrix}
1 & 0\\
-i e^{\ln(N) \frac{4c}{\tau} \int_{-A}^{z} \frac{\di s}{S(s)}} & 1
\end{pmatrix}, \quad & z \in \omega_1 \cup \omega_2
\end{array}\right\} \Bigg( \frac{1}{\mathcal{Y}(z)} \Bigg)^{\sigma_3/2}.
\end{gather}
Note that, in relation \eqref{eq:eq67n}, we have replaced $\int_{-A}^{z} \frac{\di s}{R(s)}$ with $\int_{-A}^{z} \frac{\di s}{S(s)}$ since the two agree on the contours $\omega_1$ and  $\omega_2$. It is more convenient to use the function $S(s)$ instead of $R(s)$, because it is branched along $\eta_3$ and $\eta_4$, which are the relevant contours of non-analyticity in this section.

Next we define the function
\begin{gather}
P(z) = m(z) \, \Bigg( \frac{1}{\mathcal{Y}(z)} \Bigg)^{-\sigma_3/2}\,,
\end{gather}
which satisfies the jump relations
\begin{gather*}
P_{+}(z) = P_{-}(z) \begin{cases}
\begin{pmatrix}
0 & i\\
i & 0
\end{pmatrix}, \quad & z \in \eta_4\\\\[2pt]
\begin{pmatrix}
1 & 0\\
-i e^{\ln(N) \frac{4c}{\tau} \int_{-A}^{z} \frac{\di s}{S(s)}} & 1
\end{pmatrix}, \quad & z \in \omega_1 \cup \omega_2.
\end{cases}
\end{gather*}
Analysis of the local behavior of the function $\frac{4c}{\tau}\int_{-A}^{z} \frac{\di s}{S(s)}$ near $z=-A$ , prompts us to introduce the conformal mapping $\xi = \varphi_{-A}(z)$, with the analytic function $\varphi_{-A}$ defined implicitly via the relation
\begin{gather}\label{eq:eq67new}
\left( \varphi_{-A}(z) \right)^{1/2} \equiv  - \ln(N) \frac{c}{\tau} \int_{-A}^{z} \frac{\di s}{S(s)}\,.
\end{gather}
This function maps the point $z=-A$ to $\xi = 0$, and the contour $\eta_{4}$ emanating from $-A$ is mapped to $\mathbb{R}_{-}$.  We choose the contours $\omega_{1}$ and $\omega_{2}$ (within $\mathcal{U}_{-A}$) so that in the $\xi$ plane, they map to precisely the rays $\mbox{arg}(\xi) = \pm 2 \pi /3 $ as shown in Figure \ref{fig Contour Psi(xi)}.  

Now we consider the jumps in the $\xi$ plane satisfied by $P(\varphi_{-A}^{-1}(\xi))$:

\begin{gather}\label{eq:eq610n}
\left(P (  \varphi_{-A}^{-1}(\xi)) \right)_{+} = \left(P (  \varphi_{-A}^{-1}(\xi)) \right)_{-}
 \begin{cases}
\begin{pmatrix}
0 & -i\\
-i & 0
\end{pmatrix}, \quad & \xi \in (-\infty,0)\\ \\[4pt]
\begin{pmatrix}
1 & 0\\
i e^{- 4 \xi^{1/2}} & 1
\end{pmatrix}, \quad & \mbox{Arg}(\xi) = \pm 2 \pi / 3
\end{cases}
\end{gather}
It is then straightforward to verify that the quantity
\begin{eqnarray}
    P(\varphi_{-A}^{-1}(\xi)) e^{2 \xi^{1/2} \, \sigma_3} e^{ \frac{- i \pi \sigma_{3}}{4}}
\end{eqnarray}
satisfies the jump relationships
\begin{gather}
\left(
P(\varphi_{-A}^{-1}(\xi)) e^{2 \xi^{1/2} \, \sigma_3} e^{ \frac{ -i \pi \sigma_{3}}{4}}
\right)_{+}= \left( P(\varphi_{-A}^{-1}(\xi)) e^{2 \xi^{1/2} \, \sigma_3} e^{ \frac{ -i \pi \sigma_{3}}{4}}\right)_{-} \begin{cases}
\begin{pmatrix}
0 & 1\\
-1 & 0
\end{pmatrix}, \quad & \xi \in (-\infty,0)\\ \\[4pt]
\begin{pmatrix}
1 & 0\\
1 & 1
\end{pmatrix}, \quad & \mbox{Arg}(\xi) = \pm \frac{ 2 \pi }{3},
\end{cases}
\end{gather}
which are precisely the jump relationships satisfied by $\Psi$,  as shown in Figure \ref{fig Contour Psi(xi)}.  We therefore set 
\begin{eqnarray}
\label{eq:PzDef}
    P(\varphi_{-A}^{-1}(\xi)) e^{2 \xi^{1/2} \, \sigma_3} e^{ \frac{ -i \pi \sigma_{3}}{4}} = \Psi(\xi).
\end{eqnarray}
Then, using the definition \eqref{eq:eq233} of $\Psi(\xi) $ and the (conformal and hence invertible) transformation $\varphi_{-A}(z)$, \textit{\eqref{eq:PzDef} defines} $P(z)$ \textit{back in the $z$ plane}.   
Unraveling the prior transformations from $E_{-A}$ to $m$ and from $m$ to $P$, we have established the following proposition.

\begin{Proposition}
The matrix valued function $E_{-A}$ given by the expression
\begin{gather}\label{eq:eqEminAdfn}
E_{-A}(z) = Q_{-A}(z) \, e^{-i \pi \sigma_3/4} \, \Psi(\varphi_{-A}(z)) \, e^{i \pi \sigma_3/4}e^{-2 \varphi_{-A}(z)^{1/2}  \sigma_3}\Bigg( \frac{1}{\mathcal{Y}(z)}\Bigg)^{\sigma_3/2} \begin{cases}
e^{-(\widetilde{y}(z) - \frac{1}{\tau} \ln(N)) \sigma_3}, \quad z \in \Omega\\
e^{-(\widetilde{y}(z) + \frac{1}{\tau} \ln(N)) \sigma_3},  \quad z \in \widetilde{\Omega}
\end{cases}
\end{gather}
satisfies the jump relation \eqref{eq:eqEminA}. Here,
$\Psi(\xi)$ is as in \eqref{eq:eq233}, $Q_{-A}$ is any analytic and invertible matrix-valued function inside $\mathcal{U}_{-A}$ (to be determined later), the map $\varphi_{-A}$ is defined as in equation \eqref{eq:eq67new}, and the domains $\Omega,\, \wt \Omega$ are as in Figure \ref{fig:localpar}.
\end{Proposition}

\begin{remark}
The factor $e^{-i \pi \sigma_3/4}$ in front of the function $\Psi(\varphi_{-A}(z))$ does not arise directly from the construction of $E_{-A}(z)$. However, since the local parametrices in the upper half-plane can be constructed from those in the lower half-plane via a symmetry relation, this factor is included to serve this purpose. Moreover, because it is an analytic function, it does not affect the jump of $E_{-A}(z)$.
\end{remark}

\subsection{Local parametrix near \texorpdfstring{$z=\wo{A}$}{A conjugate}}

Recall that $\mathcal{U}_{\wo{A}}$ is a small neighborhood of $z=\wo{A}$.  We seek $E_{\wo{A}}(z)$, a parametrix which satisfies the same jump relations as $E(z)$ within $\mathcal{U}_{\wo{A}}$, and matches the global parametrix on the boundary of $\mathcal{U}_{\wo{A}}$.  The jump relations for $E$ are:
\begin{equation}\label{eq:eq21}
E_{ +}(z) = E_{ -}(z) \, \begin{cases}
\begin{pmatrix}
0 & i\\
i & 0
\end{pmatrix}, \quad & z \in \eta_4\\ \\[4pt]
\mathfrak{N}^{+}(z), \quad & z \in \omega_1 \cup \omega_2
\end{cases}
\end{equation}

\begin{Proposition}
The matrix valued function $E_{\wo{A}}$ given by the expression
\begin{gather}\label{eq:eq235}
E_{\wo{A}}(z) = Q_{\wo{A}}(z) e^{i \pi \sigma_3/4} \, \Psi \left(\varphi_{\wo{A}}(z)\right) \, e^{-i \pi \sigma_3/4} e^{-2 \varphi_{\wo A}(z)^{1/2}  \sigma_3} \Big( e^{2 \widetilde{y}(z)} \mathcal{Y}(z) \Big)^{-\sigma_3/2}
\end{gather}
satisfies the jump relations \eqref{eq:eq21}. Here $\Psi$ is given by equation \eqref{eq:eq233}, $Q_{\wo{A}}$ is an analytic function inside $\mathcal{U}_{\wo{A}}$ and the map $\varphi_{\wo{A}}$ is defined as
\begin{equation}
\label{eq:varphi_Abar}
    \varphi_{\wo A}(z)^{1/2} = -\ln(N)\frac{c}{\tau}\int_{\wo A}^z \frac{\di s}{S(s)}\,.
\end{equation}
\end{Proposition}
\begin{proof}
It suffices to show that $E_{\wo{A}}(z)$ as in equation \eqref{eq:eq235} satisfies the jump relation \eqref{eq:eq21}. We therefore compute the jump of $E_{\wo{A}}(z)$, first along $\eta_4$. Note that the non-analytic functions in the definition of $E_{\wo{A}}(z)$ along $\eta_4$ are: $\Psi \left( \varphi_{\wo{A}}(z) \right)$, $\varphi_{\wo A}(z)^{1/2}$ and $\widetilde{y}(z)$. We then have:
\begin{eqnarray}
&&\nonumber
J_{E_{\wo{A}}}(z) =\left[\left( E_{\wo{A}} \right)_{-}(z)\right]^{-1} \left( E_{\wo{A}} \right)_{+}(z) \\
\label{eq:eqjump}
&&= \left( e^{2 \widetilde{y}_{-}(z)} \mathcal{Y}(z) \right)^{\sigma_3/2} e^{2 \varphi_{\wo{A}-}(z)^{1/2} \sigma_3} e^{i \pi \sigma_3/4}J_{\Psi} \left( \varphi_{\wo{A}}(z) \right) e^{-i \pi \sigma_3/4} e^{-2 \varphi_{\wo{A}+}(z)^{1/2} \sigma_3}\left( e^{2 \widetilde{y}_{+}(z)} \mathcal{Y}(z) \right)^{\sigma_3/2}.
\end{eqnarray}
From the properties of $\Psi$, for $z \in \eta_4$, we have:
\begin{gather}\label{eq:eq618n}
e^{i \pi \sigma_3/4} J_{\Psi}\left( \varphi_{\wo{A}}(z) \right) e^{-i \pi \sigma_3/4} = \begin{pmatrix}
  0 & i \\
  i & 0
\end{pmatrix}.
\end{gather}
In addition, $\varphi_{\wo{A}}(z)^{1/2}$ has jump across $\eta_4$, specifically:
\[\varphi_{\wo{A}-}(z)^{1/2} = - \varphi_{\wo{A}+}(z)^{1/2}, \quad z \in \eta_4\ .\]
Finally, recall that the function $\widetilde{y}$ satisfies the jump relation \eqref{eq:eqytilde} along $\eta_4$. Using all of these relations, straightforward multiplication of the right-hand side of equation \eqref{eq:eqjump} yields the desired jump of $E_{\wo{A}}(z)$ along $\eta_4$. Next, we compute the jump of $E_{\wo{A}}(z)$ along $\omega_1 \cup \omega_2$. Note that the only non-analytic piece in the definition of $E_{\wo{A}}(z)$ along these two rays is the function $\Psi \left( \varphi_{\wo{A}}(z) \right)$. We then have:
\begin{align*}
J_{E_{\wo{A}}}(z) & = \left( e^{2 \widetilde{y}(z)} \mathcal{Y}(z) \right)^{\sigma_3/2} e^{2 \phi_{\wo{A}}(z)^{1/2} \sigma_3} e^{i \pi \sigma_3/4} J_{\Psi}\left( \varphi_{\wo{A}}(z) \right) e^{-i \pi \sigma_3/4} e^{-2 \varphi_{\wo{A}}(z)^{1/2} \sigma_3}\left( e^{2 \widetilde{y}(z)} \mathcal{Y}(z) \right)^{\sigma_3/2}.
\end{align*}
From the properties of $\Psi$, for $z \in \omega_1 \cup \omega_2$, we have:
\[e^{i \pi \sigma_3/4} J_{\Psi} \left( \varphi_{\wo{A}}(z) \right) e^{-i \pi \sigma_3/4} = \begin{pmatrix}
  1 & 0 \\
  -i & 1
\end{pmatrix}.\]
Inserting this into the equation for $J_{E_{\wo{A}}}(z)$, a straightforward calculation yields the desired jump.
\end{proof}

\begin{remark}
We list here some properties of the map $\varphi_{\wo A}(z)^{1/2}$. The map $z \mapsto \varphi_{\wo A}(z)^{1/2}$ is a locally analytic map from $\mathcal{U}_{\wo{A}}$ to a neighborhood of $0$ such that $\varphi_{\wo A} \left( \eta_4 \cap \mathcal{U}_{\wo{A}} \right) \subset \mathbb{R}_{-}$. It is immediate to see that $\wo{A}$ is mapped to $0$ under this map. Moreover, for $z \in \overline{\Omega_{2}^{+}} \cap \mathcal{U}_{\wo{A}}$, we have that $\varphi_{\wo A}(z)^{1/2} = - \ln(N) \wt u(z)/4$. From lemma \ref{lem:immaginary_part} it follows that $\wt u(z)$ is purely imaginary on the boundary of $\Omega_{2}^{+}$. This shows that $\varphi_{\wo A}(z) \in \mathbb{R}_{-}$ along $\eta_4 \cap \mathcal{U}_{\wo{A}}$. Lastly, we choose $\omega_1$ and $\omega_2$ (within $\mathcal{U}_{\wo{A}}$) as the preimages of the rays with arguments $\pm 2\pi/3$, respectively.
\end{remark}

\subsection{Local parametrix near \texorpdfstring{$z=A$}{A}}

Recall that $\mathcal{U}_{A}$ is a small neighborhood of $z=A$, and we denote by $E_{A}$ the solution to the local problem near $z=A$. Inside $\mathcal{U}_{A}$, $E_{A}$ satisfies the jump relation
\begin{gather}\label{eq:jumpEA}
E_{A +}(z) = E_{A +}(z) J_{E_{A}}(z), \quad J_{E_{A}}(z) = \begin{cases}
\begin{pmatrix}
0 & i\\
i & 0
\end{pmatrix}, \quad & z \in \eta_3\\\\[2pt]
\mathfrak{M}^{-}(z), \quad & z \in \omega_3 \cup \omega_4.
\end{cases}
\end{gather}
\begin{Proposition}
The matrix valued function $E_{A}$ given by the expression
\begin{gather}\label{eq:eqsolEA}
E_{A}(z) = - Q_{A}(z) \sigma_2 e^{i \pi \sigma_3/4} \Psi \big( \varphi_{A}(z) \big) e^{-i \pi \sigma_3/4} \sigma_2 e^{-2 \varphi_{A}(z)^{1/2}  \sigma_3} \left( \frac{e^{2 \widetilde{y}(z)}}{\mathcal{Y}(z) e^{4i \theta(z)}} \right)^{-\sigma_3/2}\,,
\end{gather}
satisfies the jump relation \eqref{eq:jumpEA}. 
Here, $Q_{A}$ is an analytic function inside $\mathcal{U}_{A}$ (to be determined later), and the map $\varphi_{A}$ is defined as 
\begin{gather}
\label{eq:varphi_A}
\varphi_{A} (z)^{1/2} = -\ln(N) \frac{c}{\tau} \int_{A}^{z} \frac{\di s}{S(s)}\,.
\end{gather}
\end{Proposition}
\begin{proof}
First, define 
\begin{gather}\label{eq:eq622}
\Psi_{A} (\xi) = - \sigma_2 e^{i \pi \sigma_3/4}\Psi (\xi)e^{-i \pi \sigma_3/4} \sigma_2.
\end{gather}
The reader may verify that $\Psi_{A}$ satisfies the jump relation
\begin{eqnarray}
\left(\Psi_{A}\right)_{+} (\xi)=\left(\Psi_{A}\right)_{-} (\xi) 
\begin{cases}
\begin{pmatrix}
0 & -i\\
-i & 0
\end{pmatrix}, \quad &  \xi \in (-\infty,0)\\ \\[4pt]
\begin{pmatrix}
1 & -i\\
0  & 1
\end{pmatrix}, \quad &  \mbox{Arg}(\xi) = \pm 2 \pi / 3  .
\end{cases}
\end{eqnarray}
The composition $\Psi_{A}(\varphi_{A}(z))$ possesses jumps across the contours $\eta_{3}$, $\omega_{3}$ and $\omega_{4}$.  However, these contours are oriented in the opposite way as their images in the $\xi$ plane.  The reader may verify that the jumps in the $z$ plane are the inverses of those in the $\xi$ plane:
\begin{eqnarray}
\label{eq:k6.24}
\left(\Psi_{A}(\varphi_{A}(z))\right)_{+}=\left(\Psi_{A}(\varphi_{A}(z))\right)_{-}
\begin{cases}
\begin{pmatrix}
0 & i\\
i & 0
\end{pmatrix}, \quad &  z \in \eta_{3}\\ \\[4pt]
\begin{pmatrix}
1 & i\\
0  & 1
\end{pmatrix}, \quad &  z \in \omega_{3} \cup \omega_{4} .
\end{cases}
\end{eqnarray}
For example, as $z$ approaches $\hat{z}$ from the $+$ side of $\eta_{3}$, $\xi = \varphi_{A}(z)$ approaches $\varphi_{A}(\hat{z})$ \textit{from below, i.e. from the $-$ side of $\mathbb{R}_{-}$.}  Similarly, if $z$ approaches $\hat{z}$ from the $-$ side of $\eta_{3}$, then $\xi = \varphi_{A}(z)$ approaches $\varphi_{A}(\hat{z})$ from the $+$ side of $\mathbb{R}_{-}$. This means that
\begin{eqnarray}
&&
 \lim_{\substack{z \to \hat{z} \\ z \in + \mbox{side of }  \eta_3}}  \Psi_{A}(\varphi_{A}(z) )= \left(\Psi_{A} \right)_{-}(\varphi_{A}(\hat{z})) = \left(\Psi_{A} \right)_{+}(\varphi_{A}(\hat{z})) \pmtwo{0}{i}{i}{0} \\
 && =  \lim_{\substack{z \to \hat{z} \\ z \in - \mbox{side of }  \eta_3}}  \Psi_{A}(\varphi_{A}(z) ) \pmtwo{0}{i}{i}{0},
\end{eqnarray}
which explains \eqref{eq:k6.24}.

Now the definition \eqref{eq:eqsolEA} can be rewritten as 
\begin{eqnarray*}
E_{A}(z) = Q_{A}(z) \Psi_{A}(\varphi_{A}(z))
     e^{-2 \varphi_{A}(z)^{1/2}  \sigma_3} \left( \frac{e^{2 \widetilde{y}(z)}}{\mathcal{Y}(z) e^{4i \theta(z)}} \right)^{-\sigma_3/2},
\end{eqnarray*}
and the known jump relations satisfied by the last two factors appearing on the right hand side can be used to verify that $E_{A}(z)$ satisfies the jump relations in \eqref{eq:jumpEA}.
\end{proof}

\subsection{Error function and global parametrix}

In this subsection, we define the global approximation to $E$, and introduce the error function $\cE$.  The global approximation is defined separately within each of the four disks and in the region exterior to all the disks.

\begin{gather}
\label{eq:E_approx_def}
 \quad E_{\mathrm{approx}}(z) = \begin{cases}
E_{A}(z), \quad & z\in \mathcal{U}_{A}\\
E_{\wo{A}}(z), \quad & z\in \mathcal{U}_{\wo{A}} \\
E_{-\wo{A}}(z), \quad & z\in\mathcal{U}_{-\wo{A}} \\
E_{-A}(z), \quad & z\in\mathcal{U}_{-A} \\
\Lambda(z), \quad & \text{otherwise}.
\end{cases}
\end{gather}

The error matrix, $\cE (z)$ is defined as follows

\begin{eqnarray}
\cE (z) = E(z) \, \left( E_{\mathrm{approx}}(z) \right)^{-1}.
\end{eqnarray}

Here, recall that $E$ is given by \eqref{eq:eqdfnE} and $\Lambda$ is the outer approximate solution given in \eqref{eq:eqLambda}. We choose the functions  $Q_{A}$, $Q_{\wo{A}}$, $Q_{-A}$, $Q_{-\wo{A}}$ appearing in the definitions of the local solutions so that the quantity $\cE(z)$ is in the small norm setting.  Specifically, we compute the jump of $\cE(z)$ across the disk boundaries, and then choose these quantities to ensure that the jump matrices are of the form $I + \mbox{small}$.  

We start with the jump relationship satisfied by $\cE(z)$ for $z\in \partial\mathcal{U}_{\wo{A}}$ where $\partial\mathcal{U}_{\wo{A}}$ (as all the four neighborhoods) has been chosen with counterclockwise orientation:
\begin{align}
J_{\cE}(z) &= \cE_{-}^{-1}(z) \, \cE_{+}(z)\\
& = \Big( E(z) \Lambda^{-1}(z) \Big)^{-1} \, E(z) E_{\wo{A}}^{-1}(z)\\
&= \Lambda(z) \Big( e^{2 \widetilde{y}(z)} \mathcal{Y}(z) \Big)^{\sigma_3/2} e^{2 \varphi_{\wo A}(z)^{1/2}  \sigma_3} e^{i \pi \sigma_3/4} \, \Psi^{-1} \left(\varphi_{\wo{A}}(z)\right) \, e^{-i \pi \sigma_3/4} Q_{\wo{A}}^{-1}(z).
\end{align}

Next we utilize the asymptotic behavior of the function $\Psi$ defined in \eqref{eq:eq233}, as presented in  \cite[p. 369]{Kuijlaars-2004}
\begin{gather}\label{eq:eq246}
\Psi(\xi) = \left( 2 \pi \xi^{1/2} \right)^{-\sigma_3/2} \, \frac{1}{\sqrt{2}} \begin{pmatrix}
1 + \mathcal{O}(\xi^{-1/2}) & i + \mathcal{O}(\xi^{-1/2})\\
i + \mathcal{O}(\xi^{-1/2}) & 1 + \mathcal{O}(\xi^{-1/2})
\end{pmatrix} e^{2 \xi^{1/2} \sigma_3}
\end{gather}
uniformly as $\xi \to \infty$. 
Using this, we find
\begin{align}
J_{\cE}(z) & = \frac{\sqrt{2}}{2} \, \Lambda(z) \Big( e^{2 \widetilde{y}(z)} \mathcal{Y}(z) \Big)^{\sigma_3/2} e^{2 \varphi_{\wo A}(z)^{1/2}  \sigma_3} e^{i \pi \sigma_3/4} e^{-2 \varphi_{\wo{A}}(z)^{1/2} \sigma_3} \\
\nonumber &  \left( I - \frac{1}{2} \begin{pmatrix}
1 & -i\\
-i & 1
\end{pmatrix} \begin{pmatrix}
\mathcal{O}\left( \varphi_{\wo{A}}^{-1/2}(z) \right) & \mathcal{O}\left( \varphi_{\wo{A}}^{-1/2}(z) \right)\\\\[0.5pt]
\mathcal{O}\left( \varphi_{\wo{A}}^{-1/2}(z) \right) & \mathcal{O}\left( \varphi_{\wo{A}}^{-1/2}(z) \right)
\end{pmatrix} + \ldots \right) \begin{pmatrix}
1 & -i\\
-i & 1
\end{pmatrix} \left( 2 \pi \varphi_{\wo{A}}^{1/2}(z) \right)^{\sigma_3/2} e^{- i \pi \sigma_3/4} Q_{\wo{A}}^{-1}(z) \,.
\end{align}
The last one simplifies to:
\begin{align}
J_\cE(z)& = \frac{\sqrt{2}}{2} \, \Lambda(z) \Big( e^{2 \widetilde{y}(z)} \mathcal{Y}(z) \Big)^{\sigma_3/2}  e^{i \pi \sigma_3/4}  \\
& \left( I - \frac{1}{2} \begin{pmatrix}
1 & -i\\
-i & 1
\end{pmatrix} \begin{pmatrix}
\mathcal{O}\left( \varphi_{\wo{A}}^{-1/2}(z) \right) & \mathcal{O}\left( \varphi_{\wo{A}}^{-1/2}(z) \right)\\\\[0.5pt]
\mathcal{O}\left( \varphi_{\wo{A}}^{-1/2}(z) \right) & \mathcal{O}\left( \varphi_{\wo{A}}^{-1/2}(z) \right)
\end{pmatrix} + \dots \right) \begin{pmatrix}
1 & -i\\
-i & 1
\end{pmatrix} \left( 2 \pi \varphi_{\wo{A}}^{1/2}(z) \right)^{\sigma_3/2} e^{- i \pi \sigma_3/4} Q_{\wo{A}}^{-1}(z)\,.
    \label{eq:eqjumperr}
\end{align}

At this point, we require that $J_{\cE}$ behaves like the identity matrix plus small corrections across $\partial\mathcal{U}_{\wo{A}}$. Therefore, matching the leading order terms in both sides in the last relation, we must define $Q_{\wo{A}}$ as follows:
\begin{gather}\label{eq:eq246n}
Q_{\wo{A}}(z) = \frac{\sqrt{2}}{2} \Lambda(z) \Big( e^{2 \widetilde{y}(z)} \mathcal{Y}(z) \Big)^{\sigma_3/2}  e^{i \pi \sigma_3/4} \begin{pmatrix}
1 & -i\\
-i & 1
\end{pmatrix} \left( 2 \pi \varphi_{\wo{A}}(z)^{1/2} \right)^{\sigma_3/2} e^{- i \pi \sigma_3/4}.
\end{gather}

\begin{Proposition}\label{prop55}
The function $Q_{\wo{A}}$ defined in  \eqref{eq:eq246n} is analytic for $z \in \mathcal{U}_{\wo{A}}$. Moreover, using this function in \eqref{eq:eq235}, we have
\begin{eqnarray}
\label{eq:jumperras}
J_{\cE} = I + \mathcal{O} \left( \frac{1}{\ln{(N)}}\right), \ \ z \in \partial \mathcal{U}_{\wo{A}}.
\end{eqnarray}

\end{Proposition}
\begin{proof}
To establish the analyticity of the function $Q_{\wo{A}}(z)$ inside the disk $\mathcal{U}_{\wo{A}}$, observe that all the functions comprising $Q_{\wo{A}}(z)$ are analytic within $\mathcal{U}_{\wo{A}}$, except for the functions $\Lambda(z)$, $\widetilde{y}(z)$, and $\varphi_{\wo{A}}^{1/2}(z)$ which have jumps across $(-\infty,0)$. Computing the jump of $Q_{\wo{A}}(z)$ on $(-\infty,0)$ by using the respective jumps of these functions and the definition of $Q_{\wo{A}}(z)$, it is straightforward to verify that these jumps cancel each other. Consequently, $Q_{\wo{A}}(z)$ is analytic on $(-\infty,0)$, and therefore analytic within $\mathcal{U}_{\wo{A}}$.  Returning to \eqref{eq:eqjumperr}, we see immediately from the definition \eqref{eq:varphi_Abar} that \eqref{eq:jumperras} holds.
\end{proof}

\begin{remark}\label{rem:small_norm}
Note that we can apply a similar approach to that used in proposition \ref{prop55} to identify the analytic functions $Q$ present in all local approximate solutions. By utilizing these functions and employing an argument analogous to that in proposition \ref{prop55} we can establish that
\begin{equation}
\label{eq:JEDiskBound}
J_{\cE}(z) = \mathbb{I} + \mathcal{O} \left(\frac{1}{\log{N}}\right), \quad z\in \partial\left(\mathcal{U}_{A} \cup  \mathcal{U}_{\wo{A}} \cup  \mathcal{U}_{-\wo{A}} \cup  \mathcal{U}_{-A} \right)\,.
\end{equation}
Since the determination of each of  $Q_{A}$, $Q_{-A}$, and $Q_{-\wo{A}}$ is achieved by similar considerations, we omit these calculations for the sake of brevity.
Next, we are interested at the jump of the function $\cE$ across the remaining contours.  As observed previously, $\cE$ has no jumps within the four disks. From its definition, we can check that $\cE$ is also analytic across $\eta_3 \cup \eta_4 \cup \gamma$.  There remains the jumps across the contours $\omega_1 \cup \omega_2 \cup \omega_3 \cup \omega_4$:
\begin{align}
 J_{\cE}(z) = \begin{cases}
\Lambda(z) \, \mathfrak{M}^{-}(z) \, \Lambda^{-1}(z), & \quad z \in \omega_3 \cup \omega_4\\
\Lambda(z) \, \mathfrak{N}^{+}(z) \, \Lambda^{-1}(z), & \quad z \in \omega_1 \cup \omega_2.
\end{cases}   
\end{align}
From proposition \eqref{prop:small_norm} we can conclude that $J_{\cE}$ behaves like identity plus small across the contour $\omega_1 \cup \omega_2 \cup \omega_3 \cup \omega_4$. 
We have established all the necessary ingredients to conclude that $\cE$ defined in \eqref{eq:E_approx_def} satisfies a small norm Riemann-Hilbert problem.  The contour for this Riemann-Hilbert problem consists of the four disk boundaries, and the contours $\{ \omega_{j} \cap \cV\}_{j=1}^{4}$, where we recall that $\cV = \C \setminus \{\cU_A\cup \cU_{\wo A}\cup\cU_{-A}\cup\cU_{-\wo A}\}$, as shown in Fig. \ref{fig:J_calE}.  On each of the contours, the jump matrix for $\cE$ is uniformly close to $I$.  Specifically, we have \eqref{eq:JEDiskBound}, and on each of the contours $\omega_{j}$, we have the uniform bound
\begin{equation}
    J_{\cE}(z) = \mathbb{I} + \mathcal{O} \left(e^{ - c \log{N} }\right), \quad z \in \omega_{j}\cap \cV, \ j = 1, \ldots, 4.
\end{equation}
Then, following (for example) the statement and proof of Theorem 7.10 on P. 1532 of \cite{DKMVZ1}, we may conclude that $\cE(z)$ not only exists, it also satisfies 
\begin{equation}\label{eq:eqfuncE}
\cE (z) = \mathbb{I} + \mathcal{O} \left(\frac{1}{\log{N} \left( 1 + |z| \right)}\right)
\end{equation}
for all $z \in \mathbb{C}$. Therefore, the global solution $E$ exists and is given by the expression
\begin{equation}\label{eq:eqdfnEglobal}
E(z) = \left( \mathbb{I} + \mathcal{O} \left(\frac{1}{\log{N}\left( 1 + |z| \right)}\right) \right) \, E_{\mathrm{approx}}(z), \quad z \in \mathbb{C}.
\end{equation}
By unraveling all the transformations we have introduced in Sections \ref{sec:dressing} and \ref{sec:Lenses}, we have established a complete asymptotic description of  the solution to RHP \ref{rhp:Atil}.  In the next section we will carry out this unraveling in order to determine the asymptotic behavior of the solution to the NLS equation.
\end{remark}

\begin{figure}
    \centering
\tikzset{every picture/.style={line width=0.75pt}}       

\begin{tikzpicture}[x=0.5pt,y=0.5pt,yscale=-1,xscale=1]

\draw [line width=1.5]    (137.34,144.08) .. controls (173.13,158.17) and (278.33,163.66) .. (329,163) .. controls (379.67,162.34) and (447.58,158.69) .. (519.68,142.75) ;
\draw [shift={(224.25,158.79)}, rotate = 4.7] [fill={rgb, 255:red, 0; green, 0; blue, 0 }  ][line width=0.08]  [draw opacity=0] (11.61,-5.58) -- (0,0) -- (11.61,5.58) -- cycle    ;
\draw [shift={(417.27,158.52)}, rotate = 354.48] [fill={rgb, 255:red, 0; green, 0; blue, 0 }  ][line width=0.08]  [draw opacity=0] (11.61,-5.58) -- (0,0) -- (11.61,5.58) -- cycle    ;

\draw [line width=1.5]    (137.34,80.97) .. controls (185.21,58.77) and (219,64) .. (330,62) .. controls (441,60) and (483.75,77.52) .. (519.68,79.64) ;
\draw [shift={(224.23,63.26)}, rotate = 358.38] [fill={rgb, 255:red, 0; green, 0; blue, 0 }  ][line width=0.08]  [draw opacity=0] (11.61,-5.58) -- (0,0) -- (11.61,5.58) -- cycle    ;
\draw [shift={(417.61,64.93)}, rotate = 5.62] [fill={rgb, 255:red, 0; green, 0; blue, 0 }  ][line width=0.08]  [draw opacity=0] (11.61,-5.58) -- (0,0) -- (11.61,5.58) -- cycle    ;

\draw  [line width=1.5]  (105.98,112.53) .. controls (105.98,95.1) and (120.02,80.97) .. (137.34,80.97) .. controls (154.66,80.97) and (168.7,95.1) .. (168.7,112.53) .. controls (168.7,129.95) and (154.66,144.08) .. (137.34,144.08) .. controls (120.02,144.08) and (105.98,129.95) .. (105.98,112.53) -- cycle ;
 
\draw  [line width=1.5]  (488.32,111.2) .. controls (488.32,93.77) and (502.36,79.64) .. (519.68,79.64) .. controls (537,79.64) and (551.05,93.77) .. (551.05,111.2) .. controls (551.05,128.62) and (537,142.75) .. (519.68,142.75) .. controls (502.36,142.75) and (488.32,128.62) .. (488.32,111.2) -- cycle ;

\draw [line width=1.5]    (137.34,299.08) .. controls (173.13,313.17) and (278.33,318.66) .. (329,318) .. controls (379.67,317.34) and (447.58,313.69) .. (519.68,297.75) ;
\draw [shift={(239.88,315.01)}, rotate = 184.2] [fill={rgb, 255:red, 0; green, 0; blue, 0 }  ][line width=0.08]  [draw opacity=0] (11.61,-5.58) -- (0,0) -- (11.61,5.58) -- cycle    ;
\draw [shift={(432.02,312.02)}, rotate = 173.94] [fill={rgb, 255:red, 0; green, 0; blue, 0 }  ][line width=0.08]  [draw opacity=0] (11.61,-5.58) -- (0,0) -- (11.61,5.58) -- cycle    ;

\draw [line width=1.5]    (137.34,235.97) .. controls (185.21,213.77) and (219,219) .. (330,217) .. controls (441,215) and (483.75,232.52) .. (519.68,234.64) ;
\draw [shift={(239.97,217.89)}, rotate = 178.95] [fill={rgb, 255:red, 0; green, 0; blue, 0 }  ][line width=0.08]  [draw opacity=0] (11.61,-5.58) -- (0,0) -- (11.61,5.58) -- cycle    ;
\draw [shift={(432.37,221.5)}, rotate = 186.44] [fill={rgb, 255:red, 0; green, 0; blue, 0 }  ][line width=0.08]  [draw opacity=0] (11.61,-5.58) -- (0,0) -- (11.61,5.58) -- cycle    ;

\draw  [line width=1.5]  (105.98,267.53) .. controls (105.98,250.1) and (120.02,235.97) .. (137.34,235.97) .. controls (154.66,235.97) and (168.7,250.1) .. (168.7,267.53) .. controls (168.7,284.95) and (154.66,299.08) .. (137.34,299.08) .. controls (120.02,299.08) and (105.98,284.95) .. (105.98,267.53) -- cycle ;

\draw  [line width=1.5]  (488.32,266.2) .. controls (488.32,248.77) and (502.36,234.64) .. (519.68,234.64) .. controls (537,234.64) and (551.05,248.77) .. (551.05,266.2) .. controls (551.05,283.62) and (537,297.75) .. (519.68,297.75) .. controls (502.36,297.75) and (488.32,283.62) .. (488.32,266.2) -- cycle ;

\draw (329,159.6) node [anchor=south] [inner sep=0.75pt]  [font=\Large]  {$\omega _{3}$};

\draw (330,58.6) node [anchor=south] [inner sep=0.75pt]  [font=\Large]  {$\omega _{4}$};

\draw (519.68,111.2) node    {$\mathcal{U}_{A}$};

\draw (137.34,112.53) node    {$\mathcal{U}_{-\overline{A}}$};

\draw (329,314.6) node [anchor=south] [inner sep=0.75pt]  [font=\Large]  {$\omega _{1}$};

\draw (330,213.6) node [anchor=south] [inner sep=0.75pt]  [font=\Large]  {$\omega _{2}$};

\draw (519.68,266.2) node    {$\mathcal{U}_{\overline{A}}$};

\draw (137.34,267.53) node    {$\mathcal{U}_{-A}$};

\end{tikzpicture}
    \caption{Jumps for the matrix $J_\cE$}
    \label{fig:J_calE}
\end{figure}

\section{Solution to the NLS equation}
\label{sec:solution}
In this section, we enforce all the previous constructions to prove Theorem \ref{thm:main_sym}.
\begin{remark}
To obtain $\psi_{SG}(x,t;N)$ (and its conjugate) we see from RHP \ref{rhp:Atil} that we need the behavior of $\tilde{A}(z)$ as $z \to \infty$. Specifically, we have:
\begin{equation}\label{eq:recformula1}
\psi_{SG}(x,t;N) =  \lim_{z\to \infty} 2iz \, \tilde{A}_{12}(z), \quad \wo{\psi_{SG}(x,t;N)} = \lim_{z\to \infty} 2iz \, \tilde{A}_{21}(z).
\end{equation}
Therefore, an expression for the matrix-valued function $\tilde{A}(z)$ for large-$z$ is required.
Observe that once $z$ is large enough, we are outside all of the various domains, which gives us
\begin{equation}\label{eq:eqapprAtilde}
\tilde{A}(z) = \xi(z) \ e^{-\left( \log{N} g_\infty + y_\infty \right) \sigma_3} \, \, e^{i \nu_0(z) \sigma_3} \Xi(z) e^{i \nu(z) \sigma_3} \,  e^{\left( \log{N} g(z) + y(z) \right) \sigma_3}.
\end{equation}
To see this, recall that in the exterior region, we have:
\begin{eqnarray}
E(z) = D(z) = C(z), \quad B(z) = \tilde{A}(z), \quad E_{\mathrm{approx}}(z) = \Lambda(z),
\end{eqnarray}
which, when combined with the definitions of $C(z)$, $\Lambda(z)$ (see relations \eqref{eq:eqdfnC}, \eqref{eq:eqLambda} and \eqref{eq:E_approx_def}) leads to \eqref{eq:eqapprAtilde}.
\end{remark}
Now, we prove Theorem \ref{thm:main_sym}, which we repeat for the convenience of the reader.
\begin{theorem}
\label{thm:main_sym_sec7}
    Under the $N$-soliton condensate scattering data assumption, for all $(x,t)$ in a compact set $\fK$, there is $N_{0}$ so that for all $N > N_{0}$, $\psi_{SG}(x,t;N)$ has the following asymptotic behavior

\begin{equation}
\label{eq:mod_square_00_sec7}
        \psi_{SG}(x,t;N)  = -2 i \frac{\Re(A)\Im(A)}{|A|} \sd\left( 2|A|x + \frac{2|A|}{\pi}\Re\left(\int_{\eta_1} \frac{\log(2\pi h(s)\rho(s))}{R_+(s)}\di  s\right); \cos(\theta)\right) e^{it(A^2+\wo A^2) + i \phi_0}+\mathcal{O} \left(\frac{1}{\log{N}}\right),
\end{equation}
where the error term $\mathcal{O}\left( \frac{1}{\log{N} } \right)$ is uniform for all $(x,t)$ in the compact set $\fK$.  Here $\sd$ is the Jacobi elliptic function $\sd$ \cite[Chap. 22]{dlmf},  
and $R(z),\phi_0$ are given by

\begin{align}
\label{eq:auxiliary_def_thm}
    &R(z) = (z-A)^{\frac{1}{2}}(z+A)^{\frac{1}{2}}(z-\wo A)^{\frac{1}{2}}(z+\wo A)^{\frac{1}{2}}\,,\\
    &\phi_0 = \ln(N) \frac{K(\cos(\theta))}{K(\sin(\theta))} + \Re\left( \int_{\eta_1} \frac{s\ln(2\pi h(s)\rho(s))}{R_+(s)}\di s\right)\,, \quad \theta = \arg(A)\
\end{align}
where the function $K(k)$ is the complete elliptic integral of the fist kind \cite[Chap. 19]{dlmf}, the branch-cuts are the canonical ones,  and the integration is performed on the straight line connecting the endpoints. 
\end{theorem}
\begin{proof}
Using equations \eqref{eq:recformula1}, \eqref{eq:eqapprAtilde} and \eqref{eq:eqfuncE}, we have:

\begin{equation}
\label{eq:solution_limit}
\begin{split}
     \psi_{SG}(x,t;N) &=\lim_{z\to \infty} 2iz \, \xi(z) \, s_{12}(z;x,t) \, \exp\left( -\ln(N)(g_\infty+g(z)) + i\nu_0 -i\nu(z) - y(z)-y_\infty \right) + \mathcal{O} \left(\frac{1}{\log{N}}\right)\\ & 
     = \lim_{z\to \infty} 2iz \, \xi(z) \, \Pi(z) \, \exp\left( -\ln(N)(g_\infty+g(z)) + i\nu_0 -i\nu(z) - y(z)-y_\infty + i\delta(z)\right) + \mathcal{O} \left(\frac{1}{\log{N}}\right)\,,\\ & 
     =2i\left(\lim_{z\to\infty}z\Pi(z)\right)\exp(-2\ln(N)g_\infty-2y_\infty+2i\nu_0) + \mathcal{O} \left(\frac{1}{\log{N}}\right)
\end{split}
\end{equation}
where we set
\begin{equation}
    \Pi(z) = \frac{1}{Z_2} \frac{\Theta \left( \mathcal{A}(z) + \frac{3+\tau}{4}; \tau \right)\Theta \left( \mathcal{A}(z) -  \frac{\zeta(x,N)}{2\pi} + \frac{\tau - 1}{4}; \tau \right)}{\Theta \left( \mathcal{A}(z) + \frac{ 1}{2}; \tau \right)\Theta \left( \mathcal{A}(z)  + \frac{\tau}{2}; \tau \right)}\,.
\end{equation}

Focusing on the exponential term, one can compute the terms  $g_\infty,y_\infty,\nu_0$ from Proposition \ref{prop:sym_R}, Proposition \ref{prop:prop210} and Proposition \ref{prop:asymptotic_for_delta}
obtaining that

\begin{align}
    g_\infty &= \frac{1}{2\tau} \\
    y_\infty & = \frac{it}{2}(A^2+\wo A^2)  + \frac{c}{\tau} \left( \pi x  +\Re\left( \int_{\eta_1} \frac{\log \,( 2 \pi h(s) \rho(s))}{R_{+}(s)} \di s\right)\right) + \frac{i}{\pi} \Re\left( \int_{\eta_1} \frac{s\log(2\pi h(s)\rho(s))}{R_+(s)}\di s\right) \\
    \nu_0 &=i\left( 4c\pi x + 2\ln(N) + 4c \Re\left( \int_{\eta_1} \frac{\log(2\pi h(s)\rho(s))}{R_+(s)}\di s\right)\right)\left(\frac{1}{4} + \frac{1}{4\tau}\right)\,.
\end{align}
With tears, one deduces that 

\begin{equation}
\label{eq:exp_part}
    \exp(-2\ln(N)g_\infty-2y_\infty+2i\nu_0) = \exp\left(\frac{i\zeta(x,N)}{2} + it(A^2+\wo A^2) + i \phi_0\right)\,,
\end{equation}
where $\phi_0$ is defined in \eqref{eq:auxiliary_def_thm} and $\zeta(x,N)$ in \eqref{eq:zeta_def_proof}. 
We now consider

\[
2i\lim_{z\to\infty}z\Pi(z) = -2ic \frac{\Theta'\left(\frac{\tau+1}{2}\right)\Theta\left(-\frac{\zeta(x,N)}{2\pi} + \frac{\tau-1}{2}\right)\Theta\left(\frac{3-\tau}{4}\right)}{\Theta(0)\Theta\left(\frac{3\tau-1}{4}\right)\Theta\left(\frac{\zeta(x,N)}{2\pi}\right)}\,,
\]
where we used Proposition \ref{prop:abel_map_values}. We can further simplify the previous equation by expressing it first in terms of the Jacobi theta functions \cite[Chap. 20]{dlmf}, and then in terms of the $\sd$ elliptic function \cite[Chap. 22]{dlmf} as 

\begin{equation}
\label{eq:lim_part}
    2i\lim_{z\to\infty}z\Pi(z) = -2 i \frac{\Re(A)\Im(A)}{|A|} \sd\left( 2|A|x + \frac{2|A|}{\pi}\Re\left(\int_{\eta_1} \frac{\log(2\pi h(s)\rho(s))}{R_+(s)}\di  s\right); \cos(\theta)\right)e^{-i\frac{\zeta(x,N)}{2}}\,.
\end{equation}
Combing \eqref{eq:exp_part}-\eqref{eq:lim_part}, we deduce the result.

\end{proof}

\begin{remark}
 Note that the asymptotic description of the solution is valid for $(x,t)$ in compact sets. The solution is a pure periodic elliptic wave and the period is independent of the distribution $\rho(z)$, and the norming constants $c_j$. Finally, we remark that the modulus of the solution does not depend on time.
\end{remark}

We show in figure \ref{fig:solution} several plots of the solution \eqref{eq:solution_limit} for different values of $N$.

\section{Connection to the kinetic theory}
\label{sec:kinetic}

As discussed in Section \ref{sec:introsolgas},  in order to connect to the kinetic theory of solitons, we will consider a tracer soliton interacting with a sequence of $N$-soliton condensate solutions.  
That is, we consider RHP \ref{rhp:reflection_RHP}, with  an additional pair of poles $\{ \lambda_o, \wo{\lambda}_o\}$, where $\lambda_o$ is a point in the discrete spectrum in $\mathbb{C}_{+}$ and outside the curve $\Gamma_1$. By symmetry $\wo{\lambda}_o$ is its conjugate in $\mathbb{C}_{-}$ and outside the curve $\Gamma_2$. Specifically, RHP \ref{rhp:reflection_RHP}, for a pure solitonic solution, becomes:
\begin{RHP}
    Given spectral data $\mathcal{D}(t)$, find a $2 \times 2$ matrix-valued function $M(z)$ such that:
\begin{itemize}
	\item $\ {M}(z)$ is meromorphic in  $ \C$, with simple poles at the points $\{ \lambda_{j}, \overline{\lambda_{j}}\}_{j=1}^{N}$ and the pair $\{ \lambda_o, \wo{\lambda}_o\}$.
	\item At each of the poles $\lambda_{j}$ in $\mathbb{C}_{+}$, $\ {M}(z)$ satisfies the residue condition 
	\begin{equation}
    \label{eq:rescondCP}
		\res_{\lambda_j} M(z) = \lim_{z \to \lambda_j } \left[ M(z)\begin{pmatrix}
			0 & 0 \\
			c_j e^{2i\theta(\lambda_j;x,t)} & 0
		\end{pmatrix}\right]\,,
    \end{equation}
and at each pole $\overline{\lambda_{j}}$ in $\mathbb{C}_{-}$:
\begin{equation}
    \label{eq:rescondCM}
\res_{\wo {\lambda_j}} M(z) = \lim_{z \to \wo \lambda_j } \left[ M(z) \begin{pmatrix}
			0 & -\wo c_j e^{-2i\theta(\wo \lambda_j;x,t)} \\
			0 & 0
		\end{pmatrix}\right].
	\end{equation}
    \item At $z=\lambda_o$, $M(z)$ satisfies the residue condition:
\begin{equation}
\label{eq:rescondoo}
		\res_{\lambda_o} M(z) = \lim_{z \to \lambda_o } \left[ M(z)\begin{pmatrix}
			0 & 0 \\
			c_o e^{2i\theta(\lambda_o;x,t)} & 0
		\end{pmatrix}\right]\,,
    \end{equation} 
and at $z=-\wo {\lambda}_o$, we have:
    \begin{equation}
    \label{eq:rescondooc}
\res_{\wo {\lambda_o}} M(z) = \lim_{z \to \wo \lambda_o } \left[ M(z) \begin{pmatrix}
			0 & -\wo c_o e^{-2i\theta(\wo \lambda_o;x,t)} \\
			0 & 0
		\end{pmatrix}\right].
	\end{equation}
\item $\ {M}(z) = \  I + O(z^{-1})$, as $z \to \infty$.
\end{itemize}
\end{RHP}

The analysis described in the previous sections still applies, except that we need to account at each step for the poles $z=\lambda_o, \wo{\lambda}_o$. So the transformations of Section \ref{sec:dressing}, encircling the poles $\lambda_{j}$ with $\Gamma_{1} $ (and similarly $\wo \lambda_j$ with $\Gamma_2$), and replacing the residue conditions at each pole with a jump across the contours $\Gamma_1$ and $\Gamma_{2}$, which we then collapse onto $\eta_1$ and $\eta_{2}$ are carried out in exactly the same way, but ensuring that $\lambda_{o}$ and $\wo{\lambda}_{o}$ are outside the contours. This implies that the series of transformations:
\[M(z) \to A(z) \to \widetilde{A}(z) \to B(z)\]
all apply. Since $\lambda_{o}$ and $\wo{\lambda_{o}}$ are outside $\Gamma_1\cup\Gamma_2$,  the residue conditions at $\lambda_{o}$ \eqref{eq:rescondoo} and $\wo {\lambda_{o}}$ \eqref{eq:rescondooc} may be rewritten, replacing $M$ with $B$:
\begin{equation}
\label{eq:rescondBoo}
		\res_{\lambda_o} B(z) = \lim_{z \to \lambda_o } \left[ B(z)\begin{pmatrix}
			0 & 0 \\
			c_o e^{2i\theta(\lambda_o;x,t)} & 0
		\end{pmatrix}\right]\,,
    \end{equation} 
and at $z=-\wo {\lambda}_o$, we have:
    \begin{equation}
    \label{eq:rescondBooc}
\res_{\wo {\lambda_o}} B(z) = \lim_{z \to \wo \lambda_o } \left[ B(z) \begin{pmatrix}
			0 & -\wo c_o e^{-2i\theta(\wo \lambda_o;x,t)} \\
			0 & 0
		\end{pmatrix}\right].
	\end{equation}
The next transformation introduces the function $C(z)$, which preserves its jump but alters the residue condition at $z= \lambda_o, \wo{\lambda}_o$. Specifically, $C(z)$ satisfies the following RHP:

\begin{RHP}
We look for a function $C(z)\in \text{Mat}(\C,2)$ such that:
\begin{itemize}
\item $C(z)$ is analytic on $\mathbb{C}\setminus\{\eta_1\cup \eta_2\cup \gamma\}$, and achieves its boundary values $C_{\pm}$ on $\eta_{1}$ and $\eta_{2}$ in the sense of smooth functions.
\item The boundary values satisfy $C_{+}(z) = C_{-}(z) J_{C}(z)$, where the jump matrix $J_{C}(z)$ is given as in RHP \ref{rhp:K_intial}.
\item At $z=\lambda_o$, $C(z)$ satisfies the residue condition:
\begin{equation}
\res_{\lambda_o} C(z) = \lim_{z \to \lambda_o } \left[ C(z) \begin{pmatrix}
			0 & 0 \\
			c_o e^{2i\theta(\lambda_o;x,t)} e^{-2 \left( \ln (N) g(\lambda_o) + y(\lambda_o)\right)}& 0
		\end{pmatrix}\right]
	\end{equation}
and at $z=\wo {\lambda}_o$, $C(z)$ satisfies the residue condition:
\begin{equation}
\res_{\wo {\lambda}_o} C(z) = \lim_{z \to \wo {\lambda}_o } \left[ C(z) \begin{pmatrix}
			0 & - \wo {c_o} e^{-2i\theta(\wo {\lambda}_o;x,t)} e^{2 \left( \ln (N) g(\wo {\lambda}_o) + y(\wo {\lambda}_o)\right)} \\
			0 & 0
		\end{pmatrix}\right].
	\end{equation}
\item $C(z) = I + \mathcal{O}(1/z)$, as $z \to \infty$.
\end{itemize}
\end{RHP}
The analysis carried out in Section \ref{sec:Lenses} may then be repeated, from $C$ to $D$ and the lens-opening from $D$ to $E$, so long as the additional poles $\lambda_{o}$ and $\wo{\lambda_{o}}$ are outside of the domains where transformations are made.  Next, the global approximation constructed in Section \ref{sec:global} is replaced by a modified global approximation that includes the additional poles $\lambda_{o}$ and $\bar{\lambda_{o}}$.  The quantity $\Lambda$ solving Riemann-Hilbert problem \ref{RHP:model_problem} then also has these two simple poles, and satisfies the residue conditions
\begin{equation}
\res_{\lambda_o} \Lambda(z) = \lim_{z \to \lambda_o } \left[ \Lambda(z) \begin{pmatrix}
			0 & 0 \\
			c_o e^{2i\theta(\lambda_o;x,t)} e^{-2 \left( \ln (N) g(\lambda_o) + y(\lambda_o)\right)}& 0
		\end{pmatrix}\right]
	\end{equation}
and 
\begin{equation}
\res_{\wo {\lambda}_o} \Lambda(z) = \lim_{z \to \wo {\lambda}_o } \left[ \Lambda(z) \begin{pmatrix}
			0 & - \wo {c_o} e^{-2i\theta(\wo {\lambda}_o;x,t)} e^{2 \left( \ln (N) g(\wo {\lambda}_o) + y(\wo {\lambda}_o)\right)} \\
			0 & 0
		\end{pmatrix}\right].
	\end{equation}
and $\Lambda(z) = I + \mathcal{O}(1/z)$, as $z \to \infty$.

The local approximations as constructed in Section \ref{sec:local} may be used without any changes, and so we arrive at an analogous Riemann-Hilbert problem for the error matrix $\mathcal{E}$, which is again in the small norm setting.  At the end of all of this, the previously determined $\psi_{SG}(x,t;N)$ (the soliton gas condensate solution) whose asymptotics are determined from the Riemann-Hilbert problem for $\Lambda$, is more complicated, and we use $\psi_{SG+T}(x,t;\lambda_{o},N)$ to denote this new solution.  It describes the leading order asymptotic interaction between the tracer soliton determined by $\lambda_{o}$ and $c_{o}$ and the condensate gas.  It is straightforward to follow the outline presented above in this section and establish that 
\begin{eqnarray}
\psi_{N+1}(x,t;\lambda_{o}) = \psi_{SG+T}(x,t;\lambda_{o},N) + \mathcal{O}\left( \frac{1}{N} \right),
    \end{eqnarray}
    with error that is uniformly controlled on any compact domain in $(x,t)$.

Heuristically, the tracer soliton is localized within the region where the residue condition is $\mathcal{O}(1)$ (i.e., neither large nor small). To achieve this, we choose $c_o$ proportional to $\exp{(2 \left( \ln(N) \Re[g(\lambda_o)]\right)}$ so that the residue condition for the tracer soliton is $\mathcal{O}(1)$ at $t=0$ and $x$ near $0$.    
 As $t$ evolves, the tracer soliton's ``effective position'' is approximated by the equation
\[\Re[2i \theta({\lambda}_o;x,t)]=0\]
which then leads to the equation 
\begin{eqnarray}
\label{eq:tracervelocity}
x(t)=-4 \Re[\lambda_o] t.
\end{eqnarray}
The analysis outlined above shows that the velocity of the tracer soliton is effectively constant, and the same as if the soliton were propagating in the absence of the background condensate.  

\begin{remark}
    It would be possible to carry out the more precise analysis in Section 6 of \cite{Girotti2023} and prove that the global maximum of $\psi_{N+1}(x,t;\lambda_{o})$ evolves in an oscillatory fashion, with average velocity $- 4 \Re[\lambda_{0}]$.  Since the details of this analysis are similar to those in \cite{Girotti2023}, we omit them.
    \end{remark}

In 2003 \cite{El2003n}, El derived a system of kinetic equations for the velocity $v(x,t,z)$ of a trial soliton propagating through a dense KdV soliton gas:
\begin{gather}\label{eq:eq87}
v(z) = 4 z^2 + \frac{1}{z} \int_{0}^{\infty} \ln \left| \frac{s+z}{s-z}\right| (v(z) - v(s)) f(s) \di s 
\end{gather}
where the local soliton density $f(x,t,z)$ satisfies the continuity equation
\[f_t + (v f)_x = 0.\]
For the $N$-soliton condensate solutions of the NLS equation considered here, our analysis leads to an analogous integral equation:
\begin{gather}\label{eq:eq88}
v(z) = -4 \Re[z] + \frac{1}{z} \int_{0}^{\infty} \ln \left| \frac{s+z}{s-z}\right| (v(z) - v(s)) f(s) \di s. 
\end{gather}
In this setting, $f$ is linked to the $x$-derivative of the jump of $g'$ across $\eta_{3}$.  While the jump of $g'$ across $\eta_{3}$ is non-trivial, it is actually independent of $x$, and so $f=0$. This not only satisfies the continuity equation trivially but also simplifies equation \eqref{eq:eq88} to $v(z) = -4 \Re[z]$.  We note, however, that location of the global maximum of the solution does not evolve with constant velocity, though its average velocity is constant.

In closing, we note that this is the first result in which the kinetic theory is rigorously confirmed for a finite collection of solitons, rather than in a long-time regime after having taking the $N \to \infty$ limit, as in \cite{Girotti2023}.  Indeed:
\begin{itemize}
    \item The asymptotic analysis holds for $N$ sufficiently large.   
    \item The soliton gas which forms at $t=0$ (although a relatively simple condensate) and the "effective velocity" conspire to satisfy the continuity equation and equation \eqref{eq:eq88} for $(x,t)$ in compact sets. 
\end{itemize}

\paragraph{Acknowledgments}
The authors thank Robert Jenkins for valuable discussions. G.M. was partially supported by the Swedish Research Council under grant no. 2016-06596 while the author was in residence at Institut Mittag-Leffler in Djursholm, Sweden during the fall semester of 2024.

\appendix 

\section{Technical Results}
\label{appendix:A}

\subsection{Proof of Proposition \ref{prop:sym_R}}

\begin{Proposition}
    The function $R(z)$ defined as

    \begin{equation}
    \label{eq:def_R_app}
        R(z)=(z-A)^{\frac{1}{2}}(z+\wo{A})^{\frac{1}{2}}(z+A)^{\frac{1}{2}}(z-\wo{A})^{\frac{1}{2}}\,,
    \end{equation}
    where for each factor $(\cdot )^{1/2}$ we use the principle branch, satisfies the following properties

\begin{enumerate}
\item $R(z) = z^{2} - \frac{A^{2} + \wo{A}^{2}}{2} + \mathcal{O} \left( \frac{1}{z^{2}}\right)$, as $z \to \infty$,
    \item $R(\wo z) = \wo R(z)$
    \item $R(- z) = R(z)$
\end{enumerate}
Furthermore, 
\begin{equation}
    \label{eq:elliptic_integral_app}
    \int_{\eta_1} \frac{1}{R_+(z)}\di z = -\frac{K(\cos(\theta))}{|A|}\,,\quad \int_{-\wo A}^{-A}\frac{1}{R(s)}\di s=-i\frac{K(\sin(\theta))}{|A|}\,,
\end{equation}
where $ \theta=\arg(A)$, and  $K(k)$ is the complete elliptic integral of the first kind \cite[Chap. 19]{dlmf}.

\end{Proposition}

\begin{proof}
    Properties \textit{1.-3.} follows from direct computations. Regarding \eqref{eq:elliptic_integral_app}, we define $m=\tan^2\left(\frac{\theta}{2}\right)$ and we perform the following change of variables
    
    \begin{equation}
        z = \frac{i\sqrt{m}w+1}{i\sqrt{m}w-1}|A|\,,
    \end{equation}
obtaining that
    \begin{align}
        & \int_{A}^{-\wo A} \frac{\di z}{R(z)} = -\frac{i}{2|A|\cos^2\left(\frac{\theta}{2}\right)}\int_{-m}^{-1}\frac{\di w}{i\sqrt{1-w^2}\sqrt{1-m^2w^2}}=-\frac{K(\sqrt{1-m^2})}{2|A|\cos^2\left(\frac{\theta}{2}\right)}\,,\\
        & \int_{-\wo A}^{-A} \frac{\di z}{R(z)} = -\frac{i}{2|A|\cos^2\left(\frac{\theta}{2}\right)}\int_{-1}^{1}\frac{\di w}{i\sqrt{1-w^2}\sqrt{1-m^2w^2}}=-i \frac{K(m)}{|A|\cos^2\left(\frac{\theta}{2}\right)}\,.
    \end{align}

    Noticing that $\frac{1-m}{1+m}=\cos(\theta)$ , we can apply the Landen's transform \cite[Chap 19.8]{dlmf} to rewrite the previous elliptic integral in terms of $\theta$ as

    \begin{equation}
        K(\sqrt{1-m^2}) = 2\cos^2\left(\frac{\theta}{2}\right)K(\cos(\theta))\,, \qquad K(\sin(\theta))=\cos^{-2}\left(\frac{\theta}{2}\right)K(m)\,,
    \end{equation}
    this concludes the proof.
    
\end{proof}

\subsection{Proof of Lemma \ref{lem:immaginary_part}}

In this subsection, we give a proof of Lemma \ref{lem:immaginary_part}. First, we prove the following auxiliary result.

\begin{Proposition}
    The $g$-function given in \eqref{eq:g_def} satisfies the following symmetry relation
\begin{gather}\label{eq:eqsymg}
\Re \left( g(z) \right)= - \Re \left( g(\wo{z}) \right), \quad z \in \mathbb{C}_{+}.
\end{gather}
\end{Proposition}

\begin{proof}
    We can rewrite the $g$-function \eqref{eq:g_def} as
\begin{gather}
g(z) = \frac{1}{2} - \frac{2 c}{\tau} \int_{A}^{z} \frac{\di s}{R(s)}, \quad z \in \mathbb{C}_{+} \\ \\[0.2pt]
g(z) = -\frac{1}{2} - \frac{2 c}{\tau} \int_{\wo{A}}^{z} \frac{\di s}{R(s)}, \quad z \in \mathbb{C}_{-}.
\end{gather}

Let $z \in \mathbb{C}_{+}$, then
\begin{equation}\label{eq:eqc2}
\begin{split}
g(\wo{z}) & = -\frac{1}{2} - \frac{2 c}{\tau} \int_{\wo{A}}^{z} \frac{\di s}{R(s)} = -\frac{1}{2} - \frac{2c}{\tau} \int_{0}^{1} \frac{(\wo{z}-\wo{A}) \di\omega}{R \left( s(\omega) \right)} \\ & = -\frac{1}{2} - \frac{2c}{\tau} \overline{ \int_{0}^{1} \frac{(z-A)}{R(\zeta(\omega))} \di \omega } = \overline{-\frac{1}{2} + \frac{2c}{\tau} \int_{A}^{z} \frac{d \zeta}{R(\zeta)}} = \overline{-g(z)}
\end{split}
\end{equation}
where we have used the change of variables 
\begin{gather}
s(\omega) = \wo{z} \omega + (1-\omega) \wo{A}, \quad \zeta(\omega) = z \wo{\omega} + (1-\wo{\omega}) A, \quad 0 \leq \omega \leq 1
\end{gather}
and the properties $R(\wo{z})=\overline{R(z)}$, and $\overline{s(\omega)}=\zeta(\omega)$.
Equation \eqref{eq:eqc2} shows the desired symmetry. 
\end{proof}

We are now ready to prove the next lemma that proves the first part of Lemma \ref{lem:immaginary_part}.

\begin{lemma}
\label{lem:first_part}
    Let the function $g(z)$ be the unique solution to RHP \ref{RHP:g}. Define the function $u(z)$ for $z\in \C_+\setminus(\eta_1\cup\gamma)$ as  
 \begin{equation}
    \label{eq:udef_app}
        u(z)=-2g(z)+1= \frac{4c}{\tau} \int_A^z\frac{1}{R(s)}\di s , \ \qquad z\in \mathbb{C}_{+} \setminus\left( \eta_{1} \cup \gamma \right),
    \end{equation}
    where the path of integration is chosen to avoid $\eta_{1}\cup \eta_{2}\cup\gamma $.         Let $\eta_{3}$ be the arc of the circle of radius $|A|$ centered at $0$ connecting $A$ and $-\overline{A}$.  We have  
      \begin{equation}
            \label{eq:prop_u}
         u_{-}(z) = g_{+}(z) - g_{-}(z), \quad z \in \eta_{1}\,, \qquad u(z) \in i\mathbb{R}, \quad z\in \eta_3
    \end{equation}
    
As with \eqref{eq:udef_app}, we define
\begin{equation} \label{eq:utdef_app}
    \tilde{u}(z) = \frac{4c}{\tau} \int_{\wo A}^z\frac{1}{R(s)}\di s , \ \qquad z\in \mathbb{C}_{-} \setminus\left( \eta_{2} \cup \gamma \right),
\end{equation}
where again the path is chosen to avoid $\eta_{1} \cup \eta_{2} \cup \gamma$.
Let $\eta_{4}$ be the arc of the circle of radius $|A|$ centered at $0$ connecting $-A$ and $\overline{A}$.  We have  
    \begin{equation}
         \wt u_{-}(z) = g_{+}(z) - g_{-}(z), \quad z \in \eta_{2}\,, \qquad u(z) \in i\mathbb{R}, \quad z\in \eta_4
    \end{equation}
     Clearly  $\eta_4$ lies below $\eta_2$ and we take it to be oriented from $-A$ to $\wo A$ (like $\eta_2$). 
\end{lemma}

\begin{proof}
We prove Lemma \ref{lem:first_part} for $u(z)$, and the proof for $\wt u(z)$ is analogous.

First, we verify the first part of \eqref{eq:prop_u} for $z\in\eta_1$

\begin{equation}
    g_+(z)-g_-(z) = -\frac{2c}{\tau}\int_A^z\frac{1}{R_+(s)}\di s +\frac{2c}{\tau}\int_A^z\frac{1}{R_-(s)}\di s  = \frac{4c}{\tau}\int_A^z\frac{1}{R_-(s)}\di s\,,
\end{equation}
where we used the definition of $g(z)$ \eqref{eq:g_def}.
Setting $\theta=\arg(A)$, we consider the function $h(\phi)$

\begin{equation}
    h(\phi) = \int_{|A|e^{i\theta}}^{|A|e^{i\phi}}\frac{\di s}{R(s)}\,,
\end{equation}
which is the restriction, up to an imaginary constant, of $u(z)$ on the circle of radius $|A|$.
Consider now $h'(\phi)$

\begin{equation}
    h'(\phi) = \frac{i|A|e^{i\phi}}{R(|A|e^{i\phi})}=\frac{\pm i}{\sqrt{2}|A|(\cos(2\phi) - \cos(2\theta))^{1/2}}\,,
\end{equation}
therefore 
\begin{equation}
    h'(\phi)\in\begin{cases}
        i\R & \phi\in(-\theta,\theta)\cup (\pi-\theta,\pi+\theta) \\
        \R & \phi\in (\theta,\pi-\theta)
    \end{cases}\,.
\end{equation}
By integration, one concludes the proof.

\end{proof}

Finally, we prove the following proposition, which implies the final part of Lemma \ref{lem:immaginary_part}.

\begin{Proposition}
\label{prop:propC1}
The $g$-function given in \eqref{eq:g_def} satisfies the following inequalities 
\begin{align}
\label{eq:gboundCplus}
&0 \le  \Re\left(g(z)\right) \leq \frac{1}{2}\quad z \in \C_+\setminus\Omega_3^+\,,\\
\label{eq:gboundCplusOmega}
&\frac{1}{2}\le \Re(g(z)) < 1\quad z\in \Omega_3^+\,.
\end{align}
Furthermore, the functions $u(z),\wt u(z)$ \eqref{eq:udef}-\eqref{eq:utdef} satisfy the following inequalities
\begin{align}
\label{eq:u_bounds_app}
    &\Re(u(z)) < 0 \quad z\in \Omega_3^+ 
    &\Re(u(z)) > 0 \quad z\in \C_+\setminus\wo\Omega_3^+ \\
    &\Re(\wt u(z)) > 0 \quad z\in \Omega_2^+ 
    &\Re(\wt u(z)) < 0 \quad z\in \C_-\setminus\wo\Omega_2^+ \\
\end{align}
and the equalities hold for $z\in\eta_3$ or $z\in\eta_4$ respectively. $\Omega_2^+,\Omega_3^+$ are the open regions enclosed by $\eta_1,\eta_3$ and $\eta_2,\eta_4$ respectively, see Figure \ref{fig:contour_imaginary}.
\end{Proposition}
\begin{proof}

To prove \eqref{eq:gboundCplus}, we note that 
\begin{itemize}
    \item $g(z)$ is an analytic function in $ \mathbb{C}_{+} \setminus ( \eta_1 \cup \gamma )$, where $\gamma$ here denotes the portion of the vertical contour $\gamma$ in $\mathbb{C}_{+}$.
    \item On $\eta_{1}$, $g$ satisfies $g_{+}(z) + g_{-}(z) = 1$, and on $\gamma$, $g$ satisfies $g_{+}(z) - g_{-}(z) = \frac{-2}{\tau}$ (see \eqref{e:eq314a}-\eqref{e:eq314c}).  
    \item From \eqref{eq:eqsymg}, it follows that $\Re\left(g(z)\right) = 0$ for $z \in \R$.
    \item Since $g_\infty \in i\R$, then $\Re\left(g(z)\right) \to 0$ as $z \to \infty$.
\end{itemize}
We introduce the related function $\tilde{g}$, defined as follows:
\begin{align}
\label{eq:gtildedef}
    \tilde{g}(z) = \begin{cases}
        1 - g(z), &  z \in \Omega_{3}^{+} \\
        g(z), &  z \in \mathbb{C}_{+} \setminus \Omega_{3}^{+}.
    \end{cases}
\end{align}
One may verify that in the upper half-plane, $\tilde{g}$ is analytic in $\mathbb{C}_{+} \setminus \left( \eta_{3} \cup \gamma \right) $.  Furthermore, it is a consequence of Lemma \ref{lem:first_part} that $\Re \left(\tilde{g}_{\pm}(z) \right) \equiv \frac{1}{2} $ for all $z\in\eta_{3}$.  

We therefore know that $\Re(\tilde{g}(z) )$ is harmonic for $z \in \mathbb{C}_{+} \setminus \left( \eta_{3} \cup \gamma \right)$. Therefore, by the maximum modulus principle, $\Re(\tilde{g}(z) )$ can only achieve its maximum and minimum on the boundary of its domain of harmonicity.  But this function is identically $0$ on $\mathbb{R}$, vanishes at $\infty$, and is identically $1/2$ on $\eta_{3}$.  

We note that $\Re(\tilde{g}(z) ) \ge 0$ for all $z \in \gamma$.  Indeed, suppose to the contrary that there is a point $\hat{z} \in \gamma$ such that $\Re(\tilde{g}(\hat{z} ) < 0$.  This would imply a  minimum $z^{*} \in \gamma$, with $\tilde{g}(z^{*}) < 0$.  But since $\tilde{g}_{+}- \tilde{g}_{-} = \frac{-2}{\tau} \in i \mathbb{R}$, the quantity 
\begin{align}
\begin{cases}
        \tilde{g}(z), & \mbox{ for } z \mbox{ to the right of } \gamma , \\
        \tilde{g}(z) - \frac{2}{\tau}, & \mbox{ for } z \mbox{ to the left of } \gamma,
\end{cases}
\end{align}
is analytic in a neighborhood of $z^{*}$, has the same real part as $\tilde{g}(z)$, and therefore has a critical point at $z^{*}$, which is impossible since $\tilde{g}'(z) = g'(z)$ for $z$ outside $\Omega_{3}^{+}$ and $g'(z)$ only vanishes at $\infty$.  Thus $\Re(\tilde{g}(z) ) \ge 0$ for all $z \in \gamma$.

Now applying the maximum modulus principle, we have that 
\begin{gather}
0 = \min_{s \in \mathbb{R} \cup \eta_3 \cup \gamma} \Re\left(\tilde{g}(s)\right) <\Re\left(\tilde{g}(z)\right) < \max_{s\in \mathbb{R} \cup \eta_3 \cup \gamma} \Re\left(\tilde{g} (s)\right) = 1/2,\,\quad z\in \mathbb{C}_{+} \setminus (\eta_{3} \cup \gamma ).
\end{gather}
Now for $z \in \mathbb{C}_{+} \setminus \Omega_{3}^{+}$, $g(z) = \tilde{g}(z)$, so that 
\begin{gather}\label{eq:eqrealpart}
    0 < \Re\left( g(z) \right) < \frac{1}{2} , \ \ z \in \mathbb{C}_{+} \setminus \Omega_{3}^{+}.
\end{gather}
However, within the set $\Omega_{3}^{+}$, using the definition \eqref{eq:gtildedef}, we have
\begin{gather}\label{eq:eqrealpart1}
       \frac{1}{2} < \Re\left( g(z) \right) < 1 , \ \ z \in \Omega_{3}^{+}.
\end{gather}
We now show \eqref{eq:u_bounds_app}. We notice that $u(z)=1-2 \, g(z)$. Therefore, using the proof of proposition \ref{prop:propC1}, and specifically relations \eqref{eq:eqrealpart} and \eqref{eq:eqrealpart1}, it follows that
\begin{align*}
0 < \Re(u(z)) < 1, & \quad z \in \mathbb{C}_{+}\setminus\overline{\Omega_3^+}\\
-1 < \Re(u(z)) < 0, & \quad z \in \Omega_3^+
\end{align*}
as desired. Because of Remark \ref{rem:sym_u_ut}, an analogous result holds for  $\wt u(z)$, therefore we have completed the proof.
\end{proof}

\subsection{Proof of Proposition \ref{prop:last_detail}}
\label{appendix:proof2}
Now, we can prove Proposition \ref{prop:last_detail}.
\begin{Proposition}
    Let $A_N(z),\wt A(z)$ be the solutions of RHP \ref{rhp:A} and RHP \ref{rhp:Atil}, respectively, then 
    \begin{equation}
    \label{eq:relation_A_A_tilde}
        A_N(z) = \left(I + \mathcal{O} \left( \frac{1}{N \left( 1 + |z| \right)} \right)\right)\wt A(z)\,,
    \end{equation}
uniformly in $z$.
\end{Proposition}
\begin{proof}
Consider the function $\cE_N(z) = A_N(z) \tilde{A}^{-1}(z)$ which is analytic on $\mathbb{C} \setminus \Gamma_1 \cup \Gamma_2$. Its jump on $\Gamma_1$ equals
\begin{eqnarray}
&&
J_{\cE_{N}}(z) = \tilde{A}_{-}(z) J_{A_N}(z) J_{\tilde{A}}^{-1}(z) \tilde{A}^{-1}(z) 
=  \tilde{A}_{-}
\pmtwo{1}{0}{\sum_{j=1}^{N} \frac{c_{j}e^{ 2 i \theta(\lambda_{j})}
}{z-\lambda_{j}} + N \int_{\eta_{1}} \frac{h(\lambda) \rho(\lambda) e^{ 2 i \theta(\lambda)} \di \lambda}{z- \lambda}}{1}
\left(  \tilde{A}_{-}\right)^{-1}
\nonumber \\
\label{eq:EcalJump}
&& = I + \tilde{A}_{-} \begin{pmatrix}
0 & 0\\
 \frac{q_1(z;N)}{N}  & 0
\end{pmatrix} \left(\tilde{A}_{-} \right)^{-1}
\end{eqnarray}
where we have introduced the quantity $q_{1}(z;N)$:
\begin{eqnarray}
\label{eq:q1def}
q_{1}(z;N) = N \left( {\sum_{j=1}^{N} \frac{c_{j}e^{ 2 i \theta(\lambda_{j})}
}{z-\lambda_{j}} + N \int_{\eta_{1}} \frac{h(\lambda) \rho(\lambda) e^{ 2 i \theta(\lambda)} \di \lambda}{z- \lambda}} \right).
\end{eqnarray}
The jump for $\cE_{N}$ on $\Gamma_2$ is similarly evaluated, and one finds that 
\begin{eqnarray}
&&
J_{\cE_{N}}(z) 
= I + \tilde{A}_{-}(z) \begin{pmatrix}
0 & \frac{q_2(z)}{N} \\
 0  & 0
\end{pmatrix} \left( \tilde{A}_{-} \right) ^{-1}
\end{eqnarray}
where
\begin{eqnarray}
q_{2}(z;N) =N \left(  {\sum_{j=1}^N \frac{\wo c_j}{z-\wo \lambda_j}e^{-2i\theta(\wo{\lambda}_j)} -  N \int_{\eta_{2}} \frac{{h(\lambda)} \rho(\lambda) e^{ - 2 i \theta(\lambda)} \di s}{z - \lambda}} \right).
\end{eqnarray}
Note that for $z \in \Gamma_{2}$, we have via symmetry that
\begin{eqnarray}
\label{eq:q2symm}
    q_{2}(z;N) = \overline{q_{1}(\overline{z};N)}.
\end{eqnarray}

Since both $A_N$ and $\tilde{A}$ solve Riemann-Hilbert problems on the same collection of contours, the quantity $\cE_{N}$ also satisfies a Riemann-Hilbert problem:
\begin{RHP}
\label{rhp:EA}
    Find a $2 \times 2$ matrix valued function $\cE_{N}(z)$ that is analytic in $\mathbb{C} \setminus \left(\Gamma_{1} \cup \Gamma_{2} \right)$, satisfies the normalization condition $\cE_{N}(z) = I + \mathcal{O} \left( \frac{1}{z} \right)$ as $z \to \infty$, has smooth boundary values on the $+$ and $-$ side of the contours $\Gamma_{1}$ and $\Gamma_{2}$, and satisfies the jump relation
    \begin{gather}
      \left(   \cE_{N} \right)_{+}(z) =      \left(   \cE_{N} \right)_{-}(z) J_{\cE_{N}}(z) \ , \ z \in \Gamma_{1} \cup \Gamma_{2} \ .
    \end{gather}
\end{RHP}

    We will prove that this Riemann-Hilbert problem is a so-called "small norm Riemann-Hilbert problem", in the sense that the jump matrix $J_{\cE_{N}}$ is uniformly close to the trivial jump matrix $I_{2}$.  We focus on the quantity $q_{1}(z;N)$, which we will show is analytic in an annular region surrounding the contour $\Gamma_{1}$ (but bounded away from the interval $\eta_{1}$), where it satisfies
\begin{eqnarray}
\label{eq:q1Uniform}
| q_{1}(z;N) |  \le \mbox{const}
\end{eqnarray}
unformly in the annular region.  (The constant depends in particular on the distance between the annular region and the interval $\eta_{1}$).
Because of the relation \eqref{eq:q2symm}, the same estimate holds for $q_{2}(z;N)$ in an annular region surrounding $\Gamma_{2}$.

First, from the definition \eqref{eq:q1def}, the analyticity of $q_{1}(z;N)$ in an annular region surrounding $\Gamma_{1}$ is clear.  In order to establish the uniform estimate \eqref{eq:q1Uniform}, we introduce the map $\Phi: \lambda \mapsto x$ defined via the relation
\begin{gather}\label{eq:map1}
 \Phi(\lambda) = \int_{\lambda_0}^{\lambda} \rho(s)\di s.
\end{gather}
Recall that the discrete eigenvalues $\lambda_j$ are defined as
\begin{gather}
\int_{\lambda_0}^{\lambda_j} \rho(s)\di s = \frac{2j-1}{2N}, \quad j=1,2, \ldots, N
\end{gather}
which we can rewrite as $x_j := \Phi(\lambda_j) = \frac{2j-1}{2N}$ from the definition of $\Phi$. 
We observe that in the $\lambda$-plane, the discrete eigenvalues $\lambda_j$ are not equidistant, and therefore they do not correspond to the midpoints of intervals that start at $\lambda_0=-a+ib$ and end at $\lambda_{N+1}=a+ib$. However, in the variable $x$, the points $x_j$ defined above become the midpoints of intervals $[\xi_{j}, \xi_{j+1}]$, with equidistant endpoints $\xi_j$ defined as
\[\xi_j = \frac{j-1}{N}, \quad j=1,\ldots,N+1,\]
where we note that  $\xi_1  = 0$ and $\xi_{N+1} = 1$.
One also has that $\xi_{j+1}$ is the midpoint of the interval $[x_{j}, x_{j+1}]$.

We introduce the following notation:
\begin{eqnarray}
\label{eq:HDEF}
H(s) = \frac{h(s) e^{2i \theta(s)}}{z-s}, \quad F(x) = H \left( \Phi^{-1}(x) \right)
\end{eqnarray}
where $h$ is the normalization constant sampling function. We have:
\begin{align*}
\sum_{j=1}^{N} H(\lambda_j) - N\int_{-\wo{A}}^{A} H(s)\di s & = \sum_{j=1}^{N} F(x_j) - N \int_{0}^{1} F(x) dx = \sum_{j=1}^{N} \left( F(x_j) - N \int_{\xi_{j}}^{\xi_{j+1}} F(x) \di x \right).
\end{align*}
Now we expand $F(s)$ using a Taylor series expansion around $x_j$. 
The first nonzero contribution arises from $F''(x_j)$, and so we find
\begin{align*}
\sum_{j=1}^{N} H(\lambda_j) - N\int_{-\wo{A}}^{A} H(s)\di s & = - \frac{1}{24 N^2} \sum_{j=1}^{N} F''(x_j) + \mathcal{O}(N^{-3})
\end{align*}
where the error term is $\mathcal{O}(N^{-3})$ because the function $F$ is analytic in a neighborhood of the interval $[0,1]$.  The sum on the right hand side of the last relation is again a Riemann sum on the same grid, and so it can be uniformly approximated:
\begin{eqnarray}
    \sum_{j=1}^{N} F''(x_j) = N\left( F'(1) -F'(0) \right)+ \mathcal{O} \left(\frac{1}{N} \right).
\end{eqnarray}
Now using (\ref{eq:q1def}), (\ref{eq:HDEF}), and some calculus, we find that 
\begin{align}
\label{eq:q1asymp}
q_1(z;N) & = \frac{1}{24} \left( \frac{(h'(-\wo{A}) + 2i h(-\wo{A}) \theta'(-\wo{A})) e^{2i \theta(-\wo{A})}}{z+\wo{A}} + \frac{h(-\wo{A}) e^{2i \theta(-\wo{A})}}{(z+\wo{A})^2} \right) \frac{1}{\rho(-\wo{A})} \\
&- \frac{1}{24} \left( \frac{(h'(A) + 2i h(A) \theta'(A)) e^{2i \theta(A)}}{z-A} + \frac{h(A) e^{2i \theta(A)}}{(z-A)^2} \right) \frac{1}{\rho(A)} + \mathcal{O} \left(\frac{1}{N}\right).
\end{align}
Therefore we have established that $q_{1}(z;N)$ is analytic in the desired annular region, and satisfies \eqref{eq:q1Uniform}.  Although such an explicit formula is not entirely necessary, \eqref{eq:q1asymp} provides the leading order asymptotic behavior of $q_{1}(z;N)$ for $N$ large.  As mentioned above, this yields the uniform boundedness of $q_{2}(z;N)$ in an annular region surrounding $\Gamma_{2}$.   

Next, we return to $J_{\cE_{N}}(z)$ on $\Gamma_1$ (see formula \eqref{eq:EcalJump}). At this point, it is convenient to assume that the contour $\Gamma_{1}$ encircles not only the interval $\eta_{1}$, but also the collection of contours $\eta_{3}, \omega_{3}$, and $\omega_{4}$ as well (see Figure \ref{rhp:small_norm}).  This assumption simplifies the representation of $\tilde{A}_{-}$ which we use to estimate the jump matrix $J_{\cE_{N}}(z)$.  Indeed, then for $z \in \Gamma_{1}$ we have the following representation for $\tilde{A}$:
\begin{gather}
\tilde{A}(z) = e^{- \left( \ln(N) g_\infty + y(z) \right) \sigma_3} \cE(z) \Lambda(z) e^{\left( \ln(N) g(z) + y(z) \right) \sigma_3} \begin{pmatrix}
1 & 0\\
-N \int_{\eta_1} \frac{h(\lambda) \rho(\lambda)}{\lambda - z} e^{2i \theta(\lambda)} \di \lambda & 1
\end{pmatrix}
\end{gather}
where $\cE(z)$ is the error function and $\Lambda(z)$ is the outer approximate solution. By direct calculation, the product $  \tilde{A}_{-}(z) \begin{pmatrix}
0 & 0 \\
\frac{q_1(z)}{N} & 0
\end{pmatrix}\tilde{A}_{-}^{-1}(z)$ in \eqref{eq:EcalJump} becomes
\begin{gather}\label{eq:product}
 \, e^{- \left( \ln(N) g_\infty + y(z) \right) \sigma_3} \cE(z) \Lambda(z) e^{y(z)\sigma_{3}}
 \pmtwo{0}{0}
 {\frac{q_1(z)e^{-2 \ln(N) g(z)  \sigma_3}}{N}}
 {0}
 e^{-y(z)\sigma_{3}}
  \left( \cE(z) \Lambda(z) \right)^{-1} e^{\left( \ln(N) g_\infty + y(z) \right) \sigma_3}.
\end{gather}

We notice that:
\begin{enumerate}  \label{eq:conditions}
    \item $g_\infty \in i\R$ and therefore $e^{\pm \ln (N) g_\infty}$ is bounded in $N$. \label{itemone}
    \item the scalar $q_{1}(z;N)$, as well as the matrices $\cE$, $\Lambda$, and  $e^{y(z) \sigma_3}$ are all uniformly bounded in $N$, for all $z\in \Gamma_{1}$ \label{itemtwo}
\end{enumerate}
and we have
\begin{gather*}
\frac{1}{N} e^{-2 \ln (N) g(z)} = e^{-\ln (N) \left( 1 + 2 g(z) \right)} = e^{-\ln (N) \Big( 1 + 2 \, \Re \left( g(z) \right) + i \left( 1 + 2 \, \mathfrak{Im} \left(g(z) \right) \right) \Big)}.
\end{gather*}
 Therefore we need an estimate on
 $\frac{1}{N} e^{-2 \ln (N) g(z)}$, and we have
\begin{gather}
\label{eq:eqestimate}
\left| \frac{1}{N} e^{-2 \ln (N) g(z)} \right| = e^{-\ln (N) \Big( 1 + 2 \, \Re \left( g(z) \right) \Big)} \leq \frac{C}{N}, \ \quad z \in \Gamma_{1},
\end{gather}
which follows because of \eqref{eq:gboundCplus}.  Finally then, we have established that
\begin{gather}
\label{eq:JEAbound1}
    J_{\cE_{N}}(z) = I_{2}  + \mathcal{O}\left( \frac{1}{N} \right), 
\end{gather}
where the error term is uniform for all $z $ in an annular region surrounding $\Gamma_{1}$.  The analysis of $J_{\cE_{N}}(z)$ on $\Gamma_2 \subset \mathbb{C}_{-}$ is entirely similar, and we leave the reader to verify that 
\begin{gather}
\label{eq:JEAbound2}
    J_{\cE_{N}}(z) = I_{2}  + \mathcal{O}\left( \frac{1}{N} \right), 
\end{gather}
uniformly for all $ z$ in an annular region surrounding $\Gamma_{2}$ as well.  

Returning to the Riemann-Hilbert problem \ref{rhp:EA}, we have established that the jump matrix satisfies \eqref{eq:JEAbound1} and is analytic in an annular region around $\Gamma_{1}$, and similarly it satisfies \eqref{eq:JEAbound2} and is analytic in an annular region surrounding $\Gamma_{2}$.  Then, following the statement and proof of Theorem 7.10 on P. 1532 of \cite{DKMVZ1}, we may conclude that $\cE_{N}(z)$ not only exists, it also satisfies  \eqref{eq:relation_A_A_tilde}.

\end{proof}

\bibliographystyle{siam}
\bibliography{mybib}

\end{document}